\documentclass[a4paper,11pt,reqno]{amsart}

\usepackage[utf8]{inputenc}
\usepackage[english]{babel}
\usepackage{amsmath}
\usepackage{amssymb}
\usepackage{amsfonts}
\usepackage{amsthm}
\usepackage{mathtools}
\usepackage{bbm}
\usepackage{bm}
\usepackage{enumitem}
\usepackage{xcolor}
\usepackage[colorlinks=false]{hyperref}

\newcommand{\eps}{\varepsilon}
\newcommand{\epsj}{{\varepsilon_j}}
\newcommand{\mc}{\mathcal}
\newcommand{\xeps}{{u^\eps}}
\newcommand{\xepsd}{{\dot{u}^\eps}}
\newcommand{\xepsdd}{{\ddot{u}^\eps}}
\newcommand{\xepsj}{{u^\epsj}}
\newcommand{\xepsjd}{{\dot{u}^\epsj}}
\newcommand{\xepsjdd}{{\ddot{u}^\epsj}}

\newcommand{\M}{\mathbb{M}}
\newcommand{\V}{\mathbb{V}}
\newcommand{\R}{\mathbb{R}}
\newcommand{\N}{\mathbb{N}}
\newcommand{\Q}{\mathbb{Q}}

\renewcommand{\to}{\rightarrow}
\renewcommand{\d}{\,\mathrm{d}}
\DeclareMathOperator*{\essup}{ess\,sup}
\DeclareMathOperator{\dist}{dist}

\numberwithin{equation}{section}
\newtheorem{thm}{Theorem}[section]
\newtheorem{defi}[thm]{Definition}
\newtheorem{prop}[thm]{Proposition}
\newtheorem{lemma}[thm]{Lemma}
\newtheorem{cor}[thm]{Corollary}

\theoremstyle{definition}
\newtheorem{rmk}[thm]{Remark}

\begin{document}
	
	\author[F. Riva, G. Scilla and F. Solombrino]{Filippo Riva, Giovanni Scilla and Francesco Solombrino}
	
	\title [IBV solutions to rate-independent systems]{Inertial Balanced Viscosity (IBV) solutions to infinite-dimensional rate-independent systems}
	
	\begin{abstract}
		A suitable notion of weak solution to infinite-dimensional rate-independent systems, called Inertial Balanced Viscosity (IBV) solution, is introduced. The key feature of such notion is that the energy dissipated at jump discontinuities takes both into account inertial and viscous effects. Under a general set of assumptions it is shown that IBV solutions arise as vanishing inertia and viscosity limits of second order dynamic evolutions as well as of the corresponding time-incremental approximations. Relevant examples coming from applications, such as Allen-Cahn type evolutions and Kelvin-Voigt models in linearized elasticity, are considered.
			\end{abstract}
	
	\maketitle
	
	{\small
		\keywords{\noindent {\bf Keywords:} Inertial Balanced Viscosity solutions, rate-independent systems, vanishing inertia and viscosity limit, minimizing movements scheme, variational methods
		}
		\par
		\subjclass{\noindent {\bf 2020 MSC:} 35A15, 49J40, 49S05, 70G75
			
		}
	}

	\pagenumbering{arabic}
	
	\medskip
	
	\tableofcontents
	
	\section{Introduction}

This paper is devoted to the asymptotic analysis, as the parameter $\eps$ tends to $0$, of the solutions to the abstract second order Cauchy problem
	\begin{equation}\label{mainprob}
		\begin{cases}
			\eps^2 \mathbb{M}\xepsdd(t)+\eps \mathbb{V}\xepsd(t)+\partial^Z\mc R(\xepsd(t))+\partial^U\mc E(t,\xeps(t))\ni 0, \quad\text{in $U^*$,}\quad \text{for a.e. }t\in [0,T],\\
			\xeps(0)=u^\eps_0\in D,\quad\xepsd(0)=u^\eps_1\in V.
		\end{cases}
	\end{equation}
Above, $\mc E \colon [0,T]\times U\to [0,+\infty]$ is a time-dependent energy functional with proper domain $D\subseteq U$, $\mc R\colon Z\to [0,+\infty)$ is a positively one-homogeneous (hence rate-independent) dissipation potential, $\V\colon V\to V^*$ is a linear operator, usually modelling viscosity, and $\M\colon U^*\to U^*$ is an inertial tensor; the symbol $\partial^X$ instead denotes the (Frech\'et) subdifferential with respect to the topology of $X\in\{U,Z\}$. It is common to assume that $U,V,Z$ are Banach spaces satisfying the embedding
\begin{equation*}
	U \stackrel{}{\hookrightarrow} V \stackrel{}{\hookrightarrow}Z.
\end{equation*}	
The Reader may have in mind the prototypical situation $U=H^1_0(\Omega)$, $V=L^2(\Omega)$ and $Z=L^1(\Omega)$, and the concrete nonlinear differential inclusion
\begin{equation*}
	\eps^2 \xepsdd(t)+\eps\xepsd(t) +{\rm Sign}(\xepsd(t))-\Delta \xeps(t)+\mathcal{W}'(\xeps(t))\ni f(t), \quad \mbox{ in $(0,T)\times\Omega$}.
\end{equation*}
Here, ${\rm Sign}$ denotes the multivalued signum operator, $f$ is a forcing term, while $\mc W$ is a possibly nonconvex nonlinear function, typically a double-well potential.

The interest in the limit behaviour of \eqref{mainprob} has a twofold motivation. On one side \eqref{mainprob} provides an abstract framework for a broad class of mechanical models driven by Newton's law of dynamics, see for instance \cite[Section~1.2]{BachoViscoplast}, \cite[Chapter~5]{MielkRoubbook}, \cite{RosThom} and references therein. The limit as $\eps\to 0$ corresponds to the so-called \emph{slow loading} regime. On the other side, this asymptotic analysis allows one to recover a (mechanically motivated) notion of weak solution to the rate-independent doubly nonlinear differential inclusion
	\begin{equation}\label{riprob}
	\begin{cases}
		\partial^Z\mc R(\dot{u}(t))+\partial^U\mc E(t,u(t))\ni 0, \quad\text{in $U^*$,}\quad \text{for a.e. }t\in [0,T],\\
		u(0)=u_0\in D,
	\end{cases}
\end{equation}
which has been the object of an intensive study in the last decades.

In this latter context, several notions of weak solutions have been developed to cope with possible lack of differentiability or even continuity of the evolutions induced by a nonconvex driving energy $\mc E$. A first concept was given by the seminal notion of \emph{energetic solution} \cite{MielkTheil1, MielkTheil}, which recasted \eqref{riprob} as a coupling of a global stability condition and an energy-dissipation balance. This notion however  excludes evolutions of local minimizers, which are instead physically reasonable, and so it turns out to be not suitable to capture the correct energetic behaviour at jumps. 
To overcome this issue, \emph{Balanced Viscosity (BV) solutions} were introduced later on in \cite{MielkRosSav09, MielkRosSav12, MielkRosSav16}, coupling a local stability condition with an energy-dissipation balance with a viscous transition cost accounting for the energy dissipated at jumps. Remarkably, these solutions arise as a vanishing viscosity limit of a first order singular perturbation of \eqref{riprob}, namely system \eqref{mainprob} with $\M=0$, i.e. when inertia is neglected. We also mention \cite{MielkRosmulti, Rind} for recent developments of this analysis.

However, based on mechanics principles, it would be desirable to have a notion of solution to \eqref{riprob} where the role of the vanishing inertia is also taken into account. In this sense, a natural attempt to recover a proper notion is considering the dynamic problem \eqref{mainprob} in the limit $\eps\to 0$. When the parameter $\eps$ is fixed, the well-posedeness of system \eqref{mainprob} has attracted some attention in recent times (see \cite{BachoViscoplast}, or \cite{RosThom} for a related investigation in the context of multi-rate systems). Its asymptotic analysis instead has been so far limited to some case studies with convex potential energy $\mc E$ \cite{DMSap, DMSca,LazNar,LazRosThomToad, Rivquas, Sca}, while a general perspective has been pursued only for a related abstract model with quadratic potential energy (but with additional multi-rate structure), see \cite{MielkPetr} and \cite[Section~5.1.2.2]{MielkRoubbook}; the nonconvex case has been finally investigated only in a finite-dimensional setting \cite{GidRiv,RivScilSol}.

The purpose of the present paper is to extend the above mentioned results to an infinite-dimensional abstract setting complying with relevant applications. In doing so we have to face some major difficulties.
\begin{itemize}
	\item A crucial point, both for achieving compactness of dynamic solutions and for the characterization of the dissipation cost emerging at jumps, is an a priori bound on the total variation of the limit function. Now, in infinite dimension, this one (which stems out of the rate-independent dissipation cost $\mc R$) is provided in a weaker topology than the energy domain $U$, namely in the topology of the space $Z$ which is usually also non reflexive (a typical choice is $Z=L^1(\Omega)$). Therefore both tasks require a fine understanding of the different topological settings coming into play.
	\item A further problem is that one usually is forced to take into account driving energy functionals $\mc E$ which are only weakly lower semicontinuous and possibly nonsmooth. Hence the convergence of all the terms appearing in the local stability condition as well as in the energy-dissipation balance may require additional care in comparison to the simplified finite-dimensional setting where all the involved topologies are equivalent.
\end{itemize}

Handling with these (and other) issues results in a general set of assumptions ensuring the convergence of the solutions to \eqref{mainprob} to a weak solution to \eqref{riprob} which we are able to completely characterize from an energetic point of view. The limit evolution we recover is an \emph{Inertial Balanced Viscosity (IBV) solution} in the sense introduced in \cite{RivScilSol}, namely it complies with the local stability condition
	\begin{equation*}
	\partial^Z\mc R(0)+\partial^U\mc E(t,u(t))\ni 0,\qquad\text{ at continuity points of }u,
\end{equation*}
and the viscoinertial energy-dissipation balance 
	\begin{equation*}
	\mc E(t,u^+(t))+V_\mc R(u_{\rm co};s,t)+\sum_{r\in J_u\cap [s,t]}c^{\M,\V}(r;u^-(r),u^+(r))=\mc E(s,u^-(s))+\int_{s}^{t}\partial_t\mc E(r,u(r))\d r,
\end{equation*}
for all times $0\le s\le t\le T$. Above, $V_\mc R(u_{\rm co};s,t)$ is the $\mc R$-variation of the continuous part of $u$ in the time-interval $[s,t]$, while $c^{\M,\V}$ is a suitable viscoinertial cost concentrated at discontinuity times $J_u$ of the evolution $u$. This dissipation, which is typical of the nonconvex setting, is reminiscent of both inertial and viscous effects and is rigorously defined in Section~\ref{subsec:cost}. A crucial difference from the vanishing-viscosity setting, which thus distinguishes between IBV and BV solutions, is that the presence of the inertial term gives rise to a \emph{rate-dependent} transition cost defined on an infinite time-horizon.

The assumptions on the potential energy $\mc E$ we consider are listed in Section~\ref{subsec:poten} and they can be divided into two groups. 
\begin{itemize}
	\item Assumptions \ref{hyp:E1}-\ref{hyp:E5} are crucial to establish compactness of the dynamic evolutions. Observe that \ref{hyp:E5} encompasses a wide class of nonconvex energies although they are required to fulfil a mild notion of nonconvexity, namely $\lambda$-convexity, which is a customary feature in this kind of problems. On the one hand this clearly introduces a restriction on the focus of our analysis; on the other hand this is a fair assumption met by a large class of significant problems. Relevant examples are discussed in Section~\ref{sec:Applications}. Among them we find evolutions driven by Allen-Cahn type energies or Kelvin-Voigt models in linearized elasticity.
	
	In the special case of a convex driving energy, conditions \ref{hyp:E1}-\ref{hyp:E4} alone are enough to completely characterize the limit evolution as an energetic solution to \eqref{riprob} and indeed as a classic one if the energy is even uniformly convex. Thus, in a convex framework no inertial nor viscous effects survive in the slow-loading limit procedure.
	\item The second group of assumptions \ref{hyp:E6}-\ref{hyp:E10} (which are still satisfied by the aforementioned applications) is needed in the nonconvex setting for the task of characterizing jump discontinuities as optimal transitions for the viscoinertial dissipation cost. This behaviour can be heuristically seen blowing up the time-scale around discontinuity points. 
	
	In turn, this calls for finer a priori estimates on the solutions which we recover as a consequence of additional \lq\lq continuity-type\rq\rq properties of the energy and its subdifferential. Again, this results in a fair and general set of assumptions. Some care is only to be devoted to the stronger compactness requirement \ref{hyp:E10} that, as we highlight in the concrete applications of Section~\ref{sec:Applications}, is usually the outcome of elliptic regularity.
\end{itemize}

Although most of the focus is on nonconvex energy functionals, since in this framework IBV solutions come actually into play (see Theorem~\ref{mainthm:nonconvex}), we stress that our results are also new in the convex setting (Theorems~\ref{mainthm:convex}, \ref{mainthm:unifconvex}), where they provide a mechanical validation of energetic solutions.

Finally, in the last part of the paper we investigate the compliance of the notion of IBV solutions with discrete approximation schemes, which can be also useful for computational purposes. Via a Minimizing Movements approach, we namely consider the discrete-in-time approximation (see \eqref{schemeincond}) of the second order problem \eqref{mainprob} and we perform a simultaneous limit in the time-step $\tau$ and in the inertial-viscous parameter $\eps$ in the regime $\tau/\eps\to 0$. Under this assumption we retrieve analogous results as for the continuous-in-time approximation. Indeed in Theorem~\ref{mainthm:nonconvexdiscr} we show that the standard interpolants converge in general to an IBV solution to the rate-independent system \eqref{riprob}, which actually turns out to be an energetic or even a classic one if the potential energy is convex (or uniformly convex), as stated in Theorems~\ref{mainthm:convexdiscr} and \ref{mainthm:unifconvexdiscr}.
	
	\subsection{Plan of the paper} 
	The paper is structured as follows. After recalling some preliminary material in Section~\ref{sec:Notation}, we fix the functional framework and the set of assumptions we consider in Section~\ref{sec:setting}. Section~\ref{sec:IBV} is devoted to the rigorous definition of IBV solutions and to the statement of our main results. Several examples and applications included in our abstract analysis are depicted in Section~\ref{sec:Applications}. The technical content of the paper is confined in Sections~\ref{sec:slowloading} and \ref{sec:multiscale}, where we prove convergence to an IBV solution of the dynamic evolutions \eqref{mainprob} and of the corresponding Minimizing Movements approximation \eqref{schemeincond}, respectively.
	
	At the end of the paper we attach two appendices containing auxiliary material. In Appendix~\ref{App:Existence} we consistly discuss the existence theory for the dynamic problem \eqref{mainprob} in the case $V=U$, as our setting slightly differs from the existing literature \cite{BachoViscoplast}. Appendix~\ref{App:Chainrule} instead collects useful chain-rule formulas in a nonsmooth setting.

	\section{Notation and preliminaries}\label{sec:Notation}
		
	For any normed space $(X,\|\cdot\|_X)$ we denote by $(X^*,\|\cdot\|_{X^*})$ its topological dual and by $\langle x^*, x\rangle_X$ the duality product between $x^*\in X^*$ and $x\in X$. If $H$ is a Hilbert space, we still use the symbol $\langle h_1,h_2\rangle_H$ for the scalar product between $h_1$ and $h_2$, but we write $|h|_H$ for the norm of $h\in H$. Strong and weak convergence in $X$ are denoted with arrows $\xrightarrow{X}$ and $\xrightharpoonup[]{X}$, respectively. For $R>0$, we denote by $B^X_R(x)$ the open ball in $X$ of radius $R$ centered at $x\in X$, and by $\overline{B}^X_R(x)$ its strong closure. If a ball is centered at $x=0$, we use the shortcuts $B^X_R$ and $\overline{B}^X_R$.
	
	Given two normed spaces $X, Y$, we write $X\hookrightarrow Y$ whenever $X$ is continuously embedded in $Y$. If the embedding is dense we use the symbol $\stackrel{d}{\hookrightarrow}$, while by $\hookrightarrow\hookrightarrow$ we mean a compact embedding.
	
	We use standard notations for Lebesgue, Sobolev and Bochner spaces. By $B([a,b];X)$ we mean the set of everywhere defined measurable functions $f\colon [a,b]\to X$ which are bounded in $X$, while we denote by $C_{\rm w}([a,b];X)$ the subset of $B([a,b];X)$ composed by functions which are continuous with respect to the weak topology of $X$. Furthermore, we denote by $AC([a,b];X)$ the set of absolutely continuous functions $f$ from $[a,b]$ to $X$, namely satisfying the inequality
	\begin{equation*}
		\|f(t)-f(s)\|_X\le \left|\int_s^t\varphi(r)\d r\right|,\quad\text{for all }s,t\in [a,b],
	\end{equation*}
for some nonnegative function $\varphi\in L^1(a,b)$. We refer to \cite[Appendix~2]{Brez} for more details on these functional spaces.
	
	Given a subset $A\subset X$, we denote with $\chi_A\colon X\to [0,+\infty]$ its characteristic function , defined as
		\begin{equation*}
				\chi_A(x):=\begin{cases}
						0,&\text{if $x\in A$},\\
						+\infty, &\text{if $x\notin A$}.
			\end{cases}
		\end{equation*}
	The indicator function $\bm 1_A$ of the subset $A$ is instead defined as 
	\begin{equation*}
		\bm 1_A(x):=\begin{cases}
			1,&\text{if $x\in A$},\\
			0, &\text{if $x\notin A$}.
		\end{cases}
	\end{equation*}
 
\subsection{Fréchet and convex subdifferentials}
	We briefly recall some basic definitions in convex analysis (see for instance \cite{Rocka}). Here, $X$ is a generic normed space. Given a function $\Phi\colon X\to(-\infty,+\infty]$, its \emph{Fréchet subdifferential} $\partial^X \Phi\colon X\rightrightarrows X^*$ at a point $x\in X$ is defined as
	\begin{equation*}
		\partial^X \Phi(x)=\left\{\zeta\in X^*:\,\liminf\limits_{y\stackrel{X}{\rightarrow} x} \frac{\Phi(y)-\Phi(x)-\langle\zeta,y-x\rangle_X}{\|y-x\|_X}\ge 0\right\},
	\end{equation*}
with the convention that $\partial^X \Phi(x)=\emptyset$ whenever $\Phi(x)=+\infty$, namely if $x\notin{\rm dom}\Phi:=\{\Phi<+\infty\}$. Note that $\partial^X \Phi(x)$ is a closed convex subset of $X^*$.

If $\Phi$ is convex, then the Fréchet subdifferential coincides with the \emph{convex subdifferential}, namely
	\begin{equation*}
	\partial^X \Phi(x)=\left\{\zeta\in X^*:\,\Phi(y)\ge \Phi(x)+\langle \zeta,y-x\rangle_X,  \quad\text{for every $y\in X$}\right\}.
\end{equation*}
An analogous result, which we need for our scopes, holds true under the following notion of $\lambda$-convexity:
\begin{lemma}\label{lem:lemma3}
	Let $X$ be a Banach space and let $H$ be a Hilbert space such that $X\hookrightarrow H$. Let $\Phi:X\to(-\infty,+\infty]$ and assume that for some $\lambda\geq0$ the map $X\ni x\mapsto \Phi(x)+\frac{\lambda}{2}|x|_H^2$ is convex. Then it holds 
	\begin{equation*}
		\partial^X \Phi(x) = \partial^X \Phi_x^\lambda(x)\,,\quad \text{for every } x\in X,
	\end{equation*}
	where we set $\Phi_x^\lambda(y):=\Phi(y)+\frac{\lambda}{2}|y-x|_H^2$. 
	
	In particular we have
	\begin{equation}\label{lambdaconvexineq}
		\zeta \in \partial^X \Phi(x) \iff  \Phi(y) \geq \Phi(x) + \langle \zeta, y-x \rangle_X - \frac{\lambda}{2} |y-x|_H^2\,,\,\, \text{ for every } y\in X\,.
	\end{equation}
\end{lemma}
\begin{proof}
	We may adapt the argument of \cite[Remark~2.4.4]{AGS2008} to our case once we note that
	\begin{equation*}
		\lim_{y\stackrel{X}{\rightarrow} x} \frac{|y-x|_H^2}{\|y-x\|_X}=0
	\end{equation*}
	being $X\hookrightarrow H$.
\end{proof}

We also recall a sum rule for convex subdifferentials, whose proof can be found in \cite[Theorem~7.1]{MNRT}.

\begin{lemma}\label{lem:lemma4}
	Let $X$ be a normed space, and let $\Phi,\Psi:X\to(-\infty,+\infty]$ be convex functions. If any of the following holds:
	\begin{itemize}
		\item[i)] $\Phi$ is finite and continuous at some $\bar{x}\in {\rm dom}\Phi\cap {\rm dom}\Psi$;
		\item[ii)] $X$ is a Banach space, $\Phi$ and $\Psi$ are lower semicontinuous and $\R^+({\rm dom}\Phi-{\rm dom}\Psi)$ is a subspace of $X$; 
	\end{itemize}
	then
	\begin{equation*}
		\partial^X(\Phi+\Psi)(x) = \partial^X \Phi(x) + \partial^X \Psi(x)\,, \quad \text{for all } x\in {\rm dom}\Phi\cap {\rm dom}\Psi\,.
	\end{equation*}
\end{lemma}

	We finally introduce the \emph{Fenchel conjugate} of the (proper) function $\Phi\colon X\to (-\infty,+\infty]$, namely the convex, lower semicontinuous function
	\begin{equation*}
		\Phi^*\colon X^*\to (-\infty,+\infty],\quad \text{defined as }\quad \Phi^*(\zeta):=\sup\limits_{y\in X} \{\langle\zeta,y\rangle_X-\Phi(y)\}.
	\end{equation*}
	If $\Phi$ is convex, the very definition yields that for every $\zeta\in X^*$ and $x\in X$ the Fenchel conjugate $\Phi^*$ satisfies
	\begin{equation}\label{Fenchelprop}
		\Phi^*(\zeta)+\Phi(x)\ge \langle{\zeta},{x}\rangle_X,\quad\text{ with equality if and only if }\quad \zeta\in \partial^X \Phi(x).
	\end{equation}

	\subsection{Functions of bounded $\mathcal{R}$-variation}
	
	We recall here a suitable generalization of functions of bounded variation useful to deal with potentials $\mc R\colon X\to [0,+\infty)$ satisfying condition \ref{hyp:R1} in Section~\ref{sec:setting}, condition which will be assumed throughout the whole paper.
	
	We say that a function $f:[a,b]\to X$ is a \emph{function of bounded $\mathcal{R}$-variation} in $[a,b]$, and we write $f\in BV_{\mathcal{R}}([a,b];X)$, if its $\mathcal{R}$-variation
		\begin{equation*}
		V_{\mathcal{R}}(f;a,b):=\sup\left\{\sum_{k=1}^n\mathcal{R}(f(t_k)-f(t_{k-1})):\, a=t_0<t_1<\dots<t_{n-1}<t_n=b\right\}\,,
	\end{equation*}
	is finite. In the classical case $\mc R(\cdot)=\|\cdot\|_X$ we write $V_X(f;a,b)$ for the usual variation.
	
	Notice that, by virtue of \eqref{Rbounds}, we have $f\in BV_{\mathcal{R}}([a,b];X)$ if and only if $f\in BV([a,b];X)$ in the classical sense. In particular, if $X$ is a Banach space, any $f\in BV_{\mathcal{R}}([a,b];X)$ is regulated, that is, it admits left and right \emph{strong} limits at every $t\in[a,b]$:
	\begin{equation*}
		f^+(t):=\lim_{s\to t^+}f(s), \quad \mbox{ and }\quad f^-(t):=\lim_{s\to t^-}f(s)\,,\quad\text{in }X,
	\end{equation*}
	with the convention $f^-(a):=f(a)$ and $f^+(b):=f(b)$. Moreover, its pointwise jump set $J_f\subseteq [a,b]$ is at most countable. For future purposes, we also introduce the \emph{essential} jump set
	\begin{equation}\label{essjump}
		J_f^{\rm e}:=\{t\in J_f: f^+(t)\neq f^-(t)\}.
	\end{equation}
	
	{It is also well known (see for instance \cite{Mor}) that any $f\in BV_{\mc R}([a,b];X)$ can be uniquely decomposed as
		\begin{equation}
			f=f_{\rm co}+ f_{J},
			\label{eq:decompositionf}
		\end{equation}
		with $f_{\rm co}$ being a continuous function from $[a,b]$ to $X$, and $f_J$ a purely jump function. Moreover, if we set $v(t):=V_{\mc R}(f;a,t)$, there also holds
		\begin{equation}\label{vcovj}
			V_{\mc R}(f_{\rm co};a,t)=v_{\rm co}(t),\quad\text{and }\quad V_{\mc R}(f_{J};a,t)=v_{J}(t),\qquad\text{ for all }t\in [a,b].
		\end{equation}

	\section{Abstract setting}\label{sec:setting}
	We list below the main assumptions we will use throughout the paper. Several situations covered by this set of assumptions are depicted in Section~\ref{sec:Applications}.
	
\subsection{Functional framework} Let $U$ be a reflexive and separable Banach space. Let $V,W$ be Hilbert spaces and let $Z$ be a Banach space such that
\begin{equation}
U \stackrel{d}{\hookrightarrow} V \stackrel{d}{\hookrightarrow} W \stackrel{d}{\hookrightarrow} Z \quad \mbox{and} \quad  U \hookrightarrow\hookrightarrow W\,.
\label{eq:embeddings}
\end{equation}	
In particular, also $V,W$ and $Z$ are separable. Moreover, up to identify $W$ with $W^*$, we have
\begin{align*}
	U \stackrel{d}{\hookrightarrow} V \stackrel{d}{\hookrightarrow}& W \stackrel{d}{\hookrightarrow} Z\\
	&\parallel\\
	Z^* \stackrel{d}{\hookrightarrow} &W^* \stackrel{d}{\hookrightarrow} V^* \stackrel{d}{\hookrightarrow} U^*\,.
\end{align*}
The above chain of inclusions yields that all the involved duality products coincide whenever they are meaningful, for instance
\begin{equation*}
	\begin{split}
		&\langle \eta, u \rangle_U = \langle \eta, u \rangle_Z,\qquad\text{for all }\eta\in Z^*\,,\, u\in U,\\
		&\langle w, v \rangle_V = \langle w, v \rangle_W,\qquad \text{for all }w\in W=W^*\,,\, v\in V.
	\end{split}
\end{equation*}

\subsection{Mass and viscosity operators}	In the dynamic problem \eqref{mainprob} the inertial term is described by a linear operator $\mathbb{M}\colon U^*\to U^*$, which represents a \emph{mass distribution}, such that
	\begin{equation}\label{mass}
		\text{\parbox{13cm}{$\mathbb{M}$ is an automorphism of $U^*$, and the restriction $\mathbb{M}_{|_W}\colon W\to W$ is continuous, symmetric and positive-definite.} }
	\end{equation}
	 	This assumption implies that $|w|_{\mathbb{M}}:=\sqrt{\langle \mathbb{M}w,w\rangle_W}$ defines an equivalent norm on $W$.  
	 	
	The presence of \emph{viscosity} is also modeled by the linear operator $\mathbb{V}\colon V\to V^*$, which we require to satisfy
	\begin{equation}\label{viscosity}
			\text{$\mathbb V$ is continuous, symmetric and positive-definite.}
	\end{equation}
As before, by setting $|v|_\mathbb{V}:=\sqrt{\langle \mathbb{V}v,v\rangle_V}$ we define an equivalent norm on $V$. 

We finally observe that \eqref{viscosity} yields that the viscous operator $\mathbb V$ admits the inverse $\V^{-1}\colon V^*\to V$, which is still a linear, continuous, symmetric and positive-definite operator.

\subsection{Rate-independent dissipation}	
Rate-independent effects are included by considering a \emph{rate-independent dissipation potential} $\mc R\colon Z\to[0,+\infty)$. We make the following assumption:
	\begin{enumerate}[label=\textup{(R\arabic*)}, start=1]
		\item \label{hyp:R1}  the function $\mc R$ is convex, positively homogeneous of degree one and there exist two positive constants $\rho_2\ge\rho_1>0$ for which
	\begin{equation}\label{Rbounds}
		\rho_1\|z\|_Z\le \mc R(z)\le \rho_2\|z\|_Z,\quad\text{ for every }z\in Z.
	\end{equation}
	\end{enumerate}
	Assumption \ref{hyp:R1} implies subadditivity, namely 
	\begin{equation*}
		\mc R (z_1+z_2)\leq \mc R(z_1)+\mc R(z_2)\,, \quad \mbox{ for every }z_1,z_2\in Z\,,
	\end{equation*} 
whence we infer that $\mc R$ is Lipschitz continuous on $Z$, indeed
\begin{equation*}
|\mc R(z_1) - \mc R(z_2)| \leq \max\{\mc R(z_1-z_2), \mc R(z_2-z_1)\} \leq \rho_2\|z_1 - z_2\|_Z\,, \quad \mbox{ for every }z_1,z_2\in Z\,.
\end{equation*}
In particular, this fact together with the convexity of $\mc R$ implies the weak lower semicontinuity of $\mc R$ on $Z$. 

	Furthermore, since $\mc R$ is positively one-homogeneous, for every $z \in Z$ its subdifferential $\partial^Z\mc R(z)$ can be characterized by
	\begin{equation}
		\partial^Z \mc R(z) = \{\eta\in \partial^Z\mc R(0):\,\,  \langle \eta,z\rangle_Z=\mc R(z)\}\subseteq \partial^Z\mc R(0).
		\label{2.2mielke}
	\end{equation}
	By \eqref{Rbounds} we also notice that there holds
	\begin{equation}\label{boundedness}
		\partial^Z\mc R(0)\subseteq \overline{B}^{Z^*}_{\rho_2}.
	\end{equation}
	It is also well-known that $\partial^Z\mc R(0)$ coincides with the proper domain of the Fenchel conjugate $\mc R^*$ of $\mc R$, indeed it actually holds
	\begin{equation}\label{eq:Rstar}
		\mc R^*(\eta)= \chi_{\partial^Z\mc R(0)}(\eta), \qquad\text{for all }\eta\in Z^*.
	\end{equation}

Due to one-homogeneity, the (convex) subdifferential of $\mc R$ is invariant with respect to the underlying topology, as stated in the next proposition.
\begin{prop} Assume \eqref{eq:embeddings}. Then for any functional $\mc R\colon Z\to [0,+\infty)$ satisfying \ref{hyp:R1} the following identities hold:
\begin{equation}\label{eqsubdtopology}
\partial^Z\mc R(u) = \partial^U\mc R(u)\,,\quad \text{for all } u\in U\,, \qquad \partial^Z\mc R(v) = \partial^V\mc R(v)\,,\quad \text{ for all } v\in V\,.
\end{equation}
\end{prop}
\begin{proof}
We only prove the first assertion, the other one being analogous.
Let $\eta  \in \partial^Z\mc R(u)$ for some $u\in U$. Then, for all $y\in U$ we have
\begin{equation*}
\mc R(y) \geq \mc R(u) + \langle \eta, y-u \rangle_Z = \mc R(u) + \langle \eta, y-u \rangle_U\,,
\end{equation*}
whence $\eta  \in \partial^U\mc R(u)$. Thus, $\partial^Z\mc R(u)\subseteq \partial^U\mc R(u)$.

To prove the reverse inclusion, let $\eta  \in \partial^U\mc R(u)$ for some $u\in U$. We first show that $\eta$ can be extended to an element of $Z^*$. Indeed, since $\mc R$ is one-homogeneous, from the very definition of subdifferential we have $\mc R(y)\geq \langle \eta, y \rangle_U$ for every $y\in U$, whence by \eqref{Rbounds} we get $|\langle \eta, y \rangle_U|\leq \rho_2 \|y\|_Z$ for every $y\in U$. Since $U$ is dense in $Z$, we obtain that $\eta$ belongs to $Z^*$. 

Now, let $z\in Z$ and $y_n\in U$ be such that $y_n\stackrel{Z}{\longrightarrow} z$. We then have
\begin{equation*}
\begin{split}
\mc R(z)& = \lim_{n\to+\infty} \mc R(y_n) \geq \lim_{n\to+\infty} \left( \mc R(u) + \langle \eta, y_n - u \rangle_U\right)  = \lim_{n\to+\infty} \left( \mc R(u) + \langle \eta, y_n - u \rangle_Z\right)  \\
& = \mc R(u) + \langle \eta, z - u \rangle_Z \,,
\end{split}
\end{equation*}  
which gives $\eta \in \partial^Z\mc R(u)$. This concludes the proof.  
\end{proof}
	
	\subsection{Potential energy}\label{subsec:poten}
	Both the dynamic and the rate-independent systems \eqref{mainprob} and \eqref{riprob} are driven by the time-dependent \emph{potential energy} $\mc E\colon [0,T]\times U\to [0,+\infty]$, with proper domain ${\rm dom}\,\mc E = [0,T]\times D$, for some nonempty set $D\subseteq U$. We assume that $\mc E$ possesses the following properties:
	\begin{enumerate}[label=\textup{(E\arabic*)}]
		\item \label{hyp:E1} (\textbf{Weak Lower Semicontinuity}) for every $t\in[0,T]$ the function $\mc E(t,\cdot)$ is weakly lower semicontinuous in $U$;
		\item \label{hyp:E2} (\textbf{Power control}) for every $u\in D$ the function $\mc E(\cdot,u)$ is absolutely continuous in $[0,T]$, and there holds 
		\begin{equation*}
			\left|{\partial_t} \mc E(t,u)\right|\le b(t)(\mc E(t,u)+1),\quad\text{for a.e. $t\in[0,T]$ and for every $u\in D$},
		\end{equation*}
		where $b\in L^1(0,T)$ is nonnegative;
		\item \label{hyp:E3} (\textbf{Coercivity}) the function $\mc E(0,\cdot)$ has bounded sublevels;
		\item \label{hyp:E4} (\textbf{Power continuity}) for all $M>0$ there exist $\gamma_M\in L^1(0,T)$ and a modulus of continuity $\omega_M\colon [0,+\infty)\to[0,+\infty)$ such that for almost every $t\in[0,T]$ and for all $u_1, u_2\in \{\mc E(0,\cdot)\le M\}$ there holds
		\begin{equation*}
			|\partial_t\mc E(t,u_1)-\partial_t\mc E(t,u_2)|\le\gamma_M(t)\omega_M(|u_1-u_2|_W);
		\end{equation*}
		\item \label{hyp:E5} (\textbf{$\lambda$-convexity}) there exists $\lambda\ge0$ such that for every $t\in[0,T]$ the map $ U\ni u \mapsto \mc E(t,u) + \frac{\lambda}{2}|u|_W^2 $ is convex.
	\end{enumerate}
\begin{rmk}\label{rmk:gronwall}
	By employing Gr\"onwall inequality, assumption \ref{hyp:E2} ensures that
	\begin{equation}\label{eq:ineqgronwall}
		\mc E(t,u)+1\le e^{\left|\int_{s}^{t}b(r)\d r\right|}(\mc E(s,u)+1),\qquad\text{for every }s,t\in [0,T].
	\end{equation}
	By using \ref{hyp:E1}, the above inequality easily implies that $\mc E$ is weakly lower semicontinuous in the product space $[0,T]\times U$. Moreover, by combining \eqref{eq:ineqgronwall} with \ref{hyp:E3}, there follows that sublevels of $\mc E(t,\cdot)$ are bounded \emph{uniformly} with respect to $t\in [0,T]$.	
\end{rmk} 

In the case of nonconvex energies (namely, when the parameter $\lambda$ in \ref{hyp:E5} is necessarily positive) we will need to add the following less standard assumptions:
\begin{enumerate}[label=\textup{(E\arabic*)}, start=6]
	\item \label{hyp:E6} (\textbf{Subgradient control}) for all $C_1>0$ there exists $C_2>0$ such that for every $(t,u)\in [0,T]\times D$ there holds
	\begin{equation*}
		\mc E(t,u)\le C_1\implies\sup\limits_{\xi\in\partial^U\mc E(t,u)}\|\xi\|_{U^*} \le C_2;
	\end{equation*}
	\item \label{hyp:E7} (\textbf{Continuity}) for every $t\in [0,T]$ the restriction of $\mc E(t,\cdot)$ to $D$ is strongly continuous in $U$, namely
	\begin{equation}\label{eq:strongcont}
		u_n\stackrel{U}{\rightarrow}u,\,\text{ with }u_n,u\in D \implies \mc E(t,u_n)\to \mc E(t,u);
	\end{equation} 
	\item \label{hyp:E8}(\textbf{Improved convergence}) for every $t\in [0,T]$ it holds
		\begin{equation*}
			u_n\stackrel{U}{\rightharpoonup}u\,\text{ and }\mc E(t,u_n)\to \mc E(t,u),\,\text{ with }u_n,u\in D\implies u_n\stackrel{U}{\rightarrow}u,
		\end{equation*}
	namely weak convergence of elements in $D$ together with convergence of energies imply strong convergence;
	\item \label{hyp:E9}(\textbf{Time-continuity of subdifferential}) for all $M>0$ there exists a modulus of continuity $\widetilde{\omega}_M\colon[0,+\infty)\to [0,+\infty)$ such that
	\begin{equation*}
		\partial^U\mc E(t,u)\subseteq\partial^U\mc E(s,u)+\overline{B}^{Z^*}_{\widetilde{\omega}_M(|t-s|)},\quad\text{for every }s,t\in [0,T]\,\text{and }u\in \overline{B}^U_M;
	\end{equation*}
\item \label{hyp:E10} (\textbf{Improved coercivity}) for all $C>0$ the set 
\begin{equation*}
	\left\{u\in U:\, \mc E(0,u)+\min\limits_{\xi\in\partial^U\mc E(0,u)\cap Z^*}\|\xi\|_{Z^*}\le C\right\},
\end{equation*}
is \emph{precompact} in $Z^*$.
\end{enumerate}

\begin{rmk}\label{rmk:precompact}
	If \ref{hyp:E2} and \ref{hyp:E9} are in force, then the validity of \ref{hyp:E10} can be extended in a uniform way to all times $t\in [0,T]$. Indeed, it can be proved that for all $C>0$ there exists $D>0$ such that
	\begin{equation*}
	\bigcup_{t\in [0,T]}\!\!\left\{u\in U:\, \mc E(t,u)+\!\!\min\limits_{\xi\in\partial^U\mc E(t,u)\cap Z^*}\|\xi\|_{Z^*}\le C\right\}\!\subseteq\!\left\{u\in U:\, \mc E(0,u)+\!\!\!\min\limits_{\xi\in\partial^U\mc E(0,u)\cap Z^*}\|\xi\|_{Z^*}\le D\right\}\!.
	\end{equation*}
\end{rmk}
	
	\section{Inertial Balanced Viscosity solutions}\label{sec:IBV}
	
	We begin this section by introducing the notion of solution to the dynamic problem \eqref{mainprob} we intend to consider. It is based on a suitable energy equality in the spirit of De Giorgi's \emph{energy-dissipation principle} (see also \cite{BachoViscoplast}, \cite[Section~5]{MielkRoubbook}, \cite{RivScilSol} and \cite{RosThom}), in which the key idea consists in keeping together all the dissipative effects, modelled by means of the viscous and the rate-independent potentials. 
	
	To this aim, we thus introduce the functional $\mc R_\eps\colon V\to [0,+\infty)$ defined as
	\begin{equation}\label{eq:Reps}
		\mc R_\eps(v):= \frac\eps 2|v|^2_\V+\mc R(v).
	\end{equation}
	By using \eqref{eq:Rstar}, it is well known (see for instance \cite[Section~2.3]{MielkRosSav16}) that its Fenchel conjugate $\mc R^*_\eps$ can be explicitely computed:
	\begin{equation}\label{eq:Reps*}
		\mc R_\eps^*(\zeta)=\frac{1}{2\eps}\dist^2_{\V^{-1}}(\zeta;\partial^Z\mc R(0))=\frac{1}{2\eps}\min\limits_{\eta\in \partial^Z\mc R(0)}|\zeta-\eta|_{\V^{-1}}^2,\quad\text{ for }\zeta\in V^*.
	\end{equation}
	Notice that $\mc R^*_\eps$ is continuous in $V^*$ and convex.
	
	Moreover, by Lemma~\ref{lem:lemma4} and \eqref{eqsubdtopology} one can write
	\begin{equation}\label{eq:split}
		\partial^V\mc R_\eps(v)=\eps\V v+\partial^Z\mc R(v),\qquad\text{for all }v\in V,
	\end{equation}
so that the dynamic differential inclusion can be equivalently rewritten as
	\begin{equation*}
		\eps^2 \mathbb{M}\xepsdd(t)+\partial^V\mc R_\eps(\xepsd(t))+\partial^U\mc E(t,\xeps(t))\ni 0, \quad\text{in $U^*$,}\quad \text{for a.e. }t\in [0,T].
	\end{equation*}
By using \eqref{Fenchelprop}, it is thus natural to expect, at least at a formal level (assuming everything is regular and smooth), that solutions to \eqref{mainprob} should satisfy in any subinterval interval $[s,t]\subseteq [0,T]$ an energy balance of the form
\begin{align*}
	&\quad\,\int_s^t {\mc R}_\eps(\xepsd(r))+\mc R_\eps^*(-\eps^2 \mathbb{M}\xepsdd(r)-\xi^\eps(r))\d r= \int_s^t \langle{-\eps^2 \mathbb{M}\xepsdd(r)-\xi^\eps(r)},{\xepsd(r)}\rangle_V\d r\\
	=& -\frac{\varepsilon^2}{2}|\xepsd(t)|_{\mathbb{M}}^2+\frac{\varepsilon^2}{2}|\xepsd(s)|_{\mathbb{M}}^2- \mc E(t,\xeps(t))+ \mc E(s,\xeps(s))+  \int_s^t \partial_t\mc E(r,\xeps(r))\,\mathrm{d}r,
\end{align*}
where $\xi^\eps(t)$ is an element of $\partial^U\mc E(t,\xeps(t))$.

Previous considerations lead to the following definition.
	\begin{defi}[\textbf{Dissipative dynamic solutions}]\label{def:dynsol}		
	Given initial data $u_0^\eps\in D$, $u_1^\eps\in V$, we say that a function $u^\eps\colon[0,T]\to D$ is a \emph{dissipative dynamic solution} to \eqref{mainprob} if the following conditions are satisfied:
	\begin{itemize}
		\item[$\bullet$] $u^\eps\in C_{\rm w}([0,T];U)\cap W^{1,2}(0,T;V)\cap C^1_{\rm w}([0,T];W)\cap W^{2,2}(0,T;U^*)$, and \\
		$\mc E(\cdot,u^\eps(\cdot))\in L^\infty(0,T)$;
		\item[$\bullet$] $u^\eps(0)=u^\eps_0$ and $\dot u^\eps(0)=u^\eps_1$;
		\item[$\bullet$] for almost every $t\in [0,T]$ there holds
		\begin{equation}\label{dineq}
			\eps^2 \mathbb{M}\xepsdd(t)+\eps \mathbb{V}\xepsd(t)+\eta^\eps(t)+\xi^\eps(t)= 0, \quad\text{in $U^*$,}
		\end{equation}
		for some (measurable) $\eta^\eps(t)\in\partial^Z\mc R(\xepsd(t))$ and $\xi^\eps(t)\in\partial^U\mc E(t,\xeps(t))$;
		\item[$\bullet$] there exists a set $N\subseteq (0,T]$ of null measure such that for all $t\in [0,T]\setminus N$ the energy-dissipation balance
		\begin{equation}\label{eq:energybalance}
			\begin{aligned}
				&\frac{\varepsilon^2}{2}|\xepsd(t)|_{\mathbb{M}}^2+ \mc E(t,\xeps(t))+\int_0^t {\mc R}_\eps(\xepsd(r))+\mc R_\eps^*(-\eps^2 \mathbb{M}\xepsdd(r)-\xi^\eps(r))\,\mathrm{d}r \\
				= \,&\frac{\varepsilon^2}{2}|u^\eps_1|_{\mathbb{M}}^2 +  \mc E(0,u_0^\eps) +  \int_0^t \partial_t\mc E(r,\xeps(r))\,\mathrm{d}r,
			\end{aligned}
		\end{equation}
	holds true.
	\end{itemize}
If the energy-dissipation balance \eqref{eq:energybalance} is actually satisfied for all $t\in [0,T]$, we call $u^\eps$ an \emph{exact} dissipative dynamic solution.
\end{defi}
\begin{rmk}
	By using equation \eqref{dineq} together with \eqref{eq:split} and \eqref{Fenchelprop}, the energy-dissipation balance \eqref{eq:energybalance} implies the validity of the following (integral) chain-rule formula for every $s,t\in [0,T]\setminus N$ with $s\le t$:
	\begin{equation}\label{eq:chainrule}
		\begin{aligned}
		\int_s^t\langle\eps^2\M\ddot{u}^\eps(r)+\xi^\eps(r),\dot{u}^\eps(r)\rangle_V\d r=& \frac{\varepsilon^2}{2}|\xepsd(t)|_{\mathbb{M}}^2+ \mc E(t,\xeps(t))-\frac{\varepsilon^2}{2}|\xepsd(s)|_{\mathbb{M}}^2- \mc E(s,\xeps(s))\\
		&-\int_s^t \partial_t\mc E(r,\xeps(r))\,\mathrm{d}r.
	\end{aligned}
	\end{equation}
\end{rmk}
\begin{rmk}
	Since $u^\eps\in C_{\rm w}([0,T];U)$ and $\dot u^\eps\in  C_{\rm w}([0,T];W)$, by weak lower semicontinuity the energy-dissipation balance \eqref{eq:energybalance} yields the inequality
	\begin{equation}\label{energyineq}
		\begin{aligned}
			&\frac{\varepsilon^2}{2}|\xepsd(t)|_{\mathbb{M}}^2+ \mc E(t,\xeps(t))+\int_s^t {\mc R}_\eps(\xepsd(r))+\mc R_\eps^*(-\eps^2 \mathbb{M}\xepsdd(r)-\xi^\eps(r))\,\mathrm{d}r \\
			\le \,&\frac{\varepsilon^2}{2}|\xepsd(s)|_{\mathbb{M}}^2 +  \mc E(s,\xeps(s)) +  \int_s^t \partial_t\mc E(r,\xeps(r))\,\mathrm{d}r,
		\end{aligned}
	\end{equation}
	for all $t\in [0,T]$ and for all $s\in [0,T]\setminus N$ with $s\le t$.
\end{rmk}

	Existence of dissipative dynamic solutions to \eqref{mainprob} has been proved in \cite[Theorem~2.4]{BachoViscoplast} in a more general functional framework than \eqref{eq:embeddings}, but requiring slightly stronger assumptions on the potential energy $\mc E$. To the best of our knowledge, existence of such solutions is not known just under our assumptions. Nevertheless, in the case $V=U$, even existence of \emph{exact} dissipative dynamic solutions can be proved assuming solely conditions \ref{hyp:E1}-\ref{hyp:E6}; a key role is played by the validity of a suitable chain-rule formula (see Lemma~\ref{lemma:crappendix}). For self-containedness we present the proof in Appendix~\ref{App:Existence}, and we refer again to \cite[Theorem~2.4]{BachoViscoplast} or to \cite[Theorem~5.6]{RosThom} for similar results. We however point out that the situation $V=U$ still includes interesting applications, such as Kelvin-Voigt models in viscoelasticity (see Section~\ref{subsec:vectorialcase}).
	
	 We finally observe that, exploiting \ref{hyp:E2}, \ref{hyp:E3} and \eqref{Rbounds}, a standard application of Gr\"onwall's Lemma to the energy-dissipation inequality \eqref{energyineq} (see \cite[Proposition~3.3]{GidRiv} for more details) yields the following uniform bounds.
	\begin{prop}\label{prop:unifbounds}
		Assume \eqref{eq:embeddings}, \eqref{mass} and \eqref{viscosity}. Let the rate-independent dissipation potential $\mc R$ satisfy \ref{hyp:R1} and let the potential energy $\mc E$ satisfy \ref{hyp:E1}-\ref{hyp:E3}. Let $u^\eps$ be a dissipative dynamic solution to \eqref{mainprob} with initial data $u^\eps_0\in D$, $u_1^\eps\in V$, and assume that $\eps u_1^\eps$ is uniformly bounded in $W$ and $\mc E(0,u_0^\eps)$ is uniformly bounded. Then there exists a constant $\overline{C}>0$, independent of $\eps$, such that
		\begin{itemize}
			\item[(i)] $\sup\limits_{t\in [0,T]}\mc E(t, u^\eps(t))\le \overline{C}$;
			\item[(ii)] $\sup\limits_{t\in [0,T]}\mc \| u^\eps(t)\|_U\le \overline{C}$;
			\item[(iii)] $\sup\limits_{t\in [0,T]}\mc \eps| \dot u^\eps(t)|_\mathbb{M}\le \overline{C}$;
			\item[(iv)] $\displaystyle\int_{0}^{T}\|\dot u^\eps(r)\|_Z\d r\le \overline{C}$;
			\item[(v)] $\displaystyle\eps\int_{0}^{T}|\dot u^\eps(r)|^2_\mathbb{V}\d r\le \overline{C}$.
		\end{itemize}
	If in addition \ref{hyp:E6} is in force, then there also holds 
	\begin{itemize}
		\item[(vi)] $\essup\limits_{t\in [0,T]}\|\xi^\eps(t)\|_{U^*}\le \overline{C}$.
	\end{itemize}
	\end{prop}
	
	\subsection{Viscous contact potential and viscoinertial cost}\label{subsec:cost}
	In the asymptotic analysis of the dynamic problem \eqref{mainprob} as $\eps\to 0$ a pivotal role is played by the so-called \emph{viscous contact potential} $p_\V\colon V\times V^*\to [0,+\infty)$, introduced in \cite{MielkRosSav09, MielkRosSav12, MielkRosSav16} and defined as  
\begin{equation}\label{eq:contactpotential}
	p_\V(v,\zeta):=\inf\limits_{\eps>0}(\mc R_\eps(v)+\mc R_\eps^*(\zeta)), \qquad\text{for }(v,\zeta)\in V\times V^*.
\end{equation}
By using formulas \eqref{eq:Reps} and \eqref{eq:Reps*}, an easy computation yields the explicit expression
\begin{equation}\label{eq:contactpotexpl}
	 p_\V(v,\zeta)=\mc R(v)+|v|_\V\dist_{\V^{-1}}(\zeta;\partial^Z\mc R(0)), \qquad\text{for all }(v,\zeta)\in V\times V^*.
\end{equation}
We collect here the main properties of the viscous contact potential $p_\mathbb{V}$, whose proofs are a direct consequence of the explicit formula \eqref{eq:contactpotexpl}. For more details we refer to \cite{MielkRosSav16}.

\begin{prop}\label{prop:visccont}
	Assume \eqref{viscosity} and \ref{hyp:R1}. Then the \emph{viscous contact potential} $p_\V$ defined in \eqref{eq:contactpotential} fulfils the following properties:
	\begin{itemize}
	 \item[(a)] for all $\zeta\in V^*$ the map $p_\V(\cdot,\zeta)$ is convex and positively one-homogeneous;
	 \item[(b)] for all $v\in V$ the map $p_\V(v,\cdot)$ is convex;
	 \item[(c)] $p_\V(v,\zeta)\ge\max\left\{\mc R(v),\langle\zeta,v\rangle_V\right\}$ for all $(v,\zeta)\in V\times V^*$;
	 \item[(d)] $p_\V(v(\cdot),\zeta(\cdot))\in L^1(a,b)$ whenever $v\in L^2(a,b;V)$ and $\zeta\in L^2(a,b;V^*)$;
	 \item[(e)] $p_\V(v,\zeta)\le C|v|_V(1+|\zeta|_{V^*})$ for all  $(v,\zeta)\in V\times V^*$.		
	\end{itemize}
\end{prop}

Besides the viscous contact potential, which describes the residual viscous effects, in order to measure the rate-dependent inertial effects surviving in the limit $\eps\to 0$ we need to introduce a suitable \emph{viscoinertial cost}. Differently from the finite-dimensional setting studied in \cite{RivScilSol}, where this concept has been proposed (Definition~3.5 therein), in the current infinite-dimensional setting some care is needed for a proper definition.

Here we discuss how to rigorously formulate this notion. For the sake of brevity, for all $a<b$ we define the set
\begin{equation}\label{eq:G}
	\begin{aligned}
		G(a,b):=\Big\{v\colon [a,b]\to D:\,&  v\in C_{\rm w}([a,b];U)\cap W^{1,2}(a,b;V)\cap  C^1_{\rm w}([a,b];W)
		\cap W^{2,2}(a,b;U^*),\\
		&\text{and }\mc E(0,v(\cdot))\in B([a,b]) \Big\},
	\end{aligned}
\end{equation}
namely the set of functions possessing the same regularity than dissipative dynamic solutions to \eqref{mainprob}. Now, for all $t\in [0,T]$, for all $u_1,u_2\in U$ and for all $\alpha,\beta\ge 0$ we set
\begin{equation}\label{eq:costalphabeta}
	c^\alpha_\beta(t;u_1,u_2):=\inf\left\{\int_{0}^{\sigma}p_\mathbb{V}(\dot{v}(r),-\mathbb{M}\ddot{v}(r)-\xi(r))\d r:\, \sigma>0,\, (v,\xi)\in AD^\alpha_\beta(t;u_1,u_2;\sigma)\right\},
\end{equation}
where the set of admissible functions is given by
\begin{equation}\label{eq:admfunct}
	\begin{aligned}
		AD^\alpha_\beta(t;u_1,u_2;\sigma):=\Big\{& (v,\xi)\in G(0,\sigma)\times L^\infty(0,\sigma;U^*):\,\text{conditions \eqref{eq:adm} below are satisfied} \Big\}.
	\end{aligned}		
\end{equation}

The constants $\overline{C}$ and $\widetilde{C}$ appearing in the conditions below, characterizing the set of admissible functions, are the ones introduced in Propositions~\ref{prop:unifbounds} and \ref{prop:inequality0}:
\begin{subequations}\label{eq:adm}
\begin{align}
	&\bullet\,\int_{0}^{\sigma}\|\dot v(r)\|_Z\d r\le 2(\overline{C}\vee\widetilde{C});\label{eq:adm1}\\
	&\bullet\, \xi(r)\in\partial^U\mc E(t,v(r))+\overline{B}_\alpha^{Z^*}\text{ for almost every }r\in[0,\sigma];\label{eq:adm2}\\
	&\bullet\, v(0)\in\overline{B}_\beta^{U}(u_1),\, v(\sigma)\in\overline{B}_\beta^{U}(u_2),\text{ and } \dot v(0),\dot{v}(\sigma)\in\overline{B}_\beta^{V};\label{eq:adm3}\\
	&\nonumber\bullet\, \mathbb{M}\ddot{v}+\xi\in L^2(0,\sigma;V^*), \text{ and there holds }\\
	&\quad\, \int_{0}^{\sigma}\langle \mathbb{M}\ddot{v}(r)+\xi(r), \dot{v}(r)\rangle_V\d r\le \frac 12|\dot{v}(\sigma)|_\M^2+\mc E(t,v(\sigma))-\frac 12|\dot{v}(0)|_\M^2-\mc E(t,v(0))+2(\overline{C}\vee\widetilde{C})\alpha.\label{eq:adm4} 
\end{align}
\end{subequations}

Here, with a little abuse of notation, in the case $\alpha=0$ or $\beta=0$ we mean $\overline{B}_0^{X}(x)=\{x\}$.

\begin{rmk}\label{rmk:V=U}
	Observe that in the case $V=U$ we can equivalently rewrite \eqref{eq:G}  in the simpler form
	\begin{equation}\label{eq:Gsimpl}
		\begin{aligned}
			G(a,b)=\Big\{v\colon [a,b]\to D:\,& v\in W^{1,2}(a,b;U)\cap W^{2,2}(a,b;U^*), \text{ and }\mc E(0,v(\cdot))\in B([a,b])\Big\}.
		\end{aligned}
	\end{equation}
	Moreover, if \ref{hyp:E7} is in force, condition $\mc E(0,v(\cdot))\in B([a,b])$ in \eqref{eq:Gsimpl} can be dropped: indeed, since $W^{1,2}(a,b;U)\subseteq C([a,b];U)$, in this case the map $r\mapsto\mc E(0,v(r))$ is continuous, and hence bounded, in $[a,b]$.
	
	Assuming \ref{hyp:E1},\ref{hyp:E2},\ref{hyp:E5} and \ref{hyp:E6}, we can also rewrite \eqref{eq:admfunct} as 
	\begin{equation*}
		\begin{aligned}
			AD^\alpha_\beta(t;u_1,u_2;\sigma)\!=\!\Big\{& (v,\xi)\in G(0,\sigma)\times L^\infty(0,\sigma;U^*):\text{conditions \eqref{eq:adm1}-\eqref{eq:adm3} are satisfied}  \Big\}.
		\end{aligned}		
	\end{equation*}
Condition \eqref{eq:adm4} is indeed automatically satisfied by using Lemma~\ref{lemma:chainruleequality} in the Appendix~\ref{App:Chainrule}. To see this, pick any $(v,\xi)\in G(0,\sigma)\times L^\infty(0,\sigma;U^*)$ satisfying \eqref{eq:adm1} and \eqref{eq:adm2}, and let $\xi_t\in \partial^U\mc E(t,v(\cdot))$ such that $\|\xi-\xi_t\|_{L^\infty(0,\sigma;Z^*)}\le \alpha$ be given by \eqref{eq:adm2}. Then, by using \eqref{eq:chainruleequality} there holds
\begin{align*}
	&\quad\,\int_{0}^{\sigma}\langle \mathbb{M}\ddot{v}(r)+\xi(r), \dot{v}(r)\rangle_U\d r= \int_{0}^{\sigma}\langle \mathbb{M}\ddot{v}(r)+\xi_t(r), \dot{v}(r)\rangle_U\d r+\int_{0}^{\sigma}\langle \xi(r)-\xi_t(r), \dot{v}(r)\rangle_Z\d r\\
	&\le \frac 12|\dot{v}(\sigma)|_\M^2+\mc E(t,v(\sigma))-\frac 12|\dot{v}(0)|_\M^2-\mc E(t,v(0))+\|\xi-\xi_t\|_{L^\infty(0,\sigma;Z^*)}\int_{0}^{\sigma}\|\dot{v}(r)\|_Z\d r\\
	&\le \frac 12|\dot{v}(\sigma)|_\M^2+\mc E(t,v(\sigma))-\frac 12|\dot{v}(0)|_\M^2-\mc E(t,v(0))+2(\overline{C}\vee\widetilde{C})\alpha.
\end{align*}
\end{rmk}

We now possess all the ingredients in order to introduce the notion of \emph{viscoinertial cost} which will govern the jump transient of Inertial Balanced Viscosity solutions (see Definition~\ref{def:IBV} below). 
\begin{defi}[\textbf{Viscoinertial cost}]
	For any $t\in [0,T]$, $u_1,u_2\in U$, the \emph{viscoinertial cost} related to $\M,$ $\V$, $\mc R$ and $\mc E$ is defined by
	\begin{equation}\label{eq:cost}
		c^{\M,\V}(t;u_1,u_2):=\sup\limits_{\alpha,\beta>0}c^\alpha_\beta(t;u_1,u_2)=\lim\limits_{(\alpha,\beta)\to(0,0)}c^\alpha_\beta(t;u_1,u_2),
	\end{equation}
 where $c^\alpha_\beta(t;u_1,u_2)$ has been introduced in \eqref{eq:costalphabeta}.
\end{defi}

We point out that the equality in \eqref{eq:cost}, namely the fact that the cost can be computed as a limit, follows since the functions $\alpha\mapsto c^\alpha_\beta(t;u_1,u_2)$ and $\beta\mapsto c^\alpha_\beta(t;u_1,u_2)$ are nonincreasing in $[0,+\infty)$.

It will also be useful to introduce the following more precise version of the viscoinertial cost, in which the boundary data of the admissible functions are attained:
\begin{equation}\label{eq:cost0}
	c^{\M,\V}_0(t;u_1,u_2):=\sup\limits_{\alpha>0}c^\alpha_0(t;u_1,u_2)=\lim\limits_{\alpha\to 0}c^\alpha_0(t;u_1,u_2).
\end{equation}
Note that by monotonicity for all $t\in [0,T]$ and $u_1,u_2\in U$ there holds
\begin{equation}\label{eq:ordcost}
	c^{\M,\V}(t;u_1,u_2)\le c^{\M,\V}_0(t;u_1,u_2).
\end{equation}

We conclude this section by showing a crucial property possessed by the viscoinertial cost, namely that it always provides an upper bound for the energy gap between two points $u_1,u_2\in D$.

\begin{prop}\label{prop:uppercost}
	Assume \eqref{viscosity}, \ref{hyp:R1} and let \ref{hyp:E7} be in force. Then for all $t\in [0,T]$ and for all $u_1,u_2\in D$ there holds
	\begin{equation*}
		\mc E(t,u_1)-\mc E(t,u_2)\le c^{\M,\V}(t;u_1,u_2).
	\end{equation*}
\end{prop}
\begin{proof}
	Fix $t\in [0,T]$ and without loss of generality let $c^{\M,\V}(t;u_1,u_2)<+\infty$, so that in particular $c^\alpha_\beta(t;u_1,u_2)<+\infty$ for all $\alpha,\beta>0$ as well. Let us fix $\delta,\alpha>0$, and by continuity let $\beta>0$ be such that
	\begin{equation}\label{eq:vicino}
		|\mc E(t,u_i)-\mc E(t,x)|\le\frac\delta 8, \quad\text{for all $x\in \overline{B}^U_\beta(u_i)\cap D$},\qquad |y|_\M^2\le \frac\delta 4, \quad\text{for all $y\in \overline{B}^V_\beta$}.
	\end{equation}
	By definition \eqref{eq:costalphabeta}, let us now choose $\sigma>0$ and $(v,\xi)\in AD^\alpha_\beta(t;u_1,u_2;\sigma)$ such that
	\begin{equation}\label{eq:sup}
		\int_{0}^{\sigma}p_\V(\dot v(r),-\M\ddot v(r)-\xi(r))\d r\le c^\alpha_\beta(t;u_1,u_2)+\frac \delta2.
	\end{equation}
	Then, using \eqref{eq:vicino} and \eqref{eq:sup}, and exploiting \eqref{eq:adm4} and $(c)$ in Proposition~\ref{prop:visccont} we obtain
	\begin{align*}
		\mc E(t,u_1)-\mc E(t,u_2)&\le \frac 12 |\dot v(0)|_\M^2+\mc E(t,v(0))-\frac 12 |\dot v(\sigma)|_\M^2-\mc E(t,v(\sigma))+\frac\delta 2\\
		&\le \int_{0}^{\sigma}\langle -\M\ddot v(r)-\xi(r),\dot v(r)\rangle_V\d r+2(\overline{C}\vee\widetilde{C})\alpha+\frac\delta 2\\
		&\le \int_{0}^{\sigma}p_\V( \dot v(r), -\M\ddot v(r)-\xi(r))\d r+2(\overline{C}\vee\widetilde{C})\alpha+\frac\delta 2\\
		&\le c^\alpha_\beta(t;u_1,u_2)+2(\overline{C}\vee\widetilde{C})\alpha+\delta\le  c^{\M,\V}(t;u_1,u_2)+2(\overline{C}\vee\widetilde{C})\alpha+\delta.
	\end{align*}
	We thus conclude by sending $\delta\to 0$ and $\alpha\to 0$.
\end{proof}

\subsection{Definition of IBV solutions}
Once the viscoinertial cost has been introduced, we are now in a position to define the core notion of this paper, namely the concept of \emph{Inertial Balanced Viscosity solution} to the rate-independent system \eqref{riprob}. In order to make the definition more clear, we first state the following lemma.
	\begin{lemma}\label{lemma:localstability}
	Assume \eqref{eq:embeddings}, let the rate-independent dissipation potential $\mc R$ satisfy \ref{hyp:R1} and let the potential energy $\mc E$ satisfy \ref{hyp:E5}. Then for $(t,\bar u)\in [0,T]\times D$ the following two conditions are equivalent:
	\begin{itemize}
		\item[$(a)$] $\partial^Z\mc R(0)+\partial^U\mc E(t,\bar u)\ni 0$, namely there exists $\bar\xi_t\in \partial^U\mc E(t,\bar u)$ satisfying $-\bar \xi_t\in \partial^Z\mc R(0)$;
		\item[$(b)$] $\mc E(t,\bar u)\le \mc E(t,x)+\mc R(x-\bar u)+\frac \lambda 2|x-\bar u|_W^2$ for all $x\in U$.
	\end{itemize}
\end{lemma}
\begin{proof}
	By definition of convex subdifferential, condition $(b)$ is equivalent to
	\begin{equation*}
		0\in \partial^U\Big(\mc E(t,\cdot)+\mc R(\cdot -\bar u)+\frac \lambda 2|\cdot-\bar u|_W^2\Big)(\bar u).
	\end{equation*}
	By using Lemmas~\ref{lem:lemma3} and \ref{lem:lemma4}, the right-hand side of the above line can be rewritten as
	\begin{align*}
		&\partial^U\Big(\mc E(t,\cdot)+\frac \lambda 2|\cdot-\bar u|_W^2\Big)(\bar u)+\partial^U\Big(\mc R(\cdot -\bar u)\Big)(\bar u)\\
		=\,& \partial^U\mc E(t,\bar u)+\partial^U\mc R(0)=\partial^U\mc E(t,\bar u)+\partial^Z\mc R(0),
	\end{align*}
	where we exploited \eqref{eqsubdtopology} in the last equality. Hence we conclude.
\end{proof}

\begin{defi}[\textbf{IBV solutions}]\label{def:IBV}
	Given an initial datum $u_0\in D$, we say that a function $u\colon[0,T]\to D$ is an \emph{Inertial Balanced Viscosity (IBV) solution} to the rate-independent system \eqref{riprob} if the following conditions are satisfied:
	\begin{itemize}
		\item[$\bullet$] $u\in B([0,T];U)\cap BV([0,T];Z)$;
		\item[$\bullet$] $u(0)=u_0$;
		\item[$\bullet$] the \emph{local stability condition}
		\begin{equation}\label{eq:LS}
			\partial^Z\mc R(0)+\partial^U\mc E(t,u(t))\ni 0,\qquad\text{for every } t\in [0,T]\setminus J_u,
		\end{equation}
	holds true at each continuity point of $u$;
	\item[$\bullet$] the \emph{viscoinertial energy-dissipation balance}
	\begin{equation}\label{eq:energybalanceIBV}
		\mc E(t,u^+(t))+V_\mc R(u_{\rm co};s,t)+\sum_{r\in J_u^{\rm e}\cap [s,t]}c^{\M,\V}(r;u^-(r),u^+(r))=\mc E(s,u^-(s))+\int_{s}^{t}\partial_t\mc E(r,u(r))\d r,
	\end{equation}
	holds true for all $0\le s\le t\le T$, where $u_{\rm co}$ is the continuous part of $u$ (in the topology of $Z$) introduced in \eqref{eq:decompositionf}, $J_u^{\rm e}$ is the essential jump set defined in \eqref{essjump} and the viscoinertial cost $c^{\M,\V}(r;u^-(r),u^+(r))$ has been defined in \eqref{eq:cost}.
	\end{itemize}
If the viscoinertial energy-dissipation balance \eqref{eq:energybalanceIBV} is satisfied with $c^{\M,\V}_0(r;u^-(r),u^+(r))$, defined in \eqref{eq:cost0}, in place of $c^{\M,\V}(r;u^-(r),u^+(r))$, we call $u$ a \emph{precise IBV solution} to the rate-independent system \eqref{riprob}. 
\end{defi}
\begin{rmk}\label{rmk:upm}
	We observe that the above definition is well posed in our framework \eqref{eq:embeddings}. Indeed, being $u$ in $BV([0,T];Z)$, for all $t\in [0,T]$ there exist the strong limits in $Z$
	\begin{subequations}\label{limitpm}
	\begin{equation}\label{limitzeta}
		u(s)\xrightarrow[s\to t^\pm]{Z}u^\pm(t).
	\end{equation}
	Moreover, since $u$ also belongs to $B([0,T];U)$, the same limits hold true in the weak topology of $U$ as well (see also Lemma~\ref{lem:lemma5}), namely
	\begin{equation}\label{limitu}
		u(s)\xrightharpoonup[s\to t^\pm]{U}u^\pm(t).
	\end{equation}

	We now show that $u^\pm(t)$ necessarily are in the proper domain $D$ of the energy $\mc E(t,\cdot)$ for all $t\in [0,T]$. First notice that assumption \eqref{eq:embeddings} yields
	\begin{equation}\label{limitvw}
		u(s)\xrightharpoonup[s\to t^\pm]{V}u^\pm(t),\qquad\text{and}\qquad u(s)\xrightarrow[s\to t^\pm]{W}u^\pm(t).
	\end{equation}
\end{subequations}
Then, by means of Lemma~\ref{lemma:localstability}, the local stability condition \eqref{eq:LS} is equivalent to the global $\lambda$-stability
\begin{equation*}
	\mc E(t,u(t))\le \mc E(t,x)+\mc R(x-u(t))+\frac \lambda 2|x-u(t)|_W^2,\quad\text{ for all $x\in U$ and for all $t\in [0,T]\setminus J_u$}.
\end{equation*}
By using \eqref{limitpm} and recalling the weak lower semicontinuity of $\mc E$, we now deduce
\begin{itemize}
	\item[$\bullet$] $\mc E(t,u^+(t))\le \mc E(t,x)+\mc R(x-u^+(t))+\frac \lambda 2|x-u^+(t)|_W^2$ for all $x\in U$ and for all $t\in [0,T]$,
	\item[$\bullet$] $\mc E(t,u^-(t))\le \mc E(t,x)+\mc R(x-u^-(t))+\frac \lambda 2|x-u^-(t)|_W^2$ for all $x\in U$ and for all $t\in (0,T]$,
\end{itemize}
which, recalling that $u_0\in D$, yield $u^\pm(t)\in D$ for all $t\in [0,T]$ as well.
\end{rmk}

\subsection{Statement of the main results}\label{subsec:mainthms}
We now collect the main results of the paper, regarding the asymptotic behaviour of \emph{dissipative dynamic solutions} to \eqref{mainprob} as $\eps\to 0$. The first theorem deals with the more general case of nonconvex potential energy $\mc E$: in this setting, jumps of the limit evolution are expected, and moreover both viscous and inertial effects are needed for their description. For this reason, \emph{IBV solutions} to \eqref{riprob} naturally come into play. If the energy $\mc E$ is convex, Theorem~\ref{mainthm:convex} instead states that only rate-independent effects survive in the asymptotic analysis, thus the more classical notion of (essential) \emph{energetic solution} \cite{MielkTheil1, MielkTheil} is enough to capture all the features of the limit evolution. Finally, in the particular case of uniformly convex potential energy depicted in Theorem~\ref{mainthm:unifconvex}, the limit function turns out to be (absolutely) continuous, and hence classic solutions to the rate-independent system \eqref{riprob} are obtained.

The proof of the theorems below will be given in Section~\ref{sec:slowloading}.

\begin{thm}[\textbf{Nonconvex case $\lambda>0$}]\label{mainthm:nonconvex}
	Assume \eqref{eq:embeddings}. Let the mass and viscosity operators $\mathbb{M}$ and $\mathbb{V}$ satisfy \eqref{mass} and \eqref{viscosity}. Let the rate-independent dissipation potential $\mc R$ satisfy \ref{hyp:R1} and let the potential energy $\mc E$ satisfy \ref{hyp:E1}-\ref{hyp:E10} for a \emph{positive} parameter $\lambda$ in \ref{hyp:E5}.
	
	Let $u^\eps$ be a dissipative dynamic solution to \eqref{mainprob} with initial data $u^\eps_0\in D$, $u_1^\eps\in V$ in the sense of Definition~\ref{def:dynsol}, and assume that
	\begin{equation*}
		\eps u^\eps_1\xrightarrow[\eps\to 0]{V}0,\qquad\text{and}\qquad u^\eps_0\xrightarrow[\eps\to 0]{U}u_0,\qquad\text{for some $u_0\in D$.}
	\end{equation*}

	Then there exists an \emph{Inertial Balanced Viscosity (IBV) solution} $u\colon [0,T]\to D$ to the rate-independent system \eqref{riprob} in the sense of Definition~\ref{def:IBV} such that, up to subsequences, there hold:
	\begin{itemize}
		\item $u^\eps(t)\xrightharpoonup[\eps\to 0]{U}u(t),\quad$ for all $t\in [0,T]$;
		\item $u^\eps(t)\xrightarrow[\eps\to 0]{U}u(t),\quad$ for almost every $t\in [0,T]$;
		\item $\eps\dot{u}^\eps(t)\xrightarrow[\eps\to 0]{V}0,\quad$ for almost every $t\in [0,T]$.
	\end{itemize}

Moreover, if we assume $D=U=V$, then the function $u$ is actually a \emph{precise IBV solution} to \eqref{riprob}.
\end{thm}

\begin{thm}[\textbf{Convex case $\lambda=0$}]\label{mainthm:convex}
	Assume \eqref{eq:embeddings}. Let the mass and viscosity operators $\mathbb{M}$ and $\mathbb{V}$ satisfy \eqref{mass} and \eqref{viscosity}. Let the rate-independent dissipation potential $\mc R$ satisfy \ref{hyp:R1} and let the potential energy $\mc E$ satisfy \ref{hyp:E1}-\ref{hyp:E5}, with $\lambda=0$ in \ref{hyp:E5}.
	
		Let $u^\eps$ be a dissipative dynamic solution to \eqref{mainprob} with initial data $u^\eps_0\in D$, $u_1^\eps\in V$ in the sense of Definition~\ref{def:dynsol}, and assume that
	\begin{equation*}
		\eps u^\eps_1\xrightarrow[\eps\to 0]{W}0,\qquad u^\eps_0\xrightharpoonup[\eps\to 0]{U}u_0,\qquad\text{and}\qquad \mc E(0,u^\eps_0)\to\mc E(0,u_0),
	\end{equation*}
	for some $u_0$ satisfying the (global) stability condition
	\begin{equation}\label{eq:GS0}
		\mc E(0,u_0)\le \mc E(0,x)+\mc R(x-u_0),\quad\text{for all $x\in U$}.
	\end{equation}
	
	Then there exists a function $u\colon [0,T]\to D$ satisfying $u\in B([0,T];U)\cap BV([0,T];Z)$ such that, up to subsequences, there hold:
	\begin{itemize}
		\item $u^\eps(t)\xrightharpoonup[\eps\to 0]{U}u(t),\quad$ for all $t\in [0,T]$;
		\item $\eps\dot{u}^\eps(t)\xrightarrow[\eps\to 0]{V}0,\quad$ for almost every $t\in [0,T]$.
	\end{itemize}
The limit function $u$ is an \emph{essential energetic solution} to the rate-independent system \eqref{riprob}, namely $u(0)=u_0$ and it fulfils the following global stability condition and rate-independent energy-dissipation balance
\begin{equation}\label{eq:essensol}
	\begin{cases}
		\mc E(t,u(t))\le \mc E(t,x)+\mc R(x-u(t)),\quad\text{for all $x\in U$ and for all }t\in [0,T]\setminus J_u,\\
		\mc E(t,u^+(t))+V_\mc R^{\rm e}(u;s,t)=\mc E(s,u^-(s))+\int_{s}^{t}\partial_t\mc E(r,u(r))\d r,\quad\text{for all }0\le s\le t\le T,
	\end{cases}
\end{equation}
where the essential $\mc R$-variation $V_\mc R^{\rm e}$ is defined as 
\begin{equation*}
	V_\mc R^{\rm e}(u;s,t):=V_\mc R(u_{\rm co};s,t)+\sum_{r\in J_u^{\rm e}\cap [s,t]}\mc R(u^+(r)-u^-(r)).
\end{equation*}
\end{thm}

\begin{thm}[\textbf{Uniformly convex case}]\label{mainthm:unifconvex}
	Assume \eqref{eq:embeddings}. Let the mass and viscosity operators $\mathbb{M}$ and $\mathbb{V}$ satisfy \eqref{mass} and \eqref{viscosity}. Let the rate-independent dissipation potential $\mc R$ satisfy \ref{hyp:R1} and let the potential energy $\mc E$ satisfy \ref{hyp:E1}-\ref{hyp:E4}, with \ref{hyp:E5} replaced with the following uniform convexity condition:
	\begin{enumerate}[label=\textup{(E\arabic*')}, start=5]
		\item\label{hyp:E5'}there exists $\mu>0$ such that for all $t\in [0,T]$, $\theta\in (0,1)$, $u_1,u_2\in U$ one has
		\begin{equation*}
			\mc E(t,\theta u_1+(1-\theta u_2))\le \theta\mc E(t,u_1)+(1-\theta)\mc E(t,u_2)-\frac\mu 2\theta(1-\theta)\|u_1-u_2\|^2_X,
		\end{equation*}
		for some Banach space $X$ with $U\hookrightarrow X$.
	\end{enumerate}

Let $u^\eps$ be a dissipative dynamic solution to \eqref{mainprob} with initial data $u^\eps_0\in D$, $u_1^\eps\in V$ in the sense of Definition~\ref{def:dynsol}, and assume that
\begin{equation*}
	\eps u^\eps_1\xrightarrow[\eps\to 0]{W}0,\qquad u^\eps_0\xrightharpoonup[\eps\to 0]{U}u_0,\qquad\text{and}\qquad \mc E(0,u^\eps_0)\to\mc E(0,u_0),
\end{equation*}
for some $u_0$ satisfying the (global) stability condition \eqref{eq:GS0}. 

 Then Theorem~\ref{mainthm:convex} applies, and the limit function $u$ in addition belongs to $C([0,T];X)$, and so also to $C_{\rm w}([0,T];U)$. Moreover $u$ is an \emph{energetic solution} of the rate-independent system \eqref{riprob}, namely for all $t\in [0,T]$ there holds
	\begin{equation}\label{eq:ensol}
		\begin{cases}
			\mc E(t,u(t))\le \mc E(t,x)+\mc R(x-u(t)),\quad\text{for all $x\in U$ };\\
			\mc E(t,u(t))+V_\mc R(u;0,t)=\mc E(0,u_0)+\int_{0}^{t}\partial_t\mc E(r,u(r))\d r.
		\end{cases}
	\end{equation}
	
	If in addition the following stronger version of \ref{hyp:E4} is in force:
	\begin{enumerate}[label=\textup{(E\arabic*')}, start=4]
		\item\label{hyp:E4'} for all $M>0$ there exist $\gamma_M\in L^1(0,T)$ such that for almost every $t\in[0,T]$ and for all $u_1, u_2\in \{\mc E(0,\cdot)\le M\}$ there holds
	\begin{equation}\label{lipschitzpartialt}
		|\partial_t\mc E(t,u_1)-\partial_t\mc E(t,u_2)|\le\gamma_M(t)\|u_1-u_2\|_X,
	\end{equation}
	\end{enumerate}
	then $u$ actually belongs to $AC([0,T];X)$.
	
	If finally one also has $X\hookrightarrow Z$ and $X$ is reflexive, then $u$ satisfies the rate-independent system \eqref{riprob} in the classical sense, namely
	\begin{equation}\label{rieq}
		\begin{cases}
			\partial^Z\mc R(\dot{u}(t))+\partial^U\mc E(t,u(t))\ni 0, \quad\text{in $U^*$,}\quad \text{for a.e. }t\in [0,T],\\
			u(0)=u_0.
		\end{cases}
	\end{equation}
\end{thm}
	
	\subsection{Time-incremental minimization scheme}\label{subsec:minmov}
	
	In this section we instead collect our results concerning the convergence of the Minimizing Movements scheme associated to \eqref{mainprob} to an IBV solution. We start by fixing some notation and by introducing the scheme we want to analyze.
		
	Let $\tau\in (0,1)$ be a fixed time step such that $\frac{T}{\tau}\in\mathbb{N}$. We consider the corresponding induced partition $\Pi_\tau:=\{t^k\}_k$ of the time-interval $[0,T]$, defined by $t^k:=k\tau$ where $k=0,1,\dots,T/\tau$. For future use we also define $t^{-1}:=-\tau$ and we set $\mathcal{K}_\tau:=\{1,\dots,T/\tau\}$ and $\mc K_\tau^0:=\mc K_\tau\cup\{0\}$. \par	
	Given $u_0^\varepsilon\in D$ and $u_1^\varepsilon\in V$, we construct a recursive sequence $\{u^k_{\tau,\varepsilon}\}_{k\in \mc K_\tau}$ by solving the following iterated minimum problem:
	\begin{subequations}\label{schemeincond}
		\begin{equation}
			u^k_{\tau,\varepsilon}\in \mathop{\rm arg\,min}\limits_{u\in U}\mathcal{F}_{\tau,\varepsilon}(t^k,u, u^{k-1}_{\tau,\varepsilon}, u^{k-2}_{\tau,\varepsilon}),\quad \text{for }k\in \mc K_\tau,
			\label{scheme}
		\end{equation}
		with initial conditions 
		\begin{equation}
			u^0_{\tau,\varepsilon}:=u_0^\varepsilon\,,\quad u^{-1}_{\tau,\varepsilon}:= u_0^\eps-\tau u_1^\varepsilon\,,
		\end{equation}
	\end{subequations}
	where the functional is defined as
	\begin{align*}
		\mathcal{F}_{\tau,\varepsilon}(t^k,u, u^{k-1}_{\tau,\varepsilon},u^{k-2}_{\tau,\varepsilon}):= \frac{\varepsilon^2}{2\tau^2}|u-2u_{\tau,\varepsilon}^{k-1}+u_{\tau,\varepsilon}^{k-2}|^2_{\mathbb{M}}
		+\tau\mc R_\varepsilon\left(\frac{{u-u_{\tau,\varepsilon}^{k-1}}}{\tau}\right) + \mc E(t^{k},u)\,,
	\end{align*}
	and $\mc R_\varepsilon$ has been introduced in \eqref{eq:Reps}. 
	We observe that the existence of a minimum in \eqref{scheme} easily follows from the direct method, since $u\mapsto\mathcal{F}_{\tau,\varepsilon}(t^k,u, u^{k-1}_{\tau,\varepsilon},u^{k-2}_{\tau,\varepsilon})$ is coercive and lower semicontinuous with respect to the weak topology of $U$. Furthermore, if $\tau/\eps$ is small enough, the minimum is unique by strict convexity of the functional $\mathcal{F}_{\tau,\varepsilon}(t^k,\cdot, u^{k-1}_{\tau,\varepsilon},u^{k-2}_{\tau,\varepsilon})$ (if \ref{hyp:E5} is in force).\par  
	By defining
	\begin{equation}\label{eq:vdiscr}
		v_{\tau,\varepsilon}^k:=\frac{u_{\tau,\varepsilon}^k-u_{\tau,\varepsilon}^{k-1}}{\tau},\quad\text{for }k\in \mc K^0_\tau,
	\end{equation}
 we notice that the Euler Lagrange equation solved by $u_{\tau,\varepsilon}^k$ reads as
	\begin{equation}\label{EL}
		\eps^2 \mathbb{M}\frac{v_{\tau,\varepsilon}^k-v^{k-1}_{\tau,\varepsilon}}{\tau}
		+\partial^V\mc R_\varepsilon(v_{\tau,\varepsilon}^k)+\partial^U\mc E(t^k,u_{\tau,\varepsilon}^k)\ni 0, \quad \mbox{ in $U^*$,} \quad \text{for all } k\in\mathcal{K}_\tau\,,
	\end{equation}
namely there exist $\eta^k_{\tau,\eps}\in \partial^Z\mc R(v^k_{\tau,\eps})$ and $\xi^k_{\tau,\eps}\in\partial^U\mc E(t^k,u^k_{\tau,\eps})$ such that
	\begin{equation}\label{eq:EL2}
	\eps^2 \mathbb{M}\frac{v_{\tau,\varepsilon}^k-v^{k-1}_{\tau,\varepsilon}}{\tau}
	+\eps\V v^k_{\tau,\eps}+\eta^k_{\tau,\eps}+\xi^k_{\tau,\eps}=0, \quad \mbox{ in $U^*$,} \quad \text{for all } k\in\mathcal{K}_\tau\,.
\end{equation}

We now introduce the standard (piecewiese constant, piecewise affine and piecewise quadratic) interpolants of $u^k_{\tau,\eps}$: 
\begin{subequations}\label{interpolants}
\begin{align}
	&\overline{u}_{\tau,\varepsilon}(t):=u^{k}_{\tau,\varepsilon},&&\text{for } t\in(t^{k-1},t^k],\,\,k\in\mathcal{K}^0_\tau\,;\\
	& \underline{u}_{\tau,\varepsilon}(t):=u^{k-1}_{\tau,\varepsilon},&&\text{for } t\in[t^{k-1},t^k),\,\, k\in\mathcal{K}^0_\tau\,;\\
	&\widehat{u}_{\tau,\varepsilon}(t):= v^k_{\tau,\varepsilon}(t-t^{k-1})+u^{k-1}_{\tau,\varepsilon},&&\text{for } t\in(t^{k-1},t^k],\,\, k\in\mc K^0_\tau\,;\\
	&\widetilde{u}_{\tau,\varepsilon}(t):=u^0_\eps-\frac{\tau}{2}u_1^\eps+\int_{0}^{t}\dot{\widetilde{u}}_{\tau,\eps}(r)\d r,&&\text{for }t\in [0,T],\\
	&\text{ where }\quad\dot{\widetilde{u}}_{\tau,\eps}(t):=\frac{v^k_{\tau,\varepsilon}-v^{k-1}_{\tau,\varepsilon}}{\tau}(t-t^{k-1})+v^{k-1}_{\tau,\varepsilon},&&\text{for } t\in(t^{k-1},t^k],\quad k\in\mc K_\tau. 
\end{align}
\end{subequations}
Our interest lies in the asymptotic behaviour of such interpolants as $(\tau,\eps)\to (0,0)$ in the regime
\begin{equation}\label{eq:regime}
	\lim\limits_{(\tau,\eps)\to (0,0)}\frac{\tau}{\eps}=0.
\end{equation}

For the general result in the nonconvex case, we will need the following technical Lipschitz continuity-type assumption for $\partial^U\mc E$:
\begin{enumerate}[label=\textup{(E\arabic*)}, start=11]
	\item \label{hyp:E11} (\textbf{Continuity of subdifferential}) for all $M>0$ there esists $C_M>0$ such that for any $t\in [0,T]$ and $u_1,u_2\in \overline{B}_M^U\cap D$ there holds
	\begin{equation*}
		\partial^U\mc E(t,u_1)\subseteq \partial^U\mc E(t,u_2)+\overline{B}_{C_M\|u_1-u_2\|_U}^{U^*}.
	\end{equation*}
\end{enumerate}
The convergence to an IBV solution will be in general recovered in the special case $D=U=V$, as stated in Theorem~\ref{mainthm:nonconvexdiscr}. This requirement, as well as \ref{hyp:E11}, is still consistent with most of the examples we will present in Section~\ref{sec:Applications}.

Instead, the convergence results in the convex case (Theorems~\ref{mainthm:convexdiscr} and \ref{mainthm:unifconvexdiscr}) will hold in the general setting $D\subseteq U\subseteq V$. Furthermore, as in the continuous-in-time framework, because of convexity only the first group of assumptions \ref{hyp:E1}-\ref{hyp:E4} will be needed.

The proof of the theorems below will be given in Section~\ref{sec:multiscale}.

\begin{thm}[\textbf{Nonconvex case $\lambda>0$}]\label{mainthm:nonconvexdiscr}
	Assume \eqref{eq:embeddings}. Let the mass and viscosity operators $\mathbb{M}$ and $\mathbb{V}$ satisfy \eqref{mass} and \eqref{viscosity}. Let the rate-independent dissipation potential $\mc R$ satisfy \ref{hyp:R1} and let the potential energy $\mc E$ satisfy \ref{hyp:E1}-\ref{hyp:E11} for a \emph{positive} parameter $\lambda$ in \ref{hyp:E5}. Assume $D=U=V$.
	
	Let $\widetilde{u}_{\tau,\eps},\widehat{u}_{\tau,\eps},\overline{u}_{\tau,\eps},\underline{u}_{\tau,\eps}$ be the interpolants \eqref{interpolants} of the Minimizing Movements scheme \eqref{schemeincond} with initial data $u^\eps_0\in U$, $u_1^\eps\in U$ for a sequence $(\tau,\eps)\to (0,0)$ satisfying \eqref{eq:regime}, and assume that
	\begin{equation*}
		\eps u^\eps_1\xrightarrow[\eps\to 0]{U}0,\qquad\text{and}\qquad u^\eps_0\xrightarrow[\eps\to 0]{U}u_0.
	\end{equation*}
	
	Then there exists a \emph{precise IBV solution} $u\colon [0,T]\to U$ to the rate-independent system \eqref{riprob} such that, up to subsequences, there hold:
	\begin{itemize}
		\item $\widetilde{u}_{\tau,\eps}(t),\widehat{u}_{\tau,\eps}(t),\overline{u}_{\tau,\eps}(t),\underline{u}_{\tau,\eps}(t)\xrightharpoonup[(\tau,\eps)\to (0,0)]{U}u(t),\quad$ for all $t\in [0,T]$;
		\item $\widetilde{u}_{\tau,\eps}(t),\widehat{u}_{\tau,\eps}(t),\overline{u}_{\tau,\eps}(t),\underline{u}_{\tau,\eps}(t)\xrightarrow[(\tau,\eps)\to (0,0)]{U}u(t),\quad$ for almost every $t\in [0,T]$;
		\item $\eps\dot{\widetilde{u}}_{\tau,\eps}(t),\eps\dot{\widehat{u}}_{\tau,\eps}(t)\xrightarrow[(\tau,\eps)\to (0,0)]{U}0,\quad$ for almost every $t\in [0,T]$.
	\end{itemize}
\end{thm}

\begin{thm}[\textbf{Convex case $\lambda=0$}]\label{mainthm:convexdiscr}
	Assume \eqref{eq:embeddings}. Let the mass and viscosity operators $\mathbb{M}$ and $\mathbb{V}$ satisfy \eqref{mass} and \eqref{viscosity}. Let the rate-independent dissipation potential $\mc R$ satisfy \ref{hyp:R1} and let the potential energy $\mc E$ satisfy \ref{hyp:E1}-\ref{hyp:E5}, with $\lambda=0$ in \ref{hyp:E5}.
	
	Let $\widetilde{u}_{\tau,\eps},\widehat{u}_{\tau,\eps},\overline{u}_{\tau,\eps},\underline{u}_{\tau,\eps}$ be the interpolants \eqref{interpolants} of the Minimizing Movements scheme \eqref{schemeincond} with initial data $u^\eps_0\in D$, $u_1^\eps\in V$ for a sequence $(\tau,\eps)\to (0,0)$ satisfying \eqref{eq:regime}, and assume that
	\begin{equation*}
		\eps u^\eps_1\xrightarrow[\eps\to 0]{W}0,\qquad u^\eps_0\xrightharpoonup[\eps\to 0]{U}u_0,\qquad\text{and}\qquad \mc E(0,u^\eps_0)\to\mc E(0,u_0),
	\end{equation*}
	for some $u_0$ satisfying the (global) stability condition \eqref{eq:GS0}.	 
	
	Then there exists an \emph{essential energetic solution} $u\colon [0,T]\to D$ to the rate-independent system \eqref{riprob} satisfying $u\in B([0,T];U)\cap BV([0,T];Z)$ and such that, up to subsequences, there hold:
	\begin{itemize}
		\item $\widetilde{u}_{\tau,\eps}(t)\xrightharpoonup[(\tau,\eps)\to (0,0)]{U}u(t),\quad$ for all $t\in (0,T]$,\ \ \ \  and $\quad\widetilde{u}_{\tau,\eps}(0)\xrightarrow[(\tau,\eps)\to (0,0)]{W}u(0)$;
		\item $\widehat{u}_{\tau,\eps}(t),\overline{u}_{\tau,\eps}(t),\underline{u}_{\tau,\eps}(t)\xrightharpoonup[(\tau,\eps)\to (0,0)]{U}u(t),\quad$ for all $t\in [0,T]$;
		\item $\eps\dot{\widetilde{u}}_{\tau,\eps}(t),\eps\dot{\widehat{u}}_{\tau,\eps}(t)\xrightarrow[(\tau,\eps)\to (0,0)]{V}0,\quad$ for almost every $t\in [0,T]$.
	\end{itemize}
	If in addition $\eps u_1^\eps$ is uniformly bounded in $V$, then there also holds
	\begin{equation*}
		\widetilde{u}_{\tau,\eps}(0)\xrightharpoonup[(\tau,\eps)\to (0,0)]{V}u(0).
	\end{equation*}
\end{thm}

\begin{thm}[\textbf{Uniformly convex case}]\label{mainthm:unifconvexdiscr}
	Assume \eqref{eq:embeddings}. Let the mass and viscosity operators $\mathbb{M}$ and $\mathbb{V}$ satisfy \eqref{mass} and \eqref{viscosity}. Let the rate-independent dissipation potential $\mc R$ satisfy \ref{hyp:R1} and let the potential energy $\mc E$ satisfy \ref{hyp:E1}-\ref{hyp:E4} and \ref{hyp:E5'}.
	
	Let $\widetilde{u}_{\tau,\eps},\widehat{u}_{\tau,\eps},\overline{u}_{\tau,\eps},\underline{u}_{\tau,\eps}$ be the interpolants \eqref{interpolants} of the Minimizing Movements scheme \eqref{schemeincond} with initial data $u^\eps_0\in D$, $u_1^\eps\in V$ for a sequence $(\tau,\eps)\to (0,0)$ satisfying \eqref{eq:regime}, and assume that
	\begin{equation*}
		\eps u^\eps_1\xrightarrow[\eps\to 0]{W}0,\qquad u^\eps_0\xrightharpoonup[\eps\to 0]{U}u_0,\qquad\text{and}\qquad \mc E(0,u^\eps_0)\to\mc E(0,u_0),
	\end{equation*}
	for some $u_0$ satisfying the (global) stability condition \eqref{eq:GS0}. 
	
	Then Theorem~\ref{mainthm:convexdiscr} applies, and the limit function $u$ in addition belongs to $C([0,T];X)$, and is an \emph{energetic solution} of the rate-independent system \eqref{riprob}.
	
	If in addition \ref{hyp:E4'} is in force
	then $u$ actually belongs to $AC([0,T];X)$.	If finally one also has $X\hookrightarrow Z$ and $X$ is reflexive, then $u$ is a \emph{classic solution} to the rate-independent system \eqref{riprob}.
\end{thm}

\section{Examples and applications}\label{sec:Applications}	
In this section we present some examples which can be included in our abstract analysis.
\subsection{Scalar case}\label{subsec:scalarcase}
We consider the differential inclusions of Allen-Cahn type
\begin{subequations}\label{eq:incl}
\begin{eqnarray}
& \eps^2 \xepsdd(t)+\eps \xepsd(t)+{\rm Sign}(\xepsd(t))-\Delta \xeps(t)+\partial\mathcal{W}(\xeps(t))\ni f(t), \quad \mbox{ in $(0,T)\times\Omega$,} \label{eq:inclA}\\
& \eps^2 \xepsdd(t)-\eps \Delta\xepsd(t)+{\rm Sign}(\xepsd(t))-\Delta \xeps(t)+\partial\mathcal{W}(\xeps(t))\ni f(t), \quad \mbox{ in $(0,T)\times\Omega$, } \label{eq:inclB} \end{eqnarray}
\end{subequations}
complemented with homogeneous Dirichlet boundary conditions and initial data. 

By ${\rm Sign}$ we mean the multivalued operator 
\begin{equation*}
{\rm Sign}(v):=
\begin{cases}
1\,, & \mbox{ if $v>0$, } \\
-1\,, & \mbox{ if $v<0$, } \\
[-1,1]\,, & \mbox{ if $v=0$. }
\end{cases}
\end{equation*}

Here, $\Omega$ is a bounded open subset of $\R^d$, with $d\ge 2$ (obviously the one-dimensional case can be treated in a simpler way), while the nonlinearity $\mc W\colon \R\to [0,+\infty]$ is lower semicontinuous and $\widetilde\lambda$-convex and the forcing term $f$ belongs to $AC([0,T];L^2(\Omega))$.

To model \eqref{eq:inclA} and \eqref{eq:inclB} we choose
\begin{equation}\label{eq:framework}
\begin{split}
&\bullet\quad U= H^1_0(\Omega)\,, \quad V =
\begin{cases}
L^2(\Omega)\,, & \mbox{ in case \eqref{eq:inclA}} \\
H^1_0(\Omega)\,, & \mbox{ in case \eqref{eq:inclB}}
\end{cases}\,,
\quad W=L^2(\Omega)\,, \quad Z=L^1(\Omega)\,, \\
&\bullet\quad \mathbb{M}\xi = \xi\,, \quad \quad \,\,\,\, \mathbb{V}v =
\begin{cases}
v\,, & \mbox{ in case \eqref{eq:inclA}} \\
-\Delta v\,, & \mbox{ in case \eqref{eq:inclB}}
\end{cases}\,, \\
&\bullet\quad \mathcal{R}(z) = \int_\Omega |z(x)|\,\mathrm{d}x\,.
\end{split}
\end{equation}
It is immediate to check \eqref{eq:embeddings}, \eqref{mass}, \eqref{viscosity} and \ref{hyp:R1}. Furthermore note that $Z^*=L^\infty(\Omega)$.

Finally, we introduce the potential energy
\begin{equation*}
\mathcal{E}(t,u) = \begin{cases}\displaystyle
	\int_\Omega\!\! \left(\frac{1}{2} |\nabla u|^2+\mathcal{W}(u)-f(t)u\right)\mathrm{d}x+C_0,&\text{if }u\in D:=\left\{u\in H^1_0(\Omega):\, \mc W(u)\in L^1(\Omega)\right\},\\
	+\infty,&\text{otherwise,}
\end{cases}
\end{equation*}
where $C_0\ge 0$ is a suitable constant so that $\mathcal{E}(t,u)$ is nonnegative (such $C_0$ exists since $\mc W$ is nonnegative and the leading term of $\mathcal{E}$ is quadratic). 

We now check conditions \ref{hyp:E1}-\ref{hyp:E5}:
\begin{itemize}
	\item Condition \ref{hyp:E1} directly follows by the weak lower semicontinuity of the norm together with Fatou's Lemma.
	\item We observe that $\partial_t \mc E(t,u)= -\int_\Omega \dot{f}(t)u\,\mathrm{d}x$ for almost every $t\in[0,T]$ and for all $u\in D$, and thus there holds
	\begin{equation*}
		\left|{\partial_t} \mc E(t,u)\right|\le |\dot{f}(t)|_{L^2(\Omega)}|u|_{L^2(\Omega)}\le C |\dot{f}(t)|_{L^2(\Omega)} (\mc E(t,u)+1)\,,
	\end{equation*}
	where in the last inequality we employed the quadratic growth (in the gradient of $u$) of $\mc E(t,\cdot)$. This gives \ref{hyp:E2}. 
	\item As for \ref{hyp:E3}, we notice that $\mc E(0,u)\leq C$ implies $|\nabla u|^2_{L^2(\Omega)}\leq C+|{f}(0)|_{L^2(\Omega)}|u|_{L^2(\Omega)}$,  and so $\|u\|_{H^1_0(\Omega)}\le C$ by Poincar\'e inequality. 
	\item We actually prove the validity of \ref{hyp:E4'} with $X=W=L^2(\Omega)$. Indeed for almost every $t\in[0,T]$ and for all $u_1,u_2\in D$ one has
	\begin{equation*}
		|\partial_t\mc E(t,u_1)-\partial_t\mc E(t,u_2)|=|\langle \dot{f}(t), u_1-u_2\rangle_{L^2(\Omega)}|\le|\dot{f}(t)|_{L^2(\Omega)}|u_1-u_2|_{L^2(\Omega)}. 
	\end{equation*}
	\item For what concerns \ref{hyp:E5}, since $\mathcal{W}$ is $\tilde{\lambda}$-convex and $\mc E(t,\cdot)$ is quadratic in $\nabla u$, for $\vartheta\in[0,1]$ and $u_1,u_2\in H^1_0(\Omega)$ it follows
	\begin{equation}\label{eq:exlambda}
		\begin{aligned}
			\mc E(t,\vartheta u_1 +(1-\vartheta)u_2) & \leq \vartheta \mc E(t,u_1) + (1-\vartheta) \mc E(t,u_2) \\
			&\quad + \frac{\vartheta}{2}(1-\vartheta) \left[\tilde{\lambda}|u_1-u_2|^2_{L^2(\Omega)} - |\nabla (u_1-u_2)|^2_{L^2(\Omega)}\right] \\
			& \leq \vartheta \mc E(t,u_1) + (1-\vartheta) \mc E(t,u_2) + \frac{\tilde{\lambda}-C_P^2}{2} \vartheta (1-\vartheta) |u_1 - u_2|^2_{L^2(\Omega)}\,,
		\end{aligned}
	\end{equation}
	where we employed Poincar\'e inequality with constant $C_P$. We thus have the following cases:
	\begin{itemize}
		\item[$(j)$] if $\lambda:=\tilde{\lambda}-C_P^2<0$, then the energy actually satisfies \ref{hyp:E5'} with $X=L^2(\Omega)$;
		\item[$(jj)$]  if $\lambda=\tilde{\lambda}-C_P^2=0$, then the energy is convex;
		\item[$(jjj)$]  if $\lambda=\tilde{\lambda}-C_P^2>0$, then the energy is non convex.
	\end{itemize}
\end{itemize}

If $(j)$ holds, for both systems \eqref{eq:inclA} and \eqref{eq:inclB} we can apply Theorem~\ref{mainthm:unifconvex} in its stronger form (since \ref{hyp:E4'} and \ref{hyp:E5'} are in force with $X=W=L^2(\Omega)$), namely the limit equation \eqref{rieq} is satisfied in the classical sense. In the same way, we can also apply Theorem~\ref{mainthm:unifconvexdiscr}.

If $(jj)$ holds, then the driving energy $\mc E(t,\cdot)$ is convex and we can apply Theorems~\ref{mainthm:convex} and \ref{mainthm:convexdiscr} for both systems \eqref{eq:inclA} and \eqref{eq:inclB}.

In the nonconvex case $(jjj)$ we also need to check conditions \ref{hyp:E6}-\ref{hyp:E10} in order to apply Theorem~\ref{mainthm:nonconvex}. To this aim we strenghten the assumptions requiring $\partial \Omega\in C^2$ and a more regular forcing term
\begin{equation}\label{eq:fZ*}
	f\in AC([0,T];L^\infty(\Omega)),
\end{equation}
and considering an \emph{everywhere finite} $\widetilde\lambda$-convex function $\mc W\colon \R\to [0,+\infty)$, which thus is locally Lipschitz, whose derivative satisfies the following growth condition for almost every $x\in \R$:
\begin{equation}\label{eq:(*)}
|\mathcal{W}'(x)| \leq C(1+|x|^{q-1}) \quad \mbox{ for some $q\in \left(1,\frac{d+2}{d-2}\right)\cap\left(1,\frac{2d-2}{d-2}\right]$,}
\end{equation}
where we convene that $\frac{d+2}{d-2}=\frac{2d-2}{d-2}=+\infty$ for $d=2$. Observe that the above condition implies
\begin{equation}
|\mathcal{W}(x)| \leq C(1+|x|^{q})\,,\quad\text{for all }x\in \R,
\label{eq:(**)}
\end{equation}
and notice that 
\begin{equation}\label{eq:q}
	2(q-1)\le 2^*,\quad q< 2^*,\quad \frac{2\cdot 2^*}{q-1}>d,
\end{equation}
where $2^*=\frac{2d}{d-2}$ denotes the Sobolev exponent. In particular, by Sobolev Embedding one has $\mathcal{W}(u)\in L^1(\Omega)$ for all $u\in H^1_0(\Omega)$, and so under these stronger assumptions there holds $D=U=H^1_0(\Omega)$. 
\begin{rmk}
	The standard double-well potential $\mathcal{W}(x)=\frac{1}{4}(1-x^2)^2$ fulfils  \eqref{eq:(*)} for $d=2$ and $d=3$.
\end{rmk}

  Let us now check conditions \ref{hyp:E6}-\ref{hyp:E10}:
  \begin{itemize}
  	\item for all $(t,u)\in[0,T]\times H^1_0(\Omega)$, under the current assumptions one has 
  	\begin{equation*}
  		\partial^U\mc E(t,u)=-\Delta u + \partial^U\mc W(u)-f(t),
  	\end{equation*}
  	with 
  	\begin{equation*}
  		\partial^U \mc W(u)=\left\{\widetilde{\xi}\in L^2(\Omega):\, \widetilde{\xi}(x)\in \partial\mc W(u(x))\text{ for a.e. }x\in\Omega\right\}.
  	\end{equation*}
  	Hence, by using \eqref{eq:(*)}, \eqref{eq:q} and Sobolev Embedding, for all $\widetilde{\xi}\in 	\partial^U \mc W(u)$ one deduces
  	\begin{equation*}
  		\|-\Delta u + \widetilde{\xi}-f(t)\|_{H^{-1}(\Omega)} \leq C (|\nabla u|_{L^2(\Omega)}+\| u\|^{q-1}_{H^1_0(\Omega)}+|f(t)|_{L^2(\Omega)}+1)\,,
  	\end{equation*}
  	and so \ref{hyp:E6} follows by means of \ref{hyp:E2} and \ref{hyp:E3} (see also Remark~\ref{rmk:gronwall}).
  	\item condition \ref{hyp:E7}, namely strong continuity of $\mc E(t,\cdot)$ in $H^1_0(\Omega)$, directly follows from \eqref{eq:(**)} and \eqref{eq:q} by an application of Dominated Convergence Theorem.
  	\item To show \ref{hyp:E8}, pick a sequence $u_n\xrightharpoonup[]{H^1_0(\Omega)} u$ such that $\mc E(t,u_n)\to\mc E(t,u)$. Then it holds
  	\begin{equation*}
  		\lim_{n\to+\infty} \left(\frac{1}{2} \int_\Omega |\nabla u_n|^2\,\mathrm{d}x + \int_\Omega \mathcal{W}(u_n)\,\mathrm{d}x \right) = \frac{1}{2} \int_\Omega |\nabla u|^2\,\mathrm{d}x + \int_\Omega \mathcal{W}(u)\,\mathrm{d}x\,.
  	\end{equation*}
  	Since both terms are weakly lower semicontinuous, it easily follows that separately
  	\begin{equation*}
  		\frac{1}{2} \int_\Omega |\nabla u_n|^2\,\mathrm{d}x \to \frac{1}{2} \int_\Omega |\nabla u|^2\,\mathrm{d}x, \quad \mbox{  and  }\quad \int_\Omega \mathcal{W}(u_n)\,\mathrm{d}x\to\int_\Omega \mathcal{W}(u)\,\mathrm{d}x\,,
  	\end{equation*}
  	whence $u_n\xrightarrow{H^1_0(\Omega)}u$.
  	\item Assumption \ref{hyp:E9} is a byproduct of \eqref{eq:fZ*}, since for all $s,t\in[0,T]$ and $u\in H^1_0(\Omega)$ one has 
  	\begin{equation*}
  		\partial^U\mc E(t,u)=-\Delta u + \partial^U\mc W(u)-f(t) = \partial^U\mc E(s,u) + f(s) - f(t),
  	\end{equation*}
  	and 
  	\begin{equation*}
  		\|f(s)-f(t)\|_{L^\infty(\Omega)} \leq C \left|\int_s^t \varphi(r)\,\mathrm{d}r\right|, \quad \mbox{ for some nonnegative $\varphi\in L^1(0,T)$. }
  	\end{equation*}
  \item Finally, we show the validity of \ref{hyp:E10}. Let $\{u_n\}_{n\in\N}\subseteq H^1_0(\Omega)$ be such that
  \begin{equation}\label{eq:boundE10}
  	\mc E(0,u_n) \leq C, \quad \mbox{ and } \quad \|-\Delta u_n + \widetilde{\xi}_n-f(0)\|_{L^\infty(\Omega)} \leq C\,\quad\text{for some }\widetilde{\xi}_n\in \partial^U\mc W(u_n).
  \end{equation}
  We need to show that, up to subsequences, $u_n$ converges in $Z^*=L^\infty(\Omega)$. By \eqref{eq:boundE10}, we know that $\|u_n\|_{H^1_0(\Omega)}\leq C$ and
  \begin{equation}\label{eq:(star)}
  	- \Delta u_n = - \widetilde{\xi}_n+f(0)+g_n, \quad \mbox{ in $H^{-1}(\Omega)$},
  \end{equation}
  with $\|g_n\|_{L^\infty(\Omega)}\le C$. Since by Sobolev Embedding $\|u_n\|_{L^r(\Omega)}\leq C$ for all $r\in[1,2^*)$, by using \eqref{eq:(*)} and \eqref{eq:q} we deduce that $\| \widetilde{\xi}_n\|_{L^{\frac{r}{q-1}}(\Omega)}\leq C$, and so the right-hand side of \eqref{eq:(star)} is uniformly bounded in $L^{\frac{r}{q-1}}(\Omega)$. 
  
  By Calder\'on-Zygmund elliptic regularity theory, for all $r\in(q-1,2^*)$ we thus obtain 
  \begin{equation*}
  	u_n\in W^{2,\frac{r}{q-1}}(\Omega), \quad  \mbox{with $\|u_n\|_{W^{2,\frac{r}{q-1}}(\Omega)}\leq C$.}
  \end{equation*}
  Now, by choosing $r$ close enough to $2^*$ such that $\frac{2r}{q-1}>d$ (such $r$ always exists due to \eqref{eq:q}), by Sobolev Embedding we then obtain $u_n\in C^{0,\alpha}(\overline{\Omega})$ with $\|u_n\|_{C^{0,\alpha}(\overline{\Omega})}\leq C$ for some $\alpha\in(0,1)$. We finally conclude by applying Ascoli-Arzel\'a Theorem.
  \end{itemize}

Summing up, if $(jjj)$ above holds, then we can apply Theorem~\ref{mainthm:nonconvex} under the additional assumptions \eqref{eq:fZ*} and \eqref{eq:(*)}. In case \eqref{eq:inclA}, one obtains an IBV solution in the limit; in case \eqref{eq:inclB}, since there holds $D=U=V$, the limit function is a precise IBV solution. 

Finally, in order to apply also Theorem~\ref{mainthm:nonconvexdiscr}, we are only left to check the validity of \ref{hyp:E11}. To this aim, we need to additionally require that $\mc W\in C^1(\mathbb{R})$ satisfies
\begin{equation}\label{eq:loclip}
	|\mc W'(x){-}\mc W'(y)|\le C(1{+}|x|^r{+}|y|^r)|x{-}y|,\text{ for all }x,y\in \mathbb{R},\text{ for some finite }r\in \left(0,\frac{4}{d-2}\right].
\end{equation} 
Notice that there holds
\begin{equation}\label{eq:propr}
	\frac{rd}{2}\le 2^*.
\end{equation}
\begin{rmk}
	One can directly check that the standard double-well potential $\mathcal{W}(x)=\frac{1}{4}(1-x^2)^2$ fulfils  \eqref{eq:loclip} with $r=2$, which lies in the admissible range for $d=2$ and $d=3$.
\end{rmk}

Since in this case $\partial^U\mc E(t,u)=\left\{-\Delta u+\mc W'(u)-f(t)\right\}$ for all $(t,u)\in[0,T]\times H^1_0(\Omega)$, in view of the standard inequality
\begin{equation*}
	\|-\Delta(u_1-u_2)\|_{H^{-1}(\Omega)}\le \|u_1-u_2\|_{H^1_0(\Omega)},\quad\text{for }u_1,u_2\in H^1_0(\Omega),
\end{equation*}
we just need to check
\begin{equation*}
	\|\mc W'(u_1)-\mc W'(u_2)\|_{H^{-1}(\Omega)}\le C_M\|u_1-u_2\|_{H^1_0(\Omega)},\quad\text{for }u_1,u_2\in H^1_0(\Omega)\text{ with }\|u_i\|_{H^1_0(\Omega)}\le M.
\end{equation*}
We prove the above inequality in the case $d>2$, being the case $d=2$ simpler. By means of \eqref{eq:loclip}, exploiting \eqref{eq:propr}, and by using H\"older inequality and Sobolev embedding, for an arbitrary $\varphi\in H^1_0(\Omega)$ we can estimate
\begin{align*}
	&\qquad\,|\langle \mc W'(u_1)-\mc W'(u_2),\varphi\rangle_{H^1_0(\Omega)}|\le C\int_{\Omega}|u_1-u_2|(1+|u_1|^r+|u_2|^r)|\varphi|\d x\\
	&\le C\left(\|u_1-u_2\|_{H^1_0(\Omega)}\|\varphi\|_{H^1_0(\Omega)}+\sum_{i=1}^2\int_{\Omega}|u_1-u_2||\varphi||u_i|^r\right)\\
	&\le C\left(\|u_1-u_2\|_{H^1_0(\Omega)}\|\varphi\|_{H^1_0(\Omega)}+\|u_1-u_2\|_{L^{2^*}(\Omega)}\|\varphi\|_{L^{2^*}(\Omega)}\sum_{i=1}^2\|u_i\|^r_{L^{rd/2}(\Omega)}\right)\\
	&\le C M^r\|u_1-u_2\|_{H^1_0(\Omega)}\|\varphi\|_{H^1_0(\Omega)},
\end{align*}
 and we conclude.

\subsubsection{$p$-Laplacian}

We slightly modify previous problem by considering the $p$-Laplacian instead of the standard Laplacian in \eqref{eq:incl}. We still are in the framework described in \eqref{eq:framework}, with the only change given by the choice
\begin{equation*}
U=W^{1,p}_0(\Omega),  \quad \mbox{ with $p>d$,}
\end{equation*}
and we observe that in this regime
\begin{equation}\label{eq:Holdembedding}
W^{1,p}_0(\Omega)\hookrightarrow C^{0,\alpha}(\overline{\Omega}), \quad \mbox{for some $\alpha\in(0,1)$.}
\end{equation}
The potential energy now takes the form
\begin{equation*}
	\mathcal{E}(t,u) = \begin{cases}\displaystyle
		\int_\Omega\!\! \left(\frac{1}{p} |\nabla u|^p{+}\mathcal{W}(u){-}f(t)u\right)\!\mathrm{d}x+C_0,&\text{if }u\in D:=\left\{u\in W^{1,p}_0(\Omega):\, \mc W(u)\in L^1(\Omega)\right\},\\
		+\infty,&\text{otherwise,}
	\end{cases}
\end{equation*}
where $f$ and $\mc W$ are as in the first part of  Section~\ref{subsec:scalarcase}.

The validity of \ref{hyp:E1}-\ref{hyp:E4'} can be checked as in Section~\ref{subsec:scalarcase}, while in the three cases appearing in condition \ref{hyp:E5} now one has to choose $\lambda=\widetilde{\lambda}$. Indeed, the $p$-integrability with respect to $\nabla u$ in the energy $\mc E$ does not allow for a quadratic reminder term in the inequality \eqref{eq:exlambda}. Thus, if $\mc W$ is convex we can apply Theorems~\ref{mainthm:convex} and \ref{mainthm:convexdiscr}, while if it is uniformly convex we can also apply Theorems~\ref{mainthm:unifconvex} and \ref{mainthm:unifconvexdiscr}.

If $\mc W$ is not convex, as before, we need to require \eqref{eq:fZ*} and $\mc W$ to be everywhere finite. However, since $p>d$, due to \eqref{eq:Holdembedding} there is no need to assume the growth condition \eqref{eq:(*)}. Conditions \ref{hyp:E6}-\ref{hyp:E9} now follow arguing exactly as in Section~\ref{subsec:scalarcase}, while \ref{hyp:E10} is automatically satisfied by \eqref{eq:Holdembedding} together with Ascoli-Arzel\'a Theorem. Thus, Theorem~\ref{mainthm:nonconvex} can be applied.  

\begin{rmk}
The regime $p\in(1,d]$, $p\neq 2$, is subtler and requires some care. Firstly, we observe that in case \eqref{eq:inclB} we need to require $p>2$ in order to have $U=W^{1,p}_0(\Omega) \stackrel{d}{\hookrightarrow}H^1_0(\Omega)= V$, while in case \eqref{eq:inclA} we need $p>\frac{2d}{d+2}$ in order to have $U=W^{1,p}_0(\Omega) \hookrightarrow\hookrightarrow L^2(\Omega)= W$.

Furthermore, a suitable growth condition on $\mc W$ has to be imposed for the validity of \ref{hyp:E1}-\ref{hyp:E10}. If, for instance, $\mc W\in C^1(\R)$ and for all $x\in \R$ satisfies
\begin{equation*}
|\mathcal{W}'(x)| \leq C(1+|x|^{q-1}) \quad \mbox{ for some $q\in [1,p^*]\cap\left(1,1+\frac{p^*}{2}\right)$,}
\end{equation*}
then \ref{hyp:E1}-\ref{hyp:E9} can be checked as before, with the usual caveats for \ref{hyp:E4'} and \ref{hyp:E5'}. As regards \ref{hyp:E10}, which we recall has to be checked only if $\mc W$ is not convex, one may argue similarly as before by applying the regularity theory for the $p$-Laplace equation (see \cite{DiBen}). 
\end{rmk}

\subsection{Vectorial case}\label{subsec:vectorialcase}

We may also consider the vectorial situation, presenting an application in Kelvin-Voigt viscoelasticity. Let $\Omega\subseteq \R^d$ be a bounded open set with Lipschitz boundary, and let $\partial_D\Omega$ be a subset of $\partial\Omega$ with positive $\mc H^{d-1}$-measure. Setting $\partial_N\Omega:=\partial\Omega\setminus\partial_D\Omega$, we consider the system
\begin{equation}\label{eq:inclC}
\begin{cases}
	\eps^2 M \xepsdd(t)+\nu{\rm Dir}(\xepsd(t)) - {\rm div}(\varepsilon \mathbb{D}e(\xepsd(t)) + \mathbb{C}e(\xeps(t))){+}\partial\mathcal{W}(\xeps(t))\ni f(t),& \text{in $(0,T)\times\Omega$,}\\
	u^\eps(t)=\dot u^\eps(t)=0,&\text{in $(0,T)\times\partial_D\Omega$},\\
	(\varepsilon \mathbb{D}e(\xepsd(t)) + \mathbb{C}e(\xeps(t)))n_\Omega=0,& \text{in $(0,T)\times\partial_N\Omega$,}
\end{cases}
\end{equation}
where $e(u)$ stands for the symmetric gradient $\frac{\nabla u+\nabla u^T}{2}$, while ${\rm Dir}$ represents the higher dimensional counterpart of ${\rm Sign}$, namely
\begin{equation*}
{\rm Dir}(v):=
\begin{cases}\displaystyle
\frac{v}{|v|}\,, & \mbox{ if $v\neq 0$, }\\
\overline{B}_1\,, & \mbox{ if $v=0$. }
\end{cases}
\end{equation*}

In \eqref{eq:inclC}, $M$ is a symmetric and positive-definite $d\times d$ matrix, $\nu>0$ is a positive number modelling a friction coefficient, while $\mathbb{C}, \mathbb{D}: \Omega \to \R^{d\times d\times d\times d}$ are fourth order tensors both satisfying the following usual assumptions in linearized elasticity (below $\mathbb{T}=\mathbb{C}$ or $\mathbb{T}=\mathbb{D}$): 
\begin{itemize}
\item $\mathbb{T}$ is uniformly continuous in $\Omega$;
\item $\mathbb{T}(x)A\in \R^{d\times d}_{\rm sym}$ for all $x\in\Omega$ and $A\in \R^{d\times d}$;
\item $\mathbb{T}(x)A=\mathbb{T}(x)A_{\rm sym}$ for all $x\in\Omega$ and $A\in \R^{d\times d}$;
\item $\mathbb{T}(x)A:B=\mathbb{T}(x)B:A$ for all $x\in\Omega$ and $A,B\in \R^{d\times d}$;
\item $\mathbb{T}(x)A:A\geq c_{\mathbb{T}} |A_{\rm sym}|^2$ for some $c_{\mathbb{T}}>0$ and for all $x\in\Omega$ and $A\in \R^{d\times d}$. 
\end{itemize}
The nonlinearity $\mc W\colon \R^d\to [0,+\infty]$ is assumed to be lower semicontinuous and $\widetilde\lambda$-convex and the forcing term $f$ belongs to $AC([0,T];L^2(\Omega;\R^d))$.

If $d=2$, system \eqref{eq:inclC} may model the dynamic evolution of a thin bidimensional viscoelastic material anchored on $\partial_D\Omega$ which slips over a flat rigid substrate under the effect of the time-dependent volume force $f(t)$, and the contact between the material and the substrate produces Coulomb friction with friction coefficient $\nu$.

In this framework we choose
\begin{equation*}
\begin{split}
& \bullet\quad U=V= \{u\in H^1(\Omega;\R^d):\, u=0\text{ $\mc H^{d-1}$-a.e. in }\partial_D\Omega\}\,, \quad W=L^2(\Omega;\R^d)\,, \quad Z=L^1(\Omega;\R^d)\,, \\
& \bullet\quad \mathbb{M}\xi =M \xi\,, \quad \quad \,\,\,\,\,\,\,\,\,\,\,\, \mathbb{V}v = - {\rm div}(\mathbb{D}e(v))\,, \\
& \bullet\quad \mathcal{R}(z) = \nu\|z\|_{L^1(\Omega;\R^d)}\,,
\end{split}
\end{equation*}
and
\begin{equation*}
	\mathcal{E}(t,u) = \begin{cases}\displaystyle
		\int_\Omega \left(\frac{1}{2}\mathbb{C}e(u):e(u) +\mathcal{W}(u)-f(t)\cdot u\right)\mathrm{d}x+C_0,&\text{if }u\in D,\\
		+\infty,&\text{otherwise,}
	\end{cases}
\end{equation*}
with $D=\left\{u\in U:\, \mc W(u)\in L^1(\Omega)\right\}$.

As before, conditions \eqref{eq:embeddings}, \eqref{mass}, and \ref{hyp:R1} are immediately satisfied, while \eqref{viscosity}, \ref{hyp:E1}-\ref{hyp:E4'} and \ref{hyp:E5} are verified arguing as in the scalar case by using Korn--Poincar\'e inequality in place of Poincar\'e inequality (the three cases $(j)$, $(jj)$ and $(jjj)$ are now distinguished by $\lambda:=\widetilde{\lambda}-c_{\mathbb C}K_P^2$, where $K_P$ denotes the Korn--Poincar\'e constant).

In cases $(j)$ and $(jj)$ we can apply Theorems~\ref{mainthm:unifconvex} (and \ref{mainthm:unifconvexdiscr}) and \ref{mainthm:convex} (and \ref{mainthm:convexdiscr}), respectively, for system \eqref{eq:inclC}. In case $(jjj)$ we also assume $\partial\Omega\in C^2$ and $\partial_D\Omega$ smooth, $f\in AC([0,T];L^\infty(\Omega;\R^d))$ and we require $\mc W\in C^1(\R^d)$ to be nonnegative, $\widetilde{\lambda}$-convex and to satisfy for all $x\in \R^d$ the following inequality:
\begin{equation} \label{eq:(***)}
	|\nabla\mathcal{W}(x)| \leq C(1+|x|^{q-1}), \quad \mbox{ for some $q\in \left(1,\frac{d}{d-2}\right)$.}
\end{equation}

The validity of \ref{hyp:E6}-\ref{hyp:E9} now follows as in the scalar case, while the validity of \ref{hyp:E10} requires the stronger assumption \eqref{eq:(***)} compared to \eqref{eq:(*)}, which also ensures $D=U$. Indeed, unlike the scalar framework, the regularity theory for weak solutions to vectorial systems (see \cite[Theorem~7.2]{GiaqMart}) only enhances the regularity of the gradient, and not of second order derivatives. Thus, arguing as in Section~\ref{subsec:scalarcase}, the regularity result \cite[Theorem~7.2]{GiaqMart} yields $u_n\in W^{1,\frac{r}{q-1}}(\Omega;\R^d)$ with $\|u_n\|_{W^{1,\frac{r}{q-1}}(\Omega;\R^d)}\leq C$ for all $r\in (2(q-1),2^*)$. Since $q<\frac{d}{d-2}$, we now can pick $r\in[1,2^*)$ such that $\frac{r}{q-1}>d$, and then we conclude by using again Sobolev Embedding. 

Thus, in case $(jjj)$, Theorem~\ref{mainthm:nonconvex} can be applied (in its stronger form $D=U=V$) for the vectorial system \eqref{eq:inclC} under the stronger assumption \eqref{eq:(***)}.

Finally, if in addition we assume
\begin{equation*}
	|\nabla\mc W(x){-}\nabla\mc W(y)|\le C(1{+}|x|^r{+}|y|^r)|x{-}y|,\text{ for all }x,y\in \mathbb{R}^d,\text{ for some finite }r\in \left(0,\frac{4}{d-2}\right],
\end{equation*} 
then \ref{hyp:E11} can be checked as in the scalar case so that Theorem~\ref{mainthm:nonconvexdiscr} may be applied as well.
	
	\section{Vanishing inertia and viscosity limit of the continuous dynamic problem}\label{sec:slowloading}

	This section is devoted to the proof of Theorems~\ref{mainthm:nonconvex}, \ref{mainthm:convex} and \ref{mainthm:unifconvex}. If not stated otherwise, from now on we assume \eqref{eq:embeddings}, as well as \eqref{mass} and \eqref{viscosity} for the operators $\mathbb{M}$ and $\mathbb{V}$, hypothesis \ref{hyp:R1} for the dissipation potential $\mc R$ and hypotheses \ref{hyp:E1}-\ref{hyp:E5} for the energy $\mc E$.
	
	For any $\eps>0$, let $u^\eps$ be a dissipative dynamic solution to \eqref{mainprob} with initial data $u^\eps_0\in D$, $u_1^\eps\in V$. In Section~\ref{subsec:unifbound} we provide a compactness result for the family $\{u^\eps\}_{\eps>0}$ by means of a suitable version of Helly's Selection Theorem, exploiting the uniform bounds of Proposition~\ref{prop:unifbounds}. Section~\ref{subsec:stabcond} deals with the passage to the limit of equation \eqref{dineq} towards the local stability condition \eqref{eq:LS} as $\eps\to 0$. We will actually employ the equivalent global $\lambda$-stability condition described in Lemma~\ref{lemma:localstability}. In Section~\ref{subsec:enbal} we then pass to the limit the energy-dissipation balance \eqref{eq:energybalance}, showing the validity of a limit energy balance in which the jump costs are described by means of an atomic measure $\mu$. We finally characterize such measure in Section~\ref{subsec:jumps}, separating the different cases of nonconvex, convex, and uniformly convex potential energy $\mc E$, and thus obtaining the results stated in Section~\ref{subsec:mainthms}.

\subsection{Compactness}\label{subsec:unifbound}

We begin by stating and proving a suitable version of Helly's Selection Theorem, which is compatible with the uniform bounds obtained in Proposition~\ref{prop:unifbounds}.

\begin{lemma}\label{helly}
	Let $X,Y$ be Banach spaces such that $X\stackrel{d}{\hookrightarrow} Y$ and $X$ is reflexive and separable. Then, given a sequence of functions $f_n\colon [a,b]\to X$ such that for a positive constant $C>0$ independent of $n$ it holds
	\begin{equation*}
		\sup\limits_{t\in [a,b]}\|f_n(t)\|_X\le C,\quad\text{ and }\quad V_Y(f_n;a,b)\le C,
	\end{equation*}
 	there exist a function $f\in B([a,b];X)\cap BV([a,b];Y) $ and a subsequence (not relabelled) such that
\begin{equation}\label{weakconvhelly}
	f_n(t)\xrightharpoonup[n\to +\infty]{X}f(t),\quad\text{ for every }t\in [a,b].
\end{equation}
\end{lemma}
\begin{proof}
 Since the dual space $X^*$ is separable, there exists a countable subset $S$ of $Y^*$ dense in $X^*$. Without loss of generality we may assume that $ S$ is a vector space over $\Q$. For all $q\in {S}$ we now define
	\begin{equation*}
		F^q_n(t):=\langle q,f_n(t)\rangle_X=\langle q,f_n(t)\rangle_Y.
	\end{equation*}
	It is immediate to check that
	\begin{align*}
		|F^q_n(t)|&\le\|q\|_{X^*}\|f_n(t)\|_X\le C \|q\|_{X^*},\\
		V_{\mathbb{R}}(F^q_n;a,b)&\le \|q\|_{Y^*}V_Y(f_n;a,b)\le C \|q\|_{Y^*}.
	\end{align*}
By using the classical Helly's Selection Theorem together with a diagonal argument we may extract a subsequence such that for every $q\in {S}$ there holds
\begin{equation*}
	F_n^q(t)\xrightarrow[n\to +\infty]{} F^q(t), \quad \text{ for every }t\in [a,b].
\end{equation*}
It is now easy to show that, for a fixed $t\in [a,b]$, the map ${S}\ni q\mapsto F^q(t)$ is linear (over $\Q$) and satisfies $|F^q(t)|\le C \|q\|_{X^*}$. By density of ${S}$ in $X^*$, it can thus be extended to the whole $X^*$ as a linear and continuous functional. Since $X$ is reflexive, then there exists an element $f(t)\in X$ such that $F^q(t)=\langle q,f(t)\rangle_X$ for all $q\in X^*$.

Let us now prove \eqref{weakconvhelly}. Fix $q\in X^*$ and for $\delta>0$ let $q_\delta\in {S}$ such that $\|q-q_\delta\|_{X^*}\le \delta$. Then one has
\begin{align*}
	&\limsup_{n\to +\infty}|\langle q,f_n(t)\rangle_X-\langle q,f(t)\rangle_X|\le \limsup_{n\to +\infty}|\langle q-q_\delta,f_n(t)-f(t)\rangle_X+\langle q_\delta,f_n(t)-f(t)\rangle_X|\\
	&\le\limsup_{n\to +\infty}\Big(\|q-q_\delta\|_{X^*}(\|f_n(t)\|_X+\|f(t)\|_X)+F_n^{q_\delta}(t)-F^{q_\delta}(t)\Big)\le \delta(C+\|f(t)\|_X),
\end{align*}
and by the arbitrariness of $\delta$ one concludes.

The fact that $f$ belongs to $ B([a,b];X)\cap BV([a,b];Y)$ is now a byproduct of \eqref{weakconvhelly} due to the weak lower semicontinuity of the norm and of the total variation.
\end{proof}

By applying Lemma~\ref{helly} with $X=U$ and $Y=Z$, and exploiting $(ii)$, $(iv)$, and $(v)$ of Proposition~\ref{prop:unifbounds}, we now deduce the following compactness result.

\begin{prop}\label{prop:compact}
	Assume $\eps u_1^\eps$ is uniformly bounded in $W$ and $\mc E(0,u_0^\eps)$ is uniformly bounded. Then there exists a function $u\colon [0,T]\to D$ such that $u\in B([0,T];U)\cap BV([0,T];Z)$ and there exists a subsequence $\eps_j\to 0$ such that:
	\begin{itemize}
		\item[(a)] $\xepsj(t)\xrightharpoonup[j\to +\infty]{U}u(t),\quad\text{ for every }t\in [0,T]$;
			\item[(b)] $\epsj\xepsjd(t)\xrightarrow[j\to +\infty]{V}0,\quad\text{ for almost every }t\in [0,T]$.
	\end{itemize}
\end{prop}
\begin{rmk}
	Notice that assumption \eqref{eq:embeddings} implies that $(a)$ also holds as a weak limit in $V$ and as a strong limit in $W$ and $Z$. Analogously, $(b)$ holds as a strong limit in $W$ and $Z$.
\end{rmk}

Since the limit function $u$ belongs to $B([0,T];U)\cap BV([0,T];Z)$, by arguing as in Remark~\ref{rmk:upm} we infer the existence of the right and left limits $u^\pm(t)$ for all $t\in [0,T]$ in the sense of \eqref{limitpm}.
	
	\subsection{Limit passage in the stability condition}\label{subsec:stabcond}	
	Next proposition shows how equation \eqref{dineq} is related to an integral form of the global $\lambda$-stability condition perturbed by a factor $\sqrt{\eps}$.	
	
	\begin{prop}\label{prop:estimate}
			Assume $\eps u_1^\eps$ is uniformly bounded in $W$ and $\mc E(0,u_0^\eps)$ is uniformly bounded. Then for any $g\colon[0,T]\to U$ of the form
			\begin{equation}\label{simplefunction}
				g(t)=\sum_{i=1}^{N}u_i\bm 1_{(s_i,t_i)}(t)+\bar u\bm 1_{[0,T]\setminus\bigcup_{i=1}^N(s_i,t_i)}(t),
			\end{equation}
			where $N\in \N$, $u_i, \bar u\in U$ and $(s_i,t_i)$ are disjoint open intervals contained in $[0,T]$, there exists a positive costant $C_g>0$ such that the following inequality holds true for every $0\le s\le t\le T$:
			\begin{equation}\label{eq:estimate}
				\int_{s}^{t}\mc E(r,u^\eps(r))\d r\le \int_{s}^{t}\Big(\mc E(r,g(r))+\mc R(g(r)-u^\eps(r))+\frac \lambda 2 |g(r)-u^\eps(r)|_W^2\Big)\d r+C_g\sqrt{\eps}.
			\end{equation}
	\end{prop}
\begin{proof}
	Without loss of generality we may assume $\bar u, u_i\in D$, otherwise the inequality is trivial.
	
	By testing equation \eqref{dineq} with $u^\eps(t)-g(t)\in U$ and exploiting \eqref{lambdaconvexineq} since $\eta^\eps(t)\in \partial^Z\mc R(0)$ and $\xi^\eps(t)\in \partial^U\mc E(t,u^\eps(t))$, for almost every $t\in [0,T]$ we obtain
	\begin{align*}
		0\ge&\, \eps^2\langle\mathbb M\ddot{u}^\eps(t), u^\eps(t)-g(t)\rangle_U+\eps\langle\mathbb V\dot{u}^\eps(t), u^\eps(t)-g(t)\rangle_V-\mc R(g(t)-u^\eps(t))\\
		&+\mc E(t,u^\eps(t))-\mc E(t,g(t))-\frac \lambda 2|g(t)-u^\eps(t)|^2_W.
	\end{align*}
By integrating the above inequality in $[s,t]$ we thus infer
 \begin{align*}
 	\int_{s}^{t}\mc E(r,u^\eps(r))\d r\le&\int_{s}^{t}\Big(\mc E(r,g(r))+\mc R(g(r)-u^\eps(r))+\frac \lambda 2 |g(r)-u^\eps(r)|_W^2\Big)\d r\\
 	&+\underbrace{\eps^2\left|\int_{s}^{t}\langle\mathbb M\ddot{u}^\eps(r), u^\eps(r)-g(r)\rangle_U\d r\right|}_{=:I^\eps}+\underbrace{\eps\left|\int_{s}^{t}\langle\mathbb V\dot{u}^\eps(r), u^\eps(r)-g(r)\rangle_V\d r\right|}_{=:II^\eps}.
 \end{align*}
We conclude if we show that $I^\eps+II^\eps\le C\sqrt{\eps}$.

We start estimating the second term by means of $(ii)$ and $(v)$ in Proposition~\ref{prop:unifbounds}:

\begin{align*}
	II^\eps &\le C\eps\int_{s}^{t}|\dot{u}^\eps(r)|_\mathbb V(\|u^\eps(r)\|_U+\|g(r)\|_U)\d r \\
	&\le C\eps(1+\|g\|_{L^\infty(0,T;U)})\left(\int_{s}^{t}|\dot{u}^\eps(r)|^2_\mathbb V\d r\right)^\frac 12\le C_g\sqrt{\eps}.
\end{align*}

As regards $I^\eps$, we first rewrite it exploiting the explicit form of $g$ and integrating by parts in time:

\begin{align*}
	I^\eps&=\eps^2\left|\int_{s}^{t}\langle\mathbb M\ddot{u}^\eps(r), u^\eps(r)-\bar u\rangle_U\d r-\sum_{i=1}^N\int_{s_i}^{t_i}\langle\mathbb M\ddot{u}^\eps(r), u_i-\bar u\rangle_U\d r\right|\\
	&=\eps^2\Bigg|\langle\mathbb M\dot{u}^\eps(t), u^\eps(t)-\bar u\rangle_W-\langle\mathbb M\dot{u}^\eps(s), u^\eps(s)-\bar u\rangle_W-\int_{s}^{t}\langle\mathbb M\dot{u}^\eps(r), \dot u^\eps(r)\rangle_W\d r\\
	&\qquad\qquad\qquad-\sum_{i=1}^N\langle\mathbb M(\dot{u}^\eps(t_i)-\dot{u}^\eps(s_i)), u_i-\bar u\rangle_W\Bigg|.
\end{align*}
By using $(ii)$ and $(iii)$ in Proposition~\ref{prop:unifbounds}, the first term (and analogously the second one) can now be bounded by
\begin{equation*}
	C\eps^2|\dot{u}^\eps(t)|_\mathbb M(\|u^\eps(t)\|_U+\|\bar u\|_U)\le C\eps(1+\|g\|_{L^\infty(0,T;U)})=C_g\eps.
\end{equation*}
In the same way we also bound the last term. We finally notice that the third term can be estimated by exploiting $(v)$ in Proposition~\ref{prop:unifbounds}. Indeed we have
\begin{align*}
	\eps^2\left|\int_{s}^{t}\langle\mathbb M\dot{u}^\eps(r), \dot u^\eps(r)\rangle_W\d r\right|\le C\eps^2\int_{s}^{t}| \dot u^\eps(r)|_\mathbb{V}^2\d r\le C\eps,
\end{align*}
and so we conclude.
\end{proof}

By sending $\eps\to 0$ in \eqref{eq:estimate}, in Proposition~\ref{prop:lambdastab} we obtain the validity of the global $\lambda$-stability for the limit function $u$. In particular, Lemma~\ref{lemma:localstability} yields that $u$ satisfies \eqref{eq:LS}.

To obtain also the full energy convergence \eqref{energyconvergence} below, we will employ the following approximation result in Bochner spaces. We refer to \cite[Lemma~A0]{Brez} for a proof.
\begin{lemma}\label{lemma:approximation}
	Let $X$ be a Banach space, and let $F\subseteq X$ be a closed subset. Let $f\in L^1(a,b;X)$ be such that $f(t)\in F$ for almost every $t\in (a,b)$. Then for all $\delta>0$ there exists a simple function $g\colon(a,b)\to X$ of the form
	\begin{equation*}
		g(t)=\sum_{i=1}^{N}x_i\bm 1_{(s_i,t_i)}(t)+\bar x\bm 1_{(a,b)\setminus\bigcup_{i=1}^N(s_i,t_i)}(t),
	\end{equation*}
	where $N\in \N$, $x_i, \bar x\in F$ and $(s_i,t_i)$ are disjoint open intervals contained in $(a,b)$, such that 
	\begin{equation*}
		\|f-g\|_{L^1(a,b;X)}\le \delta.
	\end{equation*}
\end{lemma}

\begin{prop}\label{prop:lambdastab}
	Assume $\eps u_1^\eps$ is uniformly bounded in $W$ and $\mc E(0,u_0^\eps)$ is uniformly bounded. Then the limit function $u$ obtained in Proposition~\ref{prop:compact} satisfies the global $\lambda$-stability
	\begin{itemize}
	\item[($\lambda$GS)] $\mc E(t,u(t))\le \mc E(t,x)+\mc R(x-u(t))+\frac \lambda 2|x-u(t)|_W^2$ for all $x\in U$ and for a.e. $t\in [0,T]$.
\end{itemize}
In particular the same holds true for the right and left limits $u^\pm(t)$:
	\begin{itemize}
	\item[($\lambda$GS$^+$)] $\mc E(t,u^+(t))\le \mc E(t,x)+\mc R(x-u^+(t))+\frac \lambda 2|x-u^+(t)|_W^2$ for all $x\in U$ and for all $t\in [0,T]$;
	\item[($\lambda$GS$^-$)] $\mc E(t,u^-(t))\le \mc E(t,x)+\mc R(x-u^-(t))+\frac \lambda 2|x-u^-(t)|_W^2$ for all $x\in U$ and for all $t\in (0,T]$.
\end{itemize}
Hence, one can write
\begin{equation}\label{globalstability}
	\mc E(t,u(t))\le \mc E(t,x)+\mc R(x-u(t))+\frac \lambda 2|x-u(t)|_W^2,\quad\text{ for all $x\in U$ and for all $t\in [0,T]\setminus J_u$}.
\end{equation}
Moreover it holds
\begin{equation}\label{energyliminf}
	\liminf\limits_{j\to +\infty}\mc E(t,u^\epsj(t))=\mc E(t,u(t)),\quad\text{ for a.e. }t\in [0,T].
\end{equation}
Finally, assuming in addition \ref{hyp:E7}, up to possibly extracting a further subsequence there holds
\begin{equation}\label{energyconvergence}
	\lim\limits_{j\to +\infty}\mc E(t,u^\epsj(t))=\mc E(t,u(t)),\quad\text{ for a.e. }t\in [0,T].
\end{equation}
In particular, if also \ref{hyp:E8} is in force, one has
\begin{equation}\label{eq:strongconvergence}
	\xepsj(t)\xrightarrow[j\to +\infty]{U}u(t),\quad\text{ for a.e. }t\in [0,T].
\end{equation}
\end{prop}
\begin{proof}
	Without loss of generality let us fix $x\in D$. By choosing $g(t)=x$ in Proposition~\ref{prop:estimate}, from \eqref{eq:estimate} we obtain
	\begin{equation*}
		\int_{s}^{t}\mc E(r,u^\eps(r))\d r\le \int_{s}^{t}\Big(\mc E(r,x)+\mc R(x-u^\eps(r))+\frac \lambda 2 |x-u^\eps(r)|_W^2\Big)\d r+C\sqrt{\eps}.
	\end{equation*}
By exploiting weak lower semicontinuity of $\mc E$ together with Fatou's Lemma on the left-hand side, and using Lebesgue Dominated Convergence Theorem on the right-hand side (we recall that from \eqref{eq:embeddings} we infer strong convergence of $u^\epsj(t)$ in $W$) we deduce
\begin{equation}\label{eq:intstab}
	\int_{s}^{t}\mc E(r,u(r))\d r\le \int_{s}^{t}\liminf_{j\to +\infty}\mc E(r,u^{\eps_j}(r))\d r\le \int_{s}^{t}\Big(\mc E(r,x)+\mc R(x-u(r))+\frac \lambda 2 |x-u(r)|_W^2\Big)\d r.
\end{equation}
From the arbitrariness of $s$ and $t$, the above inequality yields the validity of $(\lambda GS)$ for all $\bar t\in [0,T]\setminus J_u$ such that $\bar t$ is a Lebesgue point of the map $r\mapsto \mc E(r,u(r))$, and thus for almost every $t\in [0,T]$. The validity of $(\lambda GS^\pm)$ and \eqref{globalstability} now directly follows.

By arguing in the same way, \eqref{eq:intstab} also yields for all $x\in U$ and almost every $t\in [0,T]$
\begin{equation*}
	\mc E(t,u(t))\le \liminf_{j\to +\infty}\mc E(t,u^{\epsj}(t))\le \mc E(t,x)+\mc R(x-u(t))+\frac \lambda 2|x-u(t)|_W^2.
\end{equation*}
By choosing $x=u(t)$ one obtains \eqref{energyliminf}.

We now prove \eqref{energyconvergence}, assuming in addition \ref{hyp:E7}. Due to $(i)$ in Proposition~\ref{prop:unifbounds}, and recalling  \eqref{eq:ineqgronwall}, we obtain that $u(t)$ belongs to the closed set $F:=\{\mc E(0,\cdot)\le \overline C\}$ for all $t\in [0,T]$. Now, Lemma~\ref{lemma:approximation} gives us a sequence of simple functions $u_n\colon [0,T]\to F$ of the form \eqref{simplefunction} such that $u_n(t)\xrightarrow[n\to +\infty]{U} u(t)$ for almost every $t\in [0,T]$. Proposition~\ref{prop:estimate} yields
\begin{equation*}
	\int_{0}^{T}\mc E(r,u^\eps(r))\d r\le \int_{0}^{T}\Big(\mc E(r,u_n(r))+\mc R(u_n(r)-u^\eps(r))+\frac \lambda 2 |u_n(r)-u^\eps(r)|_W^2\Big)\d r+C_n\sqrt{\eps}.
\end{equation*}
By passing to the limit along the subsequence $\eps_j$, arguing similarly as before we deduce
\begin{align*}
	\int_{0}^{T}\mc E(r,u(r))\d r&\le\liminf\limits_{j\to +\infty}\int_{0}^{T}\mc E(r,u^\epsj(r))\d r\le\limsup\limits_{j\to +\infty}\int_{0}^{T}\mc E(r,u^\epsj(r))\d r\\
	&\le \int_{0}^{T}\Big(\mc E(r,u_n(r))+\mc R(u_n(r)-u(r))+\frac \lambda 2 |u_n(r)-u(r)|_W^2\Big)\d r.
\end{align*}
By exploiting the fact that $u_n$ is valued in $F$, using again \eqref{eq:ineqgronwall} and due to \eqref{eq:strongcont} one can apply Lebesgue Dominated Convergence Theorem to the right-hand side above sending $n\to +\infty$, thus deducing
\begin{equation}\label{eq:integralconvergence}
	\lim\limits_{j\to +\infty}\int_{0}^{T}\mc E(r,u^\epsj(r))\d r=\int_{0}^{T}\mc E(r,u(r))\d r.
\end{equation}
Since by weak lower semicontinuity there holds $\mc E(r,u(r))\le \liminf\limits_{j\to +\infty}\mc E(r,u^\epsj(r))$ everywhere in $[0,T]$, equality \eqref{eq:integralconvergence} yields that $\mc E(\cdot,u^\eps_j(\cdot))$ converges to $\mc E(\cdot,u(\cdot))$ as $j\to +\infty$ in the sense of $L^1(0,T)$. As a byproduct we finally deduce \eqref{energyconvergence}, whence \eqref{eq:strongconvergence} if also \ref{hyp:E8} is in force, and we conclude.
\end{proof}
	
	We conclude this section by stating and proving two useful lemmas regarding further properties of the right and left limits $u^\pm$.
	
		\begin{lemma}\label{lemma:energy+-}
		Assume $\eps u_1^\eps$ is uniformly bounded in $W$ and $\mc E(0,u_0^\eps)$ is uniformly bounded, and let $u$ be the limit function obtained in Proposition~\ref{prop:compact}. Then there holds
		\begin{equation*}
			\lim\limits_{\substack{r\to t^\pm\\ r\notin J_u}}\mc E(r,u(r))= \mc E(t,u^\pm (t)),\qquad\text{for every }t\in [0,T].
		\end{equation*}
	\end{lemma}
	\begin{proof}
		Fix $t\in [0,T]$. By using \eqref{globalstability} we know that
		\begin{equation*}
			\mc E(r,u(r))\le \mc E(r, u^\pm(t))+\mc R(u^\pm(t)-u(r))+\frac \lambda 2|u^\pm(t)-u(r)|^2_W,\quad\text{ for all }r\in [0,T]\setminus J_u.
		\end{equation*}
		Hence, by recalling \eqref{limitzeta} and \eqref{limitu}, by semicontinuity we obtain
		\begin{align*}
			\mc E(t,u^\pm (t))&\le \liminf\limits_{\substack{r\to t^\pm\\ r\notin J_u}}\mc E(r,u(r))\le \limsup\limits_{\substack{r\to t^\pm\\ r\notin J_u}}\mc E(r,u(r))\\
			&\le \lim\limits_{\substack{r\to t^\pm\\ r\notin J_u}}\left(\mc E(r, u^\pm(t))+\mc R(u^\pm(t)-u(r))+\frac \lambda 2|u^\pm(t)-u(r)|^2_W\right)=\mc E(t,u^\pm (t)),
		\end{align*}
		and we conclude.
	\end{proof}
	
	\begin{lemma}\label{lemma:limz*}
		Assume in addition \ref{hyp:E9} and \ref{hyp:E10}. Assume $\eps u_1^\eps$ is uniformly bounded in $W$ and $\mc E(0,u_0^\eps)$ is uniformly bounded, and let $u$ be the limit function obtained in Proposition~\ref{prop:compact}. Then for all $t\in (0,T)$ there hold
		\begin{equation}\label{limitzeta*}
			u^+(s)\xrightarrow[s\to t^\pm]{Z^*}u^\pm(t),\quad\text{ and }\quad u^-(s)\xrightarrow[s\to t^\pm]{Z^*}u^\pm(t).
		\end{equation}
	\end{lemma}
\begin{proof}
	By weak lower semicontinuity of $\mc E$, the bound $(i)$ in Proposition~\ref{prop:unifbounds} together with $(a)$ in Proposition~\ref{prop:compact} and \eqref{limitu} yields
	\begin{equation*}
		\sup\limits_{t\in (0,T)}\mc E(t, u^\pm(t))\le C.
	\end{equation*}
Furthermore, by exploiting $(\lambda GS^\pm)$ and Lemma~\ref{lemma:localstability}, for all $t\in (0,T)$ we obtain
\begin{equation*}
	\min\limits_{\xi\in\partial^U\mc E(t,u^\pm(t))\cap Z^*}\|\xi\|_{Z^*}\le \sup\limits_{\eta\in\partial^Z\mc R(0)}\|\eta\|_{Z^*}\le \rho_2,
\end{equation*}
where in the last inequality we used \eqref{boundedness}.

This means that the families $\{u^+(t)\}_{t\in (0,T)}$ and $\{u^-(t)\}_{t\in (0,T)}$ are contained in the set
\begin{equation*}
	\bigcup_{t\in (0,T)}\left\{u\in U:\, \mc E(t,u)+\min\limits_{\xi\in\partial^U\mc E(t,u)\cap Z^*}\|\xi\|_{Z^*}\le C\right\},
\end{equation*}
which by Remark~\ref{rmk:precompact} is precompact in $Z^*$. We now conclude since the possible limit points are indentified by means of \eqref{limitvw} (recall $W=W^*$).
\end{proof}
	
	\subsection{Limit passage in the energy-dissipation balance}\label{subsec:enbal}	
	We now focus on the asymptotic behaviour of the energy-dissipation balance \eqref{eq:energybalance}  (actually of the inequality \eqref{energyineq}) as $\eps\to 0$.

	\begin{prop}\label{prop:ineqpm}
		 Assume $\eps u_1^\eps$ is uniformly bounded in $W$ and $\mc E(0,u_0^\eps)$ is uniformly bounded. Then the limit function $u$ obtained in Proposition~\ref{prop:compact} satisfies the following inequalities:
		\begin{subequations}\label{eq:ineqpm}
			\begin{equation}\label{ineq+}
				\mc E(t,u^+(t))+V_\mc R(u;s+,t+)\le \mc E(s,u^+(s))+\int_{s}^{t}\partial_t\mc E(r,u(r))\d r,\quad\text{for all }0\le s\le t\le T;
			\end{equation}
			\begin{equation}\label{ineq-}
			\mc E(t,u^+(t))+V_\mc R(u;s-,t+)\le \mc E(s,u^-(s))+\int_{s}^{t}\partial_t\mc E(r,u(r))\d r,\quad\text{for all }0<s\le t\le T.
		\end{equation}
		\end{subequations}
	If in addition $\eps u_1^\eps\xrightarrow[\eps\to 0]{W}0$ and $\mc E(0,u_0^\eps)\xrightarrow[\eps\to 0]{}\mc E(0,u_0)$, then \eqref{ineq-} holds true also in $s=0$.
	\end{prop}
	\begin{proof}
		Let us fix $t\in [0,T]$ and pick $s\ge t$ such that \eqref{energyineq}, \eqref{energyliminf}, and $(b)$ in Proposition~\ref{prop:compact} hold true. By recalling that $\mc R_\eps(v)+\mc R_\eps^*(\zeta)\ge \mc R(v)$, from \eqref{energyineq} we thus know
		\begin{equation*}
			\mc E(t,\xepsj(t))+\int_s^t {\mc R}(\xepsjd(r))\d r \le \frac{\varepsilon_j^2}{2}|\xepsjd(s)|_{\mathbb{M}}^2 +  \mc E(s,\xepsj(s)) +  \int_s^t \partial_t\mc E(r,\xepsj(r))\,\mathrm{d}r.
		\end{equation*}
	Letting $j\to +\infty$, by exploiting the weak lower semicontinuity of $\mc E$ and of the $\mc R$-variation, and by using \eqref{energyliminf}, $(b)$ in Proposition~\ref{prop:compact} and assumption \ref{hyp:E4} we obtain
	\begin{equation}\label{eq:claim}
			\mc E(t,u(t))+V_\mc R(u;s,t)\le \mc E(s,u(s))+\int_{s}^{t}\partial_t\mc E(r,u(r))\d r,\quad \text{for all $t\in [0,T]$ and for a.e. $s\in [0,t]$.}
	\end{equation}
Observe that the choice $s=0$ can be done if $\eps u_1^\eps\xrightarrow[\eps\to 0]{W}0$ and $\mc E(0,u_0^\eps)\xrightarrow[\eps\to 0]{}\mc E(0,u_0)$.

In order to prove \eqref{ineq+} (and similarly \eqref{ineq-}) we now consider sequences $t_k\to t^+$, $s_k\to s^+$ such that $s_k\le t_k$, $s_k,t_k\notin J_u$, and \eqref{eq:claim} holds. By means of Lemma~\ref{lemma:energy+-}, we now conclude by letting $k\to +\infty$.
	\end{proof}
	
	As a simple corollary, we deduce the existence of a so-called \emph{defect measure} filling the gap in the inequalities \eqref{eq:ineqpm}.
	
	\begin{cor}
		Assume $\eps u_1^\eps\xrightarrow[\eps\to 0]{W}0$ and $\mc E(0,u_0^\eps)\xrightarrow[\eps\to 0]{}\mc E(0,u_0)$. Then there exists a positive Radon measure $\mu$ on $[0,T]$ such that
		\begin{equation}\label{eneq1}
			\mc E(t,u^+(t))+\mu([s,t])= \mc E(s,u^-(s))+\int_{s}^{t}\partial_t\mc E(r,u(r))\d r,\quad \text{for all $0\le s\le t\le T$.}
		\end{equation}
	\end{cor}
\begin{proof}
	From \eqref{ineq+} it follows that the function $f(t):=	\mc E(t,u^+(t))-\int_{0}^{t}\partial_t\mc E(r,u(r))\d r$ is nonincreasing on $[0,T]$, and so (see for instance \cite{AFP}) there exists a positive Radon measure $\mu$ such that
	\begin{equation*}
		f^+(t)+\mu([s,t])=f^-(s), \quad \text{for all $0\le s\le t\le T$.}
	\end{equation*}
By using Lemma~\ref{lemma:energy+-} we deduce $f^\pm(t)=\mc E(t,u^\pm(t))-\int_{0}^{t}\partial_t\mc E(r,u(r))\d r$ and we conclude.
\end{proof}

Next proposition is a crucial step towards the characterization of $u$ as an IBV solution to the rate-independent system \eqref{riprob}. We indeed show that the diffuse part of the just obtained defect measure $\mu$ coincides with the $\mc R$-variation of the continuous part of $u$.

\begin{prop}\label{prop:prop6.11}
	In the nonconvex case $\lambda>0$ assume \ref{hyp:E9} and \ref{hyp:E10}. Assume $\eps u_1^\eps\xrightarrow[\eps\to 0]{W}0$ and $\mc E(0,u_0^\eps)\xrightarrow[\eps\to 0]{}\mc E(0,u_0)$, and let $u$ be the limit function obtained in Proposition~\ref{prop:compact}. Then for every $0\le s\le t\le T$ there holds
	\begin{equation*}
		\mc E(t,u^+(t))+V_\mc R(u_{\rm co};s,t)+\sum_{r\in J_u^{\rm e}\cap [s,t]}\mu(\{r\})= \mc E(s,u^-(s))+\int_{s}^{t}\partial_t\mc E(r,u(r))\d r,
	\end{equation*}
where $u_{\rm co}$ has been introduced in \eqref{eq:decompositionf}, and $J_u^{\rm e}$ in \eqref{essjump}.
\end{prop}

\begin{proof}
	We split $\mu$ in the sum of its diffuse part $\mu_{\rm d}$ and its atomic part $\mu_{J}$, and we first observe that $\mu_{J}$ charges exactly the essential jump set $J^{\rm e}_u$ of $u$. Indeed, if $\mu_{J}(\{t\})>0$, then \eqref{eneq1} implies that $\mc E(t,u^-(t))>\mc E(t,u^+(t))$, hence $t$ belongs to $J^{\rm e}_u$. On the other hand, by using \eqref{ineq-} we deduce that for $t\in J^{\rm e}_u$ there holds $\mu_{J}(\{t\})=\mc E(t,u^-(t))-\mc E(t,u^+(t))\ge \mc R(u^+(t)-u^-(t))>0$.
	
	Thus, \eqref{eneq1} yields
	\begin{equation*}
		\mc E(t,u^+(t))+\mu_{\rm d}([s,t])+\sum_{r\in J_u^{\rm e}\cap [s,t]}\mu(\{r\})= \mc E(s,u^-(s))+\int_{s}^{t}\partial_t\mc E(r,u(r))\d r,
	\end{equation*}
and we just need to prove that $\mu_{\rm d}([s,t])=V_\mc R(u_{\rm co};s,t)$ for all $0\le s\le t\le T$ to conclude.

We set $v(t):=V_\mc R(u;0,t)$ and we denote by $v'$ its distributional derivative, which is a positive Radon measure by monotonicity. So there holds
\begin{equation*}
	v'([s,t])=V_\mc R(u,s-,t+),\quad\text{for all }0\le s\le t\le T.
\end{equation*}
Moreover, recalling \eqref{vcovj}, one has
\begin{equation*}
	(v')_{\rm d}=(v_{\rm co})',\quad\text{ and }\quad (v')_J=(v_J)'.
\end{equation*}
By using \eqref{eneq1} and \eqref{ineq-}, for every $0\le s\le t\le T$ we now obtain
\begin{align*}
	\mu([s,t])=\mc E(s,u^-(s))-\mc E(t,u^+(t))+\int_{s}^{t}\partial_t\mc E(r,u(r))\d r\ge V_\mc R(u;s-,t+)=v'([s,t]).
\end{align*}
Since both $\mu$ and $v'$ are Radon measures, the above inequality can be extended to all Borel subsets of $[0,T]$. This yields
\begin{align*}
	\mu_{\rm d}([s,t])&=\mu([s,t]\setminus J^{\rm e}_u)\ge v'([s,t]\setminus J^{\rm e}_u)\ge (v')_{\rm d}([s,t]\setminus J^{\rm e}_u)=(v')_{\rm d}([s,t])=(v_{\rm co})'([s,t])\\
	&=v_{\rm co}(t)-v_{\rm co}(s)=V_\mc R(u_{\rm co};s,t). 
\end{align*}
In order to prove the other inequality we will show that
\begin{equation}\label{RNderivative}
	\frac{\d (v_{\rm co})'}{\d\mu_{\rm d}}\ge 1,\quad \mu_{\rm d}\text{-a.e. in }(0,T),
\end{equation}
where $\frac{\d (v_{\rm co})'}{\d\mu_{\rm d}}$ is the Radon-Nikodym derivative of the measure $(v_{\rm co})'$  with respect to $\mu_{\rm d}$.

We start by observing that $v'=(v_{\rm co})'+(v_J)'$. Furthermore, since $(v_J)'$ is purely atomic while $\mu_{\rm d}$ is a diffuse measure, there must hold
\begin{equation}\label{RNderivative2}
	\frac{\d (v_{J})'}{\d\mu_{\rm d}}(t)=\lim\limits_{h\to 0^+}\frac{(v_{J})'([t-h,t+h])}{\mu_{\rm d}([t-h,t+h])}=0,\quad \text{ for $\mu_{\d}$-a.e. }t\in (0,T).
\end{equation}
Moreover, from \eqref{eneq1}, for all $0< s\le t< T$ we deduce
\begin{align*}
	\mu_{\d}([s,t])&\le \mc E(s,u^-(s))-\mc E(t,u^+(t)) +\int_{s}^{t}\partial_t\mc E(r,u(r))\d r\\
	&= \mc E(s,u^-(s))-\mc E(s,u^+(t)) +\int_{s}^{t}(\partial_t\mc E(r,u(r))-\partial_t\mc E(r,u^+(t)))\d r,
\end{align*}
whence, by means of $(\lambda GS^-)$ and \ref{hyp:E4}, we obtain
\begin{align}\label{eq:ineqprinc}
	\mu_{\d}([s,t])\le \mc R(u^+(t)-u^-(s))+\frac\lambda 2|u^+(t)-u^-(s)|^2_W +\int_{s}^{t}\gamma_{\bar C}(r)\omega_{\bar C}(|u(r)-u^+(t)|_W)\d r.
\end{align}
We will now exploit the following three facts:
\begin{itemize}
	\item[1)] by definition of $\mc R$-variation there holds
	\begin{equation}\label{eq:vardef}
		\mc R(u^+(t)-u^-(s))\le V_\mc R(u;s-,t+)=v'([s,t]);
	\end{equation}
	\item[2)] if $\lambda>0$, Lemma~\ref{lemma:limz*} ensures that $u^\pm(r)\in Z^*$ for $r\in(0,T)$, and so by using also \eqref{Rbounds} and \eqref{eq:vardef} we have
	\begin{equation}\label{eq:ineqzeta*}
		\begin{aligned}
			|u^+(t)-u^-(s)|^2_W&=\langle u^+(t)-u^-(s),u^+(t)-u^-(s)\rangle_Z\le \|u^+(t)-u^-(s)\|_{Z^*}\|u^+(t)-u^-(s)\|_Z\\
			&\le \rho_1^{-1}\|u^+(t)-u^-(s)\|_{Z^*}\mc R(u^+(t)-u^-(s))\le \rho_1^{-1}\|u^+(t)-u^-(s)\|_{Z^*} v'([s,t]);
		\end{aligned}
	\end{equation}
\item[3)] we claim that
\begin{equation}\label{eq:claim2}
	\lim\limits_{h\to 0^+}\underbrace{\frac{1}{\mu_{\d}([t-h,t+h])}\int_{t-h}^{t+h}\gamma_{\bar C}(r)\omega_{\bar C}(|u(r)-u^+(t+h)|_W)\d r}_{=:I_h(t)}=0,\text{ for $\mu_{\d}$-a.e. }t\in [0,T].
\end{equation}	
\end{itemize}
By plugging \eqref{eq:vardef} and \eqref{eq:ineqzeta*} into \eqref{eq:ineqprinc}, for every $0<s\le t< T$ we thus have 
\begin{equation*}
	\mu_{\d}([s,t])\le (1+C\lambda\|u^+(t)-u^-(s)\|_{Z^*}) v'([s,t])+\int_{s}^{t}\gamma_{\bar C}(r)\omega_{\bar C}(|u(r)-u^+(t)|_W)\d r.
\end{equation*}
Let us now fix $t\in (0,T)\setminus J_u$ such that it satisfies \eqref{RNderivative2}, \eqref{eq:claim2} and there exists $\frac{\d (v_{\rm co})'}{\d\mu_{\rm d}}(t)$. Observe that the set of times $t$ which do not fulfil the above properties has null measure with respect to $\mu_{\d}$.
Finally we obtain:
\begin{align*}
	\frac{\d (v_{\rm co})'}{\d\mu_{\rm d}}(t)&=\lim\limits_{h\to 0^+}\frac{(v_{\rm co})'([t-h,t+h])}{\mu_{\rm d}([t-h,t+h])}=\lim\limits_{h\to 0^+}\frac{v'([t-h,t+h])}{\mu_{\rm d}([t-h,t+h])}\\
	&\ge \lim\limits_{h\to 0^+}\frac{1-I_h(t)}{1+C\lambda\|u^+(t+h)-u^-(t-h)\|_{Z^*}}=1,
\end{align*}
where we also used \eqref{limitzeta*}. Hence \eqref{RNderivative} is proved, and we conclude up to show the validity of \eqref{eq:claim2}.

To this aim, we denote by $\mu_{\rm Ca}$ the Cantor part of $\mu_{\d}$, that is $\mu_{\d}=\frac{\d \mu_{\d}}{\d \mc L^1}\mc L^1+\mu_{\rm Ca}$, where $\mc L^1$ denotes the Lebesgue measure in $[0,T]$. By Besicovitch Theorem we know that for $\mu_{\rm d}$-almost every $t\in [0,T]$ there holds either
\begin{subequations}
\begin{equation}\label{case1}
	\lim\limits_{h\to 0^+}\frac{\mu_{\rm Ca}([t-h,t+h])}{\mu_{\d}([t-h,t+h])}>0,
\end{equation} 
or
\begin{equation}\label{case2}
	\text{the limit in \eqref{case1} is $0$ and }	\lim\limits_{h\to 0^+}\frac{2h}{\mu_{\d}([t-h,t+h])}<+\infty.
\end{equation}
\end{subequations}
If \eqref{case1} is in force, then 
\begin{equation}\label{eq:limit0}
	\limsup\limits_{h\to 0^+}I_h(t)\le\lim\limits_{h\to 0^+}\frac{C}{\mu_{\rm Ca}([t-h,t+h])}\int_{t-h}^{t+h}\gamma_{\bar C}(r)\d r=0,\quad\text{ for $\mu_{\rm Ca}$-a.e. t satisfying \eqref{case1},}
\end{equation}
where the limit is $0$ since $\mu_{\rm Ca}$ is singular with respect to $\mc L^1$. However, since we are considering points with positive density $\frac{\d\mu_{\rm Ca}}{\d \mu_{\d}}$, the validity of \eqref{eq:limit0} can be extended for $\mu_{\rm d}$-almost every $t$ satisfying \eqref{case1}. 

If instead \eqref{case2} is in force, by Lebesgue Differentiation Theorem for almost every t satisfying \eqref{case2} we infer
\begin{equation}\label{eq:limit02}
	\lim\limits_{h\to 0^+}I_h(t)=\lim\limits_{h\to 0^+}\frac{2h}{\mu_{\d}([t-h,t+h])}\frac {1}{2h}\int_{t-h}^{t+h}\gamma_{\bar C}(r)\omega_{\bar C}(|u(r)-u^+(t+h)|_W)\d r=0,
\end{equation}
and since we are now considering points with null density $\frac{\d\mu_{\rm Ca}}{\d \mu_{\d}}$ one obtains that actually \eqref{eq:limit02} holds for $\mu_{\rm d}$-almost every $t$ satisfying \eqref{case2}, and the proof is complete.
\end{proof}

\subsection{Characterization of the defect measure at jump points}\label{subsec:jumps}
In this section we finally conclude the proof of Theorems~\ref{mainthm:nonconvex}, \ref{mainthm:convex} by providing an explicit description of the value $\mu(\{t\})$ at (essential) jump points of $u$.

In the nonconvex situation $\lambda>0$ we show that this value coincides with the viscoinertial cost introduced in \eqref{eq:cost}, and thus we retrieve the notion of IBV solution to \eqref{riprob}. The convex case $\lambda=0$ is much simpler, and the jump value $\mu(\{t\})$ turns out to be equal to the rate-independent cost $\mc R(u^+(t)-u^-(t))$ as for (essential) energetic solutions. In the subcase of uniformly convex energies we also obtain suitable time-regularity properties of $u$, characterizing it as a classic solution to \eqref{riprob} and hence proving Theorem~\ref{mainthm:unifconvex}.

\subsubsection{\underline{Nonconvex case $\lambda>0$}}

We begin by showing that the viscoinertial cost defined in \eqref{eq:cost} provides a lower bound for the energy lost at jump points. Hence, recalling Proposition~\ref{prop:uppercost}, the limit function $u$ is an IBV solution.

\begin{prop}\label{prop:mu1}
	Assume also \ref{hyp:E6}-\ref{hyp:E10}. Assume $\eps u_1^\eps\xrightarrow[\eps\to 0]{V}0$ and $u_0^\eps\xrightarrow[\eps\to 0]{U}u_0\in D$. Then
	\begin{equation}\label{eq:mu1}
		\mu(\{t\})\ge c^{\M,\V}(t;u^-(t),u^+(t)),\quad\text{for all }t\in J_u^{\rm e}.
	\end{equation}
\end{prop}
\begin{proof}
	Let us fix $t\in  J_u^{\rm e}$ and $\alpha,\beta>0$. We will prove that
	\begin{equation}\label{eq:mualphabeta}
		\mu(\{t\})\ge c^\alpha_\beta(t;u^-(t),u^+(t)),
	\end{equation}
	whence \eqref{eq:mu1} follows by the arbitrariness of $\alpha$ and $\beta$.
	
	By exploiting Proposition~\ref{prop:compact}, \eqref{limitu}, \eqref{energyconvergence} and Lemma~\ref{lemma:energy+-}, with the argument of \cite[Proposition~5.8]{ScilSol} it is possible to construct two sequences $t_j^+\to t^+$, $t_j^-\to t^-$ ($t_j^-\equiv0$ if $t=0$), and to possibly extract a further subsequence (not relabelled) $\eps_j\to 0$ such that
	\begin{enumerate}
		\item  $t_j^+$ and $t_j^-$ belongs to $[0,T]\setminus N$, where $N$ is the set introduced in Definition~\ref{def:dynsol};
		\item 	$\xepsj(t^\pm_j)\xrightharpoonup[j\to +\infty]{U}u^\pm(t)$;
		\item $\epsj\xepsjd(t^\pm_j)\xrightarrow[j\to +\infty]{V}0$;
		\item $\mc E(t^\pm_j, \xepsj(t^\pm_j))\xrightarrow[j\to +\infty]{}\mc E(t,u^\pm(t))$.
	\end{enumerate}
	By means of \ref{hyp:E2} we also deduce that $\mc E(t, \xepsj(t^\pm_j))\xrightarrow[j\to +\infty]{}\mc E(t,u^\pm(t))$, and so assumption \ref{hyp:E8} yields strong convergence
	\begin{enumerate}
		\item [(5)] $\xepsj(t^\pm_j)\xrightarrow[j\to +\infty]{U}u^\pm(t)$.
	\end{enumerate}
	By using \eqref{eq:energybalance}, and by definition of the viscous contact potential $p_\V$ we now obtain
	\begin{align*}
		&\quad\,\mu(\{t\})=\mc E(t,u^-(t))-\mc E(t,u^+(t))\\
		&=\lim\limits_{j\to +\infty}\left(\frac{\epsj^2}{2}|\xepsjd(t_j^-)|^2_\M{+}\mc E(t^-_j, \xepsj(t^-_j)){+}\int_{t^-_j}^{t^+_j}\partial_t\mc E(r,u^\epsj(r))\d r{-}\frac{\epsj^2}{2}|\xepsjd(t_j^+)|^2_\M{-}\mc E(t^+_j, \xepsj(t^+_j))\right)\\
		&=\lim\limits_{j\to +\infty}\int_{t^-_j}^{t^+_j}(\mc R_\epsj(\xepsjd(r))+\mc R_\epsj^*(-\epsj^2\M\xepsjdd(r)-\xi^\epsj(r)))\d r\\
		&\ge \limsup\limits_{j\to +\infty}\int_{t^-_j}^{t^+_j}p_\V(\xepsjd(r), -\epsj^2\M\xepsjdd(r)-\xi^\epsj(r))\d r.
	\end{align*}
	Performing the change of variable $r=\epsj\tau+t_j^-$ and defining
	\begin{equation*}
		\sigma_j:=\frac{t_j^+-t_j^-}{\epsj},\quad v_j(\tau):=\xepsj(\epsj\tau+t_j^-),\quad\xi_j(\tau):=\xi^\epsj(\epsj\tau+t_j^-) ,
	\end{equation*}
	by $(a)$ in Proposition~\ref{prop:visccont} we thus deduce
	\begin{equation*}
		\mu(\{t\})\ge  \limsup\limits_{j\to +\infty}\int_{0}^{\sigma_j}p_\V(\dot v_j(\tau),-\M \ddot v_j(\tau)-\xi_j(\tau))\d\tau.
	\end{equation*}
	By \eqref{eq:costalphabeta}, we conclude the proof of \eqref{eq:mualphabeta} if we show that the pair $(v_j,\xi_j)$ belongs to the set of admissible functions $AD^\alpha_\beta(t;u^-(t), u^+(t);\sigma_j)$, defined in \eqref{eq:admfunct}, for $j$ large enough. 
	
	To this aim, we first observe that $v_j$ is in $G(0,\sigma_j)$, defined in \eqref{eq:G}, since it inherits the regularity properties of $\xepsj$, and from the bound $(i)$ in Proposition~\ref{prop:unifbounds} together with \eqref{eq:ineqgronwall}. Moreover, $\xi_j$ belongs to $L^\infty(0,\sigma_j;U^*)$ by $(vi)$ in Proposition~\ref{prop:unifbounds}. 
	
	We start checking condition \eqref{eq:adm1}, which simply follows by using $(iv)$ in Proposition~\ref{prop:unifbounds}:
	\begin{equation}\label{eq:boundvdotj}
		\int_{0}^{\sigma_j}\|\dot v_j(\tau)\|_Z\d \tau=\int_{t_j^-}^{t_j^+}\|\dot u^{\epsj}(r)\|_Z\d r\le \overline{C}.
	\end{equation}
	
	Then, by exploiting assumption \ref{hyp:E9}, for almost every $\tau\in [0,\sigma_j]$ we deduce
	\begin{align*}
		\xi_j(\tau)&=\xi^\epsj(\epsj\tau+t_j^-)\in \partial^U\mc E(\epsj\tau+t_j^-, v_j(\tau))\subseteq \partial^U\mc E(t, v_j(\tau))+\overline{B}^{Z^*}_{\omega_M(|\epsj\tau+t_j^--t|)}\\
		&\subseteq \partial^U\mc E(t, v_j(\tau))+\overline{B}^{Z^*}_{\omega_M(t_j^+-t_j^-)}\subseteq \partial^U\mc E(t, v_j(\tau))+\overline{B}^{Z^*}_{\alpha},
	\end{align*}
	where the last inclusion holds for so large $j$ that $\omega_M(t_j^+-t_j^-)\le \alpha$. Hence condition \eqref{eq:adm2} is fulfilled.
	
	We now observe that $v_j(0)=u^\epsj(t_j^-)$, $v_j(\sigma_j)=u^\epsj(t_j^+)$, $\dot v_j(0)=\epsj\dot u^\epsj(t_j^-)$ and $\dot v_j(\sigma_j)=\epsj\dot u^\epsj(t_j^+)$. Thus, condition \eqref{eq:adm3} is satisfied for $j$ large enough due to properties $(3)$ and $(5)$ above.
	
	We finally note that $\M\ddot{v}_j+\xi_j$ is in $L^2(0,\sigma_j;V^*)$ since by \eqref{dineq} we have $\eps_j^2\M\xepsjdd+\xi^\epsj\in L^2(0,T;V^*)$. Moreover, by using \eqref{eq:chainrule} we deduce
	\begin{align*}
		&\quad\,\int_{0}^{\sigma_j}\langle \mathbb{M}\ddot{v}_j(\tau)+\xi_j(\tau), \dot{v}_j(\tau)\rangle_V\d \tau=\int_{t_j^-}^{t_j^+}\langle \eps_j^2\M\xepsjdd(r)+\xi^\epsj(r), \dot{u}^\epsj(r)\rangle_V\d r\\	
		&=\frac{\varepsilon_j^2}{2}|\xepsjd(t_j^+)|_{\mathbb{M}}^2+ \mc E(t_j^+,\xepsj(t_j^+))-\frac{\varepsilon_j^2}{2}|\xepsjd(t_j^-)|_{\mathbb{M}}^2- \mc E(t_j^-,\xepsj(t_j^-))-\int_{t_j^-}^{t_j^+} \partial_t\mc E(r,\xepsj(r))\,\mathrm{d}r\\
		&=\frac 12|\dot{v}_j(\sigma_j)|_\M^2+\mc E(t,v_j(\sigma_j))-\frac 12|\dot{v}_j(0)|_\M^2-\mc E(t,v_j(0))\\
		&\quad\, +\underbrace{\int_{t}^{t_j^+} \partial_t\mc E(r,\xepsj(t_j^+))\,\mathrm{d}r+\int_{t_j^-}^{t} \partial_t\mc E(r,\xepsj(t_j^-))\,\mathrm{d}r-\int_{t_j^-}^{t_j^+} \partial_t\mc E(r,\xepsj(r))\,\mathrm{d}r}_{=:R_j}.
	\end{align*} 
	We thus prove the validity of the last condition \eqref{eq:adm4} if we show that $\lim\limits_{j\to +\infty}R_j=0$. To this aim, we estimate by means of \ref{hyp:E2}, \eqref{eq:ineqgronwall} and $(i)$ in Proposition~\ref{prop:unifbounds}, obtaining (here we set $B:=\|b\|_{L^1(0,T)}$):
	\begin{align*}
		|R_j|&\le \int_{t}^{t_j^+}b(r)e^B(\mc E(t_j^+,u^\epsj(t_j^+))+1) \d r+\int_{t_j^-}^{t}b(r)e^B(\mc E(t_j^-,u^\epsj(t_j^-))+1) \d r\\
		&\quad\,+\int_{t_j^-}^{t_j^+}b(r)(\mc E(r,u^\epsj(r))+1) \d r \le 2e^B(\overline C+1)\int_{t_j^-}^{t_j^+}b(r)\d r,
	\end{align*}
and so we conclude.	
\end{proof}

Finally, in the case $V=U$, if we also assume that the proper domain $D$ of the energy $\mc E(t,\cdot)$ coincides with the whole space $U$, then we can slightly improve the characterization of the cost at jump points, showing that
\begin{equation*}
	\mu(\{t\})=c^{\M,\V}(t;u^-(t),u^+(t))=c^{\M,\V}_0(t;u^-(t),u^+(t)),\qquad\text{for all }t\in J_u^{\rm e},
\end{equation*}
where $c^{\M,\V}_0$ is defined in \eqref{eq:cost0}. This means that the competitors in the definition of the cost attains the boundary values. Since by \eqref{eq:ordcost} one has $c^{\M,\V}(t;u^-(t),u^+(t))\le c^{\M,\V}_0(t;u^-(t),u^+(t))$ we only need to prove the reverse inequality in order to show that the limit function $u$ is a precise IBV solution.

\begin{prop} \label{prop:mugec0}
	Assume $D=U=V$ and conditions \ref{hyp:E6}-\ref{hyp:E10}. Assume $\eps u_1^\eps\xrightarrow[\eps\to 0]{U}0$ and $u_0^\eps\xrightarrow[\eps\to 0]{U}u_0$. Then
	\begin{equation*}
		\mu(\{t\})\ge c^{\M,\V}_0(t;u^-(t),u^+(t)),\quad\text{for all }t\in J_u^{\rm e}.
	\end{equation*}
\end{prop}
\begin{proof}
	Fix $t\in J_u^{\rm e}$ and $\alpha>0$. By \eqref{eq:cost0} we conclude if we prove that
	\begin{equation}\label{eq:mualphazero}
		\mu(\{t\})\ge c^\alpha_0(t;u^-(t),u^+(t)).
	\end{equation}
	In the proof of Proposition~\ref{prop:mu1} we already showed that 
	\begin{equation}\label{eq:ineqmu}
		\mu(\{t\})\ge \limsup\limits_{j\to +\infty}\int_{0}^{\sigma_j}p_\V(\dot v_j(\tau),-\M \ddot v_j(\tau)-\xi_j(\tau))\d\tau, 
	\end{equation}
	with $(v_j,\xi_j)\in G(0,\sigma_j)\times L^\infty(0,\sigma_j;U^*)$ satisfying \eqref{eq:boundvdotj} and such that $\xi_j\in \partial^U\mc E(t,v_j(\cdot))+\overline{B}^{Z^*}_{\alpha}$ almost everywhere in $[0,\sigma_j] $ for $j$ large enough.
	
	We now set $\widetilde{\sigma}_j:=\sigma_j+2$ and we consider $\widetilde{v}_j\colon [0,\widetilde{\sigma}_j]\to U$ defined by
	\begin{equation*}
		\widetilde{v}_j(\tau):=\begin{cases}
			u^-(t)+g(\tau)(v_j(0)-u^-(t))+h(\tau)\dot{v}_j(0),&\text{if }\tau\in [0,1],\\
			v_j(\tau-1),&\text{if }\tau\in (1,\sigma_j+1),\\
			u^+(t)+g(\widetilde\sigma_j-\tau)(v_j(\sigma_j)-u^+(t))-h(\widetilde\sigma_j-\tau)\dot{v}_j(\sigma_j),&\text{if }\tau\in [\sigma_j+1,\widetilde\sigma_j],
		\end{cases}
	\end{equation*}
	where $g(x)=x^2(3-2x)$ and $h(x)=x^2(x-1)$.
	
	It is easy to see that $\widetilde{v}_j$ belongs to $W^{1,2}(0,\widetilde\sigma_j;U)\cap W^{2,2}(0,\widetilde\sigma_j;U^*)$, which coincides with $G(0,\widetilde\sigma_j)$ by Remark~\ref{rmk:V=U} since $D=U=V$ and \ref{hyp:E7} is in force. Furthermore, by construction there holds
	\begin{equation*}
		\widetilde{v}_j(0)=u^-(t),\quad 	\widetilde{v}_j(\widetilde\sigma_j)=u^+(t), \quad\dot{\widetilde{v}}_j(0)=\dot{\widetilde{v}}_j(\widetilde\sigma_j)=0,
	\end{equation*}
	whence \eqref{eq:adm3} is satisfied with $\beta=0$.
	
	For almost every $\tau\in [0,\widetilde{\sigma}_j]$ we also set
	\begin{equation*}
		\widetilde{\xi}_j(\tau):=\begin{cases}
			\xi_j(\tau-1),&\text{if }\tau\in (1,\sigma_j+1),\\
			\overline{\xi}_j(\tau),&\text{if }\tau\in [0,1]\cup[\sigma_j+1,\widetilde\sigma_j],
		\end{cases}
	\end{equation*}
	where $\overline{\xi}_j(\tau)$ is a measurable selection of $\partial^U\mc E(t,\widetilde{v}_j(\tau))$, which is nonempty by \ref{hyp:E7} since $D=U$. Recalling that $\widetilde{v}_j\in W^{1,2}(0,\widetilde{\sigma}_j;U)\subseteq C([0,\widetilde{\sigma}_j];U)$, by \ref{hyp:E6} and \ref{hyp:E7} we infer $\widetilde{\xi}_j\in L^\infty(0,\widetilde\sigma_j;U^*)$; moreover, by construction, we easily deduce the validity of \eqref{eq:adm2}.
	
	We now check condition \eqref{eq:adm1}, yielding by Remark~\ref{rmk:V=U} that $(\widetilde v_j,\widetilde\xi_j)$ belongs to the set of admissible functions $AD^\alpha_0(t;u^-(t),u^+(t);\widetilde\sigma_j)$ for $j$ large enough. By using \eqref{eq:boundvdotj} and the definition of $\widetilde{v}_j$ we estimate
	\begin{align*}
		&\quad\,\int_{0}^{\widetilde\sigma_j}\|\dot{\widetilde{v}}_j(\tau)\|_Z\d\tau=\int_{0}^{\sigma_j}\|\dot{{v}}_j(\tau)\|_Z\d\tau+\int_{0}^{1}\|\dot{\widetilde{v}}_j(\tau)\|_Z\d\tau+\int_{\sigma_j+1}^{\widetilde\sigma_j}\|\dot{\widetilde{v}}_j(\tau)\|_Z\d\tau\\
		&\le \overline{C}+\Big(\|v_j(0){-}u^-(t)\|_Z+\|\dot v_j(0)\|_Z+\|v_j(\sigma_j){-}u^+(t)\|_Z+\|\dot v_j(\sigma_j)\|_Z\Big)\max\limits_{x\in [0,1]}(|g'(x)|{+}|h'(x)|).
	\end{align*}
Since the term within brackets in the second line above vanishes as $j\to +\infty$, we infer the validity of \eqref{eq:adm1} for $j$ large enough.

By \eqref{eq:ineqmu} and \eqref{eq:costalphabeta}, we thus have
\begin{align*}
	&\quad\,\mu(\{t\})\ge \limsup\limits_{j\to +\infty}\left(\int_{0}^{\widetilde\sigma_j}\!\!\!p_\V(\dot{\widetilde v}_j(\tau),-\M \ddot{\widetilde v}_j(\tau)-\widetilde\xi_j(\tau))\d\tau-\int_{[0,1]\cup[\sigma_j+1,\widetilde\sigma_j]}\!\!\!\!\!\!\!\!\!\!\!\!\!\!\!\!\!\!\!\!\!p_\V(\dot{\widetilde v}_j(\tau),-\M \ddot{\widetilde v}_j(\tau)-\widetilde\xi_j(\tau))\d\tau\right)\\
	&\ge c^\alpha_0(t;u^-(t),u^+(t))-\liminf\limits_{j\to +\infty}\int_{[0,1]\cup[\sigma_j+1,\widetilde\sigma_j]}p_\V(\dot{\widetilde v}_j(\tau),-\M \ddot{\widetilde v}_j(\tau)-\widetilde\xi_j(\tau))\d\tau,
\end{align*}
and we conclude the proof of \eqref{eq:mualphazero} if we show that the second term vanishes.
	
	To this aim we first claim that there exists $C>0$ such that
	\begin{equation}\label{claimbound}
		\mc E(t,\widetilde{v}_j(\tau))\le C,\quad\text{for all $\tau\in[0,1]\cup[\sigma_j+1,\widetilde\sigma_j]$ and for all $j\in\N$.}
	\end{equation}
	By using \ref{hyp:E6}, the claim implies that $\|\overline{\xi}_j\|_{L^\infty((0,1)\cup(\sigma_j+1,\widetilde\sigma_j);U^*)}$ is uniformly bounded with respect to $j\in \N$. Now, by means of $(e)$ in Proposition~\ref{prop:visccont} we deduce
	\begin{align*}
		&\,\int_{0}^{1}p_\V(\dot{\widetilde v}_j(\tau),-\M \ddot{\widetilde v}_j(\tau)-\widetilde\xi_j(\tau))\d\tau\le C\int_{0}^{1}|\dot{\widetilde v}_j(\tau)|_U\left(1+|\M \ddot{\widetilde v}_j(\tau)+\widetilde\xi_j(\tau)|_{U^*}\right)\d r\\
		\le &\, C\int_{0}^{1}\left(|v_j(0)-u^-(t)|_U+|\dot v_j(0)|_U\right)\left[1+C\left(|v_j(0)-u^-(t)|_U+|\dot v_j(0)|_U\right)+\|\overline{\xi}_j\|_{L^\infty(0,1;U^*)}\right]\d\tau\\
		\le& \, C\left(|v_j(0)-u^-(t)|_U+|\dot v_j(0)|_U\right)\xrightarrow[j\to +\infty]{}0,
	\end{align*}
and an analogous estimate holds true for the integral over $(\sigma_j+1,\widetilde\sigma_j)$.

	We are thus left to prove the claim \eqref{claimbound}.
	
	If \eqref{claimbound} was false, then $M_j=\max\limits_{\tau\in [0,1]\cup[\sigma_j+1,\widetilde\sigma_j]}\mc E(t,\widetilde{v}_j(\tau))$ would diverge to $+\infty$ as $j\to +\infty$. Let $\tau_j$ such that $M_j=\mc E(t,\widetilde{v}_j(\tau_j))$; without loss of generality we may assume that $\tau_j\in [0,1]$ or $\tau_j\in [\sigma_j+1,\widetilde\sigma_j]$. Suppose the second case is in force, the other one being analogous. Since $v_j(\sigma_j)\xrightarrow[j\to +\infty]{U} u^+(t)$ and $\dot v_j(\sigma_j)\xrightarrow[j\to +\infty]{U}0$, it is easy to see that $\widetilde v_j(\tau_j)\xrightarrow[j\to +\infty]{U} u^+(t)$. Hence, by \eqref{eq:strongcont} there must hold
	\begin{equation*}
		\lim\limits_{j\to +\infty}M_j=	\lim\limits_{j\to +\infty}\mc E(t,\widetilde{v}_j(\tau_j))=\mc E(t,u^+(t))<+\infty,
	\end{equation*}
	and we reach a contradiction. Thus \eqref{claimbound} must be true and we conclude.
\end{proof}

\subsubsection{\underline{Convex case $\lambda=0$}}
This short subsection deals with the case of a convex potential energy $\mc E(t,\cdot)$, concluding the proof of Theorem~\ref{mainthm:convex}. We stress that from now on the additional assumptions \ref{hyp:E6}-\ref{hyp:E10} are not needed.
\begin{prop}\label{prop:costconvex}
	 Assume $\eps u_1^\eps\xrightarrow[\eps\to 0]{W}0$, $\mc E(0,u_0^\eps)\xrightarrow[\eps\to 0]{}\mc E(0,u_0)$ and assume that $u_0$ is globally stable, namely \eqref{eq:GS0} is satisfied. Then there holds
	 \begin{equation*}
	 	\mu(\{t\})=\mc R(u^+(t)-u^-(t)),\qquad\text{for all }t\in J^{\rm e}_u.
	 \end{equation*}
\end{prop}
\begin{proof}
		We first observe that \eqref{eq:GS0} yields that $(\lambda GS^-)$ holds for all $t\in [0,T]$. Since we are in the convex case $\lambda=0$, this fact implies
	\begin{equation*}
		\mu(\{t\})=\mc E(t,u^-(t))-\mc E(t,u^+(t))\le\mc R(u^+(t)-u^-(t)),\qquad\text{for all }t\in J^{\rm e}_u.
	\end{equation*}
On the other hand, from \eqref{ineq-} we deduce
\begin{equation*}
	\mu(\{t\})=\mc E(t,u^-(t))-\mc E(t,u^+(t))\ge V_\mc R(u,t-,t+)\ge\mc R(u^+(t)-u^-(t)),\qquad\text{for all }t\in J^{\rm e}_u,
\end{equation*}
and so we conclude.
\end{proof}

\subsubsection{\underline{Uniformly convex case}}

In the last part of this long section we conclude the proof of Theorem~\ref{mainthm:unifconvex}. We will make use of the following lemma, whose proof can be found in \cite[Lemma~5.5]{GidRiv}. It shows that for a uniformly convex energy the global stability condition can be enhanced.
\begin{lemma}\label{lemma:strongstability}
	Assume \ref{hyp:E5'} and fix $t\in [0,T]$. If $\bar u\in U$ satisfies $\mc E(t,\bar u)\le \mc E(t,x)+\mc R(x-\bar u)$ for all $x\in U$, then it actually satisfies the stronger inequality
	\begin{equation*}
		\mc E(t,\bar u)+\frac\mu 2\|x-\bar u\|_X^2\le \mc E(t,x)+\mc R(x-\bar u),\quad\text{for all }x\in U.
	\end{equation*}
\end{lemma}

\begin{proof}[Proof of Theorem~\ref{mainthm:unifconvex}]

	In order to prove \eqref{eq:ensol} we just need to show that $u\in C([0,T];X)$. Indeed, since $U\hookrightarrow X$, this fact combined with $u\in B([0,T];U)$ implies $u\in C_{\rm w}([0,T];U)$. In particular, the embeddings \eqref{eq:embeddings} also ensure that $u$ is in $C_{\rm w}([0,T];V)\cap C([0,T];W)$. Since Theorem~\ref{mainthm:convex} applies under the current assumptions, the validity of \eqref{eq:ensol} now directly follows from \eqref{eq:essensol}.
	
	In order to prove that $u$ belongs to $C([0,T];X)$, let us fix $t\in [0,T]$ and take $s\in [0,t]$ such that
	\begin{itemize}
		\item $\mc E(s,u(s))\le\mc E(s,x)+\mc R(x-u(s))$ for all $x\in U$;
		\item $\mc E(t,u(t))+V_\mc R(u;s,t)\le \mc E(s,u(s))+\int_{s}^{t}\partial_t\mc E(r,u(r))\d r$.
	\end{itemize}
Notice that $s=0$ satisfies the above two conditions; moreover, from $(\lambda GS)$ and \eqref{eq:claim} the set of times which do not fulfil both of them has measure zero.

By using Lemma~\ref{lemma:strongstability} we thus obtain
\begin{align}\label{boundcont}
	\frac\mu 2\|u(t)-u(s)\|_X^2&\le \mc E(s,u(t))-\mc E(s,u(s))+\mc R(u(t)-u(s))\nonumber\\
	&\le \mc E(s,u(t))-\mc E(t,u(t))+\mc E(t,u(t))-\mc E(s,u(s))+V_\mc R(u;s,t)\nonumber\\
	&\le \int_{s}^{t}(\partial_t\mc E(r,u(r))-\partial_t\mc E(r,u(t)))\d r\nonumber\\
	&\le \int_{s}^{t}\gamma_{\bar C}(r)\omega_{\bar C}(|u(t)-u(r)|_W)\d r\le C \int_{s}^{t}\gamma_{\bar C}(r)\d r.
 \end{align}
Hence, for every $t\in [0,T]$ and almost every $s\le t$ we know that
\begin{equation}\label{eq:contX1}
	\|u(t)-u(s)\|_X^2\le C\int_{s}^{t}\gamma_{\bar C}(r)\d r.
\end{equation}
Inequality \eqref{eq:contX1} can be easily extended to all times $s\le t$ by continuity, and so we deduce that $u$ belongs to $C([0,T];X)$.

We now show that, assuming \ref{hyp:E4'}, the function $u$ is absolutely continuous with values in $X$. By arguing as in \eqref{boundcont}, by exploiting \eqref{lipschitzpartialt} this time we obtain
	\begin{equation*}
		\|u(t)-u(s)\|_X^2\le C\int_{s}^{t}\gamma_{\bar C}(r)\|u(t)-u(r)\|_X\d r,\quad\text{for all $0\le s\le t\le T$}.
	\end{equation*}
By using \cite[Lemma~5.6]{GidRiv} one then deduces
\begin{equation*}
	\|u(t)-u(s)\|_X\le C\int_{s}^{t}\gamma_{\bar C}(r)\d r,\quad\text{for all $0\le s\le t\le T$},
\end{equation*}
namely $u$ is in $AC([0,T];X)$.

Let us now assume $X\hookrightarrow Z$, from which $u$ belongs to $AC([0,T];Z)$, and $X$ is reflexive. By \eqref{eq:ensol} and Lemma~\ref{lemma:localstability} we know that for all $t\in [0,T]$ there exists $\xi(t)\in \partial^U\mc E(t,u(t))$ such that $-\xi(t)\in \partial^Z\mc R(0)$. We may now use Lemma~\ref{lemma:chainrule}: by computing the time-derivative in the energy balance in \eqref{eq:ensol} (we recall that in the current regular setting one has $V_\mc R(u;0,t)=\int_{0}^{t}\mc R(\dot{u}(r))\d r$) we obtain
\begin{equation*}
	\mc R(\dot{u}(t))=\langle-\xi(t),\dot{u}(t)\rangle_Z,\quad\text{for a.e. }t\in [0,T].
\end{equation*}
By means of \eqref{2.2mielke}, the above equality ensures that $-\xi(t)\in \partial^Z\mc R(\dot{u}(t))$ almost everywhere in $[0,T]$, and so \eqref{rieq} is satisfied and we conclude.
\end{proof}

\section{Discrete multiscale analysis}\label{sec:multiscale}
This section is devoted to the proof of Theorems~\ref{mainthm:nonconvexdiscr}, \ref{mainthm:convexdiscr} and \ref{mainthm:unifconvexdiscr}. Throughout the section we will always assume that \eqref{eq:embeddings}, \eqref{mass}, \eqref{viscosity}, \ref{hyp:R1} and \ref{hyp:E1}-\ref{hyp:E5} are in force. The additional assumptions used in order to prove each result will be specified when needed; in particular, the identification $D=U=V$ will be made at the very end (Proposition~\ref{prop:costdiscr}), since it is only needed for the nonconvex case.

 Let $\{u^k_{\tau,\eps}\}_{k\in\mc K^0_\tau}$, $\{v^k_{\tau,\eps}\}_{k\in\mc K^0_\tau}$ be the discrete evolution arising from the Minimizing Movements scheme \eqref{schemeincond} and from \eqref{eq:vdiscr}. As a first step we present the discrete energy-dissipation inequality they satisfy, as well as the uniform bounds that originate therefrom. The proof can be carried out, with minor adaptations, along the lines of \cite[Proposition~5.2]{RivScilSol} exploiting the Euler-Lagrange equation \eqref{EL}.

\begin{prop}\label{prop:inequality0}
	For every $m,n\in\mathcal{K}^0_\tau$ with $m\le n$ the following discrete energy-dissipation inequality holds true:
	\begin{equation}
		\begin{split}
			&\quad\frac{\varepsilon^2}{2}{\left|v^n_{\tau,\eps}\right|^2_{\mathbb{M}}} + \mc E(t^{n},u^{n}_{\tau,\eps}) + \sum_{k=m+1}^n\tau\left(\mc R_\varepsilon(v^k_{\tau,\eps}) + \mc R_\varepsilon^*\left(-\eps^2 \mathbb{M}\frac{v^k_{\tau,\eps}-v^{k-1}_{\tau,\eps}}{\tau} - \xi^k_{\tau,\eps}\right)\right)\\
			& \leq  \frac{\varepsilon^2}{2}{|v^m_{\tau,\eps}|^2_{\mathbb{M}}} + \mc E (t^m,u^m_{\tau,\eps}) + \sum_{k=m+1}^n \int_{t^{k-1}}^{t^k}\partial_t\mc E(r,u^{k-1}_{\tau,\eps})\,\mathrm{d}r +\frac{\lambda\tau}{2}\sum_{k=m+1}^n \tau|v^k_{\tau,\eps}|_{W}^2\,,
		\end{split}
		\label{eq:eniquality1}
	\end{equation}
	where $\xi^k_{\tau,\eps}$ has been introduced in \eqref{eq:EL2}.\\
	Furthermore, if $\varepsilon u_1^\varepsilon$ is uniformly bounded in $W$, $\mc E(0,u_0^\eps)$ is uniformly bounded and
	\begin{equation}\label{tauepsbdd}
		\frac \tau\eps\le \frac{\nu^2 c_1^2}{2\lambda},
\end{equation}
where $\nu,c_1>0$ are such that $|\cdot|_\V\ge \nu |\cdot|_V $ and $|\cdot|_V\ge c_1 |\cdot|_W $, then there exists $\widetilde C>0$, independent of $\varepsilon$ and $\tau$, such that 
\begin{equation*}
	\frac{\varepsilon^2}{2}{\left|v^n_{\tau,\eps}\right|^2_{\mathbb{M}}} + \mc E(t^{n},u^{n}_{\tau,\eps}) + \sum_{k=1}^n \tau\left(\mc R_\varepsilon(v^k_{\tau,\eps}) + \mc R_\varepsilon^*\left(-\eps^2 \mathbb{M}\frac{v^k_{\tau,\eps}-v^{k-1}_{\tau,\eps}}{\tau} - \xi^k_{\tau,\eps}\right)\right)\leq \widetilde C,
\end{equation*}
for every $n\in\mc K^0_\tau$.

In particular, one has
	\begin{itemize}
	\item[{$(i')$}] $\max\limits_{k\in\mc K^0_\tau}\mc E(t^k,u^k_{\tau,\eps})\leq \widetilde C$;
	\item[{$(ii')$}] $\max\limits_{k\in\mc K^0_\tau}\|u^k_{\tau,\eps}\|_U\leq \widetilde C$;
	\item[{$(iii')$}] $\max\limits_{k\in\mc K^0_\tau}\eps|v^k_{\tau,\eps}|_{\M}\leq \widetilde C$;
	\item[{$(iv')$}] $\displaystyle\sum_{k=1}^{T/\tau}\tau\|v^k_{\tau,\eps}\|_Z\leq \widetilde C$; 
	\item[{$(v')$}] $\displaystyle \varepsilon\sum_{k=1}^{T/\tau}\tau |v^k_{\tau,\eps}|_{V}^2\leq \widetilde C$.
\end{itemize}
If in addition \ref{hyp:E6} is in force, then there also holds
\begin{itemize}
	\item[{$(vi')$}] $\max\limits_{k\in\mc K_\tau}\|\xi^k_{\tau,\eps}\|_{U^*}\leq \widetilde C$;
	\item[{$(vii')$}]$\displaystyle \sum_{k=1}^{T/\tau}\tau \left\|\varepsilon^2\frac{v^k_{\tau,\eps} - v^{k-1}_{\tau,\eps}}{\tau}\right\|_{{U}^*}^2\leq \widetilde C$.
\end{itemize}
\end{prop}}

The next crucial lemma collects some estimates for the mismatch between the different interpolants introduced in \eqref{interpolants}. Before stating it we fix some additional notation which will be used several times in the sequel. Recalling \eqref{eq:EL2}, we set
\begin{subequations}
	\begin{align}
		&\overline\xi_{\tau,\varepsilon}(t):=\xi^{k}_{\tau,\eps},&&\text{for } t\in(t^{k-1},t^k],\qquad k\in\mathcal{K}_\tau\,;\label{eq:xitaueps}\\
		&\overline\eta_{\tau,\varepsilon}(t):=\eta^{k}_{\tau,\eps},&&\text{for } t\in(t^{k-1},t^k],\qquad k\in\mathcal{K}_\tau\,;\label{eq:etataueps}\\
		& \pi_\tau(t):=t^k,&&\text{for } t\in(t^{k-1},t^k],\qquad k\in\mathcal{K}^0_\tau.
	\end{align}
\end{subequations}

\begin{lemma}\label{lemma:mismatch}
	Assume that $\varepsilon u_1^\varepsilon$ is uniformly bounded in $W$, $\mc E(0,u_0^\eps)$ is uniformly bounded and \eqref{tauepsbdd}. Then there hold
	\begin{subequations}	
	\begin{gather}
		\sup\limits_{t\in [0,T]}\big(\|\overline{u}_{\tau,\varepsilon}(t)-\underline{u}_{\tau,\varepsilon}(t)\|_U+\|\widehat{u}_{\tau,\varepsilon}(t)-\overline{u}_{\tau,\varepsilon}(t)\|_U\big)+\sup\limits_{t\in [\tau,T]}\|\widetilde{u}_{\tau,\varepsilon}(t)-\widehat{u}_{\tau,\varepsilon}(t)\|_U\le C; \label{eq:mism1}\\
		 \sup\limits_{t\in [0,T]}\big(|\overline{u}_{\tau,\varepsilon}(t)-\underline{u}_{\tau,\varepsilon}(t)|_W+|\widehat{u}_{\tau,\varepsilon}(t)-\overline{u}_{\tau,\varepsilon}(t)|_W+|\widetilde{u}_{\tau,\varepsilon}(t)-\widehat{u}_{\tau,\varepsilon}(t)|_W\big)\le C\frac\tau\eps; \label{eq:mism2}\\
	 \|\overline{u}_{\tau,\varepsilon}-\underline{u}_{\tau,\varepsilon}\|_{L^1(0,T;Z)}+\|\widehat{u}_{\tau,\varepsilon}-\overline{u}_{\tau,\varepsilon}\|_{L^1(0,T;Z)}+\|\widetilde{u}_{\tau,\varepsilon}-\widehat{u}_{\tau,\varepsilon}\|_{L^1(0,T;Z)}\le C\tau;  \label{eq:mism3}\\
	 \|\overline{u}_{\tau,\varepsilon}-\underline{u}_{\tau,\varepsilon}\|_{L^2(0,T;V)}+\|\widehat{u}_{\tau,\varepsilon}-\overline{u}_{\tau,\varepsilon}\|_{L^2(0,T;V)}+\|\widetilde{u}_{\tau,\varepsilon}(t)-\widehat{u}_{\tau,\varepsilon}\|_{L^2(\tau,T;V)}\le C\frac{\tau}{\sqrt{\eps}}; \label{eq:mism4}\\
	 \|\dot{\widetilde{u}}_{\tau,\varepsilon}-\dot{\widehat{u}}_{\tau,\varepsilon}\|_{L^2(\tau,T;V)}\le \frac{C}{\sqrt{\eps}}. \label{eq:mism5}
	\end{gather}
	If in addition $\eps u_1^\eps$ is uniformly bounded in $V$, then there also hold
	\begin{equation} \label{eq:mism6}
		\sup\limits_{t\in [0,\tau]}|\widetilde{u}_{\tau,\varepsilon}(t){-}\widehat{u}_{\tau,\varepsilon}(t)|_V\le C,\quad\|\widetilde{u}_{\tau,\varepsilon}{-}\widehat{u}_{\tau,\varepsilon}\|_{L^2(0,\tau;V)}\le C\frac{\tau}{\sqrt{\eps}},\quad \|\dot{\widetilde{u}}_{\tau,\varepsilon}{-}\dot{\widehat{u}}_{\tau,\varepsilon}\|_{L^2(0,\tau;V)}\le \frac{C}{\sqrt{\eps}}.
	\end{equation}
Finally, if \ref{hyp:E6} is in force, then one has
\begin{equation}\label{eq:mism7}
	\|\dot{\widetilde{u}}_{\tau,\varepsilon}-\dot{\widehat{u}}_{\tau,\varepsilon}\|_{L^2(0,T;U^*)}\leq C\frac{\tau}{\eps^2}.
\end{equation}
\end{subequations}
\end{lemma}
\begin{proof}
We first recall that by $(ii')$ and $(iii')$ of Proposition~\ref{prop:inequality0} we have
\begin{equation*}
\max\limits_{k\in\mc K_\tau}\tau\|v^k_{\tau,\eps}\|_U\leq  C \quad \mbox{ and } \quad \max\limits_{k\in\mc K^0_\tau}\tau|v^k_{\tau,\eps}|_{W}\leq C \frac{\tau}{\varepsilon} \,.
\end{equation*}
By a direct computation we now deduce
\begin{subequations}	
	\begin{gather}
		\sup\limits_{t\in [0,T]}\big(\|\overline{u}_{\tau,\varepsilon}(t)-\underline{u}_{\tau,\varepsilon}(t)\|_X+\|\widehat{u}_{\tau,\varepsilon}(t)-\overline{u}_{\tau,\varepsilon}(t)\|_X\big)\le \max\limits_{k\in\mc K_\tau}\tau\|v^k_{\tau,\eps}\|_X \quad \mbox{ for } X\in\{U,W\}\,, \nonumber
\\\sup\limits_{t\in [\tau,T]}\|\widetilde{u}_{\tau,\varepsilon}(t)-\widehat{u}_{\tau,\varepsilon}(t)\|_U\le 2\max\limits_{k\in\mc K_\tau}\tau\|v^k_{\tau,\eps}\|_U\,, \nonumber\\
		 \sup\limits_{t\in [0,T]}|\widetilde{u}_{\tau,\varepsilon}(t)-\widehat{u}_{\tau,\varepsilon}(t)|_W\le 2\max\limits_{k\in\mc K^0_\tau}\tau|v^k_{\tau,\eps}|_{W} \,, \nonumber
	\end{gather}
\end{subequations}
whence the validity of \eqref{eq:mism1} and \eqref{eq:mism2}. 
As regards \eqref{eq:mism3} we compute
\begin{equation*}
\begin{split}
 \|\overline{u}_{\tau,\varepsilon}-\underline{u}_{\tau,\varepsilon}\|_{L^1(0,T;Z)} & = \tau \sum_{k=1}^{T/\tau} \tau\|v^k_{\tau,\eps}\|_Z \,,  \\
\|\widehat{u}_{\tau,\varepsilon}-\overline{u}_{\tau,\varepsilon}\|_{L^1(0,T;Z)} & = \sum_{k=1}^{T/\tau} \|v^k_{\tau,\eps}\|_Z \int_{t^{k-1}}^{t^k}(\tau-(r-t^{k-1}))\,\mathrm{d}r = \frac{\tau}{2} \sum_{k=1}^{T/\tau} \tau\|v^k_{\tau,\eps}\|_Z \,,  \\
\|\widetilde{u}_{\tau,\varepsilon}-\widehat{u}_{\tau,\varepsilon}\|_{L^1(0,T;Z)} & \le \sum_{k=1}^{T/\tau} \frac{\|v^k_{\tau,\eps}\|_Z}{2\tau} \int_{t^{k-1}}^{t^k}(2\tau(r-t^{k-1})-(r-t^{k-1})^2)\,\mathrm{d}r \\
& \,\,\,\,\,\, + \sum_{k=1}^{T/\tau} \frac{\|v^{k-1}_{\tau,\eps}\|_Z}{2\tau} \int_{t^{k-1}}^{t^k}(r-t^{k-1}-\tau)^2 \,\mathrm{d}r \\
& = \frac{\tau}{3} \sum_{k=1}^{T/\tau} \left(\tau\|v^k_{\tau,\eps}\|_Z + \tau \frac{\|v^{k-1}_{\tau,\eps}\|_Z}{2} \right) \\
& \le \tau \left(  \tau \|u_1^\varepsilon\|_Z + \sum_{k=1}^{T/\tau}\tau\|v^k_{\tau,\eps}\|_Z\right),
\end{split} 
\end{equation*}
and we conclude by $(iv')$ in Proposition~\ref{prop:inequality0} and \eqref{tauepsbdd}. 

Estimate \eqref{eq:mism4} can be proved similarly:  
\begin{equation*}
\begin{split}
\int_0^T |\overline{u}_{\tau,\varepsilon}(r)-\underline{u}_{\tau,\varepsilon}(r)|_{V}^2\,\mathrm{d}r & = \sum_{k=1}^{T/\tau} \tau^3 |v^k_{\tau,\eps}|_V^2 = \frac{\tau^2}{\varepsilon}\sum_{k=1}^{T/\tau} \varepsilon \tau |v^k_{\tau,\eps}|_V^2\,,  \\
\int_0^T |\widehat{u}_{\tau,\varepsilon}(r)-\overline{u}_{\tau,\varepsilon}(r)|_{V}^2 \,\mathrm{d}r & = \sum_{k=1}^{T/\tau} |v^k_{\tau,\eps}|_V^2 \int_{t^{k-1}}^{t^k}(r-t^{k-1}-\tau)^2\,\mathrm{d}r \\
& = \frac{1}{3} \sum_{k=1}^{T/\tau} \tau^3 |v^k_{\tau,\eps}|_V^2 = \frac{\tau^2}{\varepsilon}\sum_{k=1}^{T/\tau} \frac{\varepsilon \tau}{3} |v^k_{\tau,\eps}|_V^2\,,  \\
\int_\tau^T|\widetilde{u}_{\tau,\varepsilon}(r)-\widehat{u}_{\tau,\varepsilon}(r)|_{V}^2 \,\mathrm{d}r & \le \sum_{k=2}^{T/\tau} \frac{|v^k_{\tau,\eps}|_V^2}{2\tau^2} \int_{t^{k-1}}^{t^k}(2\tau(r-t^{k-1})-(r-t^{k-1}))^2\,\mathrm{d}r \\
& \,\,\,\,\,\, + \sum_{k=2}^{T/\tau} \frac{|v^{k-1}_{\tau,\eps}|_V^2}{2\tau^2} \int_{t^{k-1}}^{t^k}(r-t^{k-1}-\tau)^4 \,\mathrm{d}r \\
&\le \sum_{k=2}^{T/\tau}\frac{\tau^3}{2}(|v^k_{\tau,\eps}|^2_V+|v^{k-1}_{\tau,\eps}|^2_V)\\
& \le \sum_{k=1}^{T/\tau} \tau^3 |v^k_{\tau,\eps}|_V^2 = \frac{\tau^2}{\varepsilon}\sum_{k=1}^{T/\tau} \varepsilon \tau |v^k_{\tau,\eps}|_V^2 \,, 
\end{split} 
\end{equation*}  
and we conclude by $(v')$ in Proposition~\ref{prop:inequality0}. 

As for \eqref{eq:mism5}, we have
\begin{equation*}
\begin{split}
\int_\tau^T |\dot{\widetilde{u}}_{\tau,\varepsilon}(r){-}\dot{\widehat{u}}_{\tau,\varepsilon}(r)|_{V}^2 \,\mathrm{d}r  & =  \sum_{k=2}^{T/\tau} \frac{|v^k_{\tau,\eps}-v^{k-1}_{\tau,\eps}|_V^2}{\tau^2} \int_{t^{k-1}}^{t^k}(r-t^{k-1}-\tau)^2 \,\mathrm{d}r =  \sum_{k=2}^{T/\tau} \frac{\tau}{3}|v^k_{\tau,\eps}-v^{k-1}_{\tau,\eps}|_V^2 \\
& \le \frac{2}{3} \sum_{k=2}^{T/\tau} \tau (|v^k_{\tau,\eps}|_V^2 + |v^{k-1}_{\tau,\eps}|_V^2) \le \frac{4}{3} \sum_{k=1}^{T/\tau} \tau |v^k_{\tau,\eps}|_V^2 \le \frac{C}{\varepsilon} \,.
\end{split}
\end{equation*}

If in addition $\eps u_1^\eps$ is uniformly bounded in $V$, then by similar computations
	\begin{equation*}
		\int_0^\tau|\widetilde{u}_{\tau,\varepsilon}(r){-}\widehat{u}_{\tau,\varepsilon}(r)|_{V}^2\,\mathrm{d}r \le \tau^3 (|v^1_{\tau,\eps}|_V^2 + |u_1^\eps|_V^2) \le \frac{\tau^2}{\varepsilon}\left(\frac{\tau}{\varepsilon}|\eps u_1^\eps|_V^2 + \sum_{k=1}^{T/\tau} \varepsilon \tau |v^k_{\tau,\eps}|_V^2 \right) \le C \frac{\tau^2}{\varepsilon} \,.
	\end{equation*}
Moreover, 
\begin{equation*}
|\widetilde{u}_{\tau,\varepsilon}(0){-}\widehat{u}_{\tau,\varepsilon}(0)|_{V} = \frac{\tau}{2} |u_1^\eps|_V \le C \frac{\tau}{\eps} \le C\,,
\end{equation*}
and
\begin{equation*}
\begin{split}
\sup\limits_{t\in (0,\tau]}|\widetilde{u}_{\tau,\varepsilon}(t){-}\widehat{u}_{\tau,\varepsilon}(t)|_V & \le \sup\limits_{t\in (0,\tau]} \frac{\tau}{2} \left(|v^1_{\tau,\eps}|_V \frac{t}{\tau}\left(2-\frac{t}{\tau}\right) + |u_1^\eps|_V\left(1-\frac{t}{\tau}\right)^2\right) \\
& \le C (\tau\|v^1_{\tau,\eps}\|_U + \tau |u_1^\eps|_V) \leq C \left(1+\frac{\tau}{\eps}\right)\le C \,.
\end{split}
\end{equation*}
Finally,
\begin{equation*}
\int_0^\tau |\dot{\widetilde{u}}_{\tau,\varepsilon}(r){-}\dot{\widehat{u}}_{\tau,\varepsilon}(r)|_{V}^2 \,\mathrm{d}r = \frac{\tau}{3} |v^1_{\tau,\eps} - u_1^\eps|_V^2 \le \frac{2\tau}{3} (|v^1_{\tau,\eps}|_V^2 + |u_1^\eps|_V^2 ) \le C \left(\frac{1}{\eps}+\frac{\tau}{\eps^2}\right) \le \frac{C}{\varepsilon} \,.  
\end{equation*}
This concludes the proof of \eqref{eq:mism6}. 

Assuming also \ref{hyp:E6}, we finally have
\begin{equation*}
	\int_{0}^{T}\|\dot{\widetilde{u}}_{\tau,\eps}(r)-\dot{\widehat{u}}_{\tau,\eps}\|_{U^*}^2\d r=\sum_{k=1}^{T/\tau}\frac{\tau^3}{3}\left\|\frac{v^k_{\tau,\eps}-v^{k-1}_{\tau,\eps}}{\tau}\right\|_{U^*}^2\le C\frac{\tau^2}{\eps^4},
\end{equation*}
by using $(vii')$ in Proposition~\ref{prop:inequality0}. So \eqref{eq:mism7} is proved and we conclude.
\end{proof}

Next lemma shows how the variation (and the $L^2$-norm) of the piecewise quadratic interpolant $\widetilde{u}_{\tau,\eps}$ can be controlled, up to vanishing terms, by the one of the piecewise affine interpolant $\widehat{u}_{\tau,\eps}$. This result will be crucial in Proposition~\ref{prop:costdiscr}.

\begin{lemma}
	Assume that $\varepsilon u_1^\varepsilon$ is uniformly bounded in $W$, $\mc E(0,u_0^\eps)$ is uniformly bounded and \eqref{tauepsbdd}. Then, for all $0\le s\le t\le T$ one has
	\begin{subequations}
		\begin{align}
			&\int_{\pi_\tau(s)}^{\pi_\tau(t)}\mc R(\dot{\widetilde{u}}_{\tau,\eps}(r))\d r\le \int_{\pi_\tau(s)}^{\pi_\tau(t)}\mc R(\dot{\widehat{u}}_{\tau,\eps}(r))\d r+C\frac\tau\eps; \label{eq:nodeineq1}  \\
			& \int_{\pi_\tau(s)}^{\pi_\tau(t)} |\dot{\widetilde{u}}_{\tau,\eps}(r)|^2_{\V}\d r\le \int_{\pi_\tau(s)}^{\pi_\tau(t)}|\dot{\widehat{u}}_{\tau,\eps}(r)|^2_\V\d r+\tau|\dot{\widehat{u}}_{\tau,\eps}(\pi_\tau(s))|^2_\V. \label{eq:nodeineq2}
		\end{align}
	In particular there holds
	\begin{equation}
		\int_{\pi_\tau(s)}^{\pi_\tau(t)}\mc R_\eps(\dot{\widetilde{u}}_{\tau,\eps}(r))\d r\le \int_{\pi_\tau(s)}^{\pi_\tau(t)}\mc R_\eps(\dot{\widehat{u}}_{\tau,\eps}(r))\d r+C\frac\tau\eps\left(1+|\eps\dot{\widehat{u}}_{\tau,\eps}(\pi_\tau(s))|^2_\V\right).
\label{eq:nodeineq3}
	\end{equation}
	\end{subequations}
\end{lemma}
\begin{proof}
We start proving \eqref{eq:nodeineq2}. Let $m,n\in \mc K^0_\tau$ be such that $\pi_\tau(t)=t^n$ and $\pi_\tau(s)=t^m$, then 
\begin{equation*}
\begin{split}
\int_{\pi_\tau(s)}^{\pi_\tau(t)} |\dot{\widetilde{u}}_{\tau,\eps}(r)|^2_{\V}\d r & = \sum_{k=m+1}^n \left(\frac{|v^k_{\tau,\eps}|_{\V}^2}{\tau^2}\int_{t^{k-1}}^{t^k}(r-t^{k-1})^2\,\mathrm{d}r +\frac{|v^{k-1}_{\tau,\eps}|_{\V}^2}{\tau^2}\int_{t^{k-1}}^{t^k}(r-t^{k-1}-\tau)^2\,\mathrm{d}r \right) \\
& \,\,\,\,\,\, - 2 \sum_{k=m+1}^n \frac{1}{\tau^2}\langle v^k_{\tau,\eps} , v^{k-1}_{\tau,\eps} \rangle_{\V} \int_{t^{k-1}}^{t^k}\left[(r-t^{k-1})^2-\tau(r-t^{k-1})\right]\,\mathrm{d}r \\
& = \sum_{k=m+1}^n \frac{\tau}{3} (|v^k_{\tau,\eps}|_{\V}^2 + |v^{k-1}_{\tau,\eps}|_{\V}^2 + \langle v^k_{\tau,\eps} , v^{k-1}_{\tau,\eps} \rangle_{\V}) \\
& \le \sum_{k=m+1}^n \frac{\tau}{2} (|v^k_{\tau,\eps}|_{\V}^2 + |v^{k-1}_{\tau,\eps}|_{\V}^2 ) \le \sum_{k=m+1}^n {\tau}|v^k_{\tau,\eps}|_{\V}^2 + \frac{\tau}{2} |v^m_{\tau,\eps}|_{\V}^2 \\
&= \int_{\pi_\tau(s)}^{\pi_\tau(t)}|\dot{\widehat{u}}_{\tau,\eps}(r)|^2_\V\d r+\tau|\dot{\widehat{u}}_{\tau,\eps}(\pi_\tau(s))|^2_\V.
\end{split}
\end{equation*}
As regards \eqref{eq:nodeineq1}, by performing similar computations, only exploiting the subadditivity and the positively one-homogeneity of $\mc R$, one obtains
\begin{equation*}
\begin{split}
\int_{\pi_\tau(s)}^{\pi_\tau(t)}\mc R(\dot{\widetilde{u}}_{\tau,\eps}(r))\d r & = \sum_{k=m+1}^n \int_{t^{k-1}}^{t^k} \mc R\left( \frac{v^k_{\tau,\eps}}{\tau}(r-t^{k-1}) + \frac{v^{k-1}_{\tau,\eps}}{\tau}(\tau-(r-t^{k-1})) \right)\,\mathrm{d}r \\
& \le \sum_{k=m+1}^n  \tau  \mc R(v^k_{\tau,\eps}) + \frac{\tau}{2} \mc R(v^m_{\tau,\eps}) \le \int_{\pi_\tau(s)}^{\pi_\tau(t)}\mc R(\dot{\widehat{u}}_{\tau,\eps}(r))\d r + C\frac{\tau}{\eps}|\eps v^m_{\tau,\eps}|_{W}\,,
\end{split}
\end{equation*}
and we conclude by exploiting $(iii')$ in Proposition~\ref{prop:inequality0}. Assertion \eqref{eq:nodeineq3} is then a consequence of \eqref{eq:nodeineq1}, \eqref{eq:nodeineq2} and the definition \eqref{eq:Reps} of $\mc R_\eps$. 
\end{proof}

\subsection{Compactness}

	By using the uniform bounds listed in Proposition~\ref{prop:inequality0} together with the mismatch estimates of Lemma~\ref{lemma:mismatch}, we are able to extract convergent subsequences from the sequence of interpolants \eqref{interpolants}.
\begin{prop}\label{prop:compactnessdiscr}
	Assume that $\varepsilon u_1^\varepsilon$ is uniformly bounded in $W$, $\mc E(0,u_0^\eps)$ is uniformly bounded and \eqref{tauepsbdd}. Then $\widehat{u}_{\tau,\eps},\overline{u}_{\tau,\eps},\underline{u}_{\tau,\eps}$ are uniformly bounded in $B([0,T];U)\cap BV([0,T];Z)$. Moreover $\widetilde{u}_{\tau,\eps}$ is bounded in $B([\tau,T];U)\cap BV([0,T];Z)$ uniformly in $\tau,\eps$.
	
	In particular, for any sequence $(\tau,\eps)\to (0,0)$ such that $\tau/\eps\to 0$ there exists a subsequence $(\tau_j,\eps_j)$ and there exists a function $u\colon [0,T]\to D$ such that $u\in B([0,T];U)\cap BV([0,T];Z)$ and
	\begin{itemize}
		\item[$(a')$] $\widehat{u}_{\tau_j,\eps_j}(t),\overline{u}_{\tau_j,\eps_j}(t),\underline{u}_{\tau_j,\eps_j}(t)\xrightharpoonup[j\to +\infty]{U}u(t),\quad\text{ for every }t\in [0,T]$;
		\item[$(a'')$] $\widetilde{u}_{\tau_j,\eps_j}(t)\xrightharpoonup[j\to +\infty]{U}u(t),\quad\text{ for every }t\in (0,T]$,\ \ \ \  and $\quad\widetilde{u}_{\tau_j,\eps_j}(0)\xrightarrow[j\to +\infty]{W}u(0)$;
		\item[$(b')$] $\epsj\dot{\widetilde{u}}_{\tau_j,\eps_j}(t),\epsj\dot{\widehat{u}}_{\tau_j,\eps_j}(t)\xrightarrow[j\to +\infty]{V}0,\quad\text{ for almost every }t\in [0,T]$;
		\item[$(c')$] $\displaystyle V_\mc R(u;s,t)\le\liminf_{j\to +\infty}\int_{\pi_{\tau_j}(s)}^{\pi_{\tau_j}(t)}\mc R(\dot{\widehat{u}}_{\tau_j,\eps_j}(r))\d r$,\ \ \ \  for all $0\le s\le t\le T$.
	\end{itemize}
If in addition $\eps u_1^\eps$ is uniformly bounded in $V$, then there also holds
\begin{equation*}
	\quad\widetilde{u}_{\tau_j,\eps_j}(0)\xrightharpoonup[j\to +\infty]{V}u(0).
\end{equation*}
\end{prop}
\begin{proof}
By $(ii')$ in Proposition~\ref{prop:inequality0} we know that $\overline{u}_{\tau,\eps},\underline{u}_{\tau,\eps}$ are uniformly bounded in $B([0,T];U)$. By \eqref{eq:mism1} we deduce the same for $\widehat{u}_{\tau,\eps}$, while we obtain that $\widetilde{u}_{\tau,\eps}$ is bounded in $B([\tau,T];U)$ uniformly in $\tau, \eps$. By $(iv')$ in Proposition~\ref{prop:inequality0}, we also know that $\widehat{u}_{\tau,\eps}$ is bounded in $BV([0,T];Z)$. Hence, by \eqref{eq:nodeineq1} and \eqref{eq:mism2} we infer $\widetilde{u}_{\tau,\eps}$ bounded in $BV([0,T];Z)$ as well. As regards $\overline{u}_{\tau,\eps}$ and $\underline{u}_{\tau,\eps}$, let us pick any finite partition $\{s^i\}$ of $[0,T]$. Then there holds
\begin{equation*}
\sum_i \|\overline{u}_{\tau,\eps}(s^i)-\overline{u}_{\tau,\eps}(s^{i-1})\|_Z \le \sum_{k=1}^{T/\tau} \|u^k_{\tau,\eps}-u^{k-1}_{\tau,\eps}\|_Z = \sum_{k=1}^{T/\tau} \tau \|v^k_{\tau,\eps}\|_Z \le C\,,
\end{equation*}
and an analogous estimate can be shown for $\underline{u}_{\tau,\eps}$. So both of them are bounded in $BV([0,T];Z)$. 

Now, consider any two sequences $\eps\to0$ and $\tau\to0$ such that $\frac{\tau}{\eps}\to0$. By using Lemma~\ref{helly} (together with a diagonal argument for $\widetilde{u}_{\tau,\eps}$) we then deduce that, up to subsequences, $(a')$ and $(a'')$ hold true. Indeed all the sequences converge to the same function $u$ due to \eqref{eq:mism2}. The fact that $u$ is $D$-valued follows by weak lower semicontinuity of $\mc E$ together with $(i')$ in Proposition~\ref{prop:inequality0}. The validity of the assertions about $\widetilde{u}_{\tau_j,\eps_j}(0)$ easily follows recalling that
\begin{equation*}
\widetilde{u}_{\tau_j,\eps_j}(0) = u_0^{\eps_j} - \frac{\tau_j}{2} u_1^{\eps_j} = \widehat{u}_{\tau_j,\eps_j}(0) - \frac{\tau_j}{2 \eps_j} \eps_j u_1^{\eps_j} \,. 
\end{equation*}
The validity of $(b')$ is a direct consequence of the inequalities
\begin{equation*}
\begin{split}
&\bullet\,\eps^2 \int_{0}^{T}|\dot{\widehat{u}}_{\tau,\eps}(r)|^2_V\d r = \eps \sum_{k=1}^{T/\tau} \eps\tau |v^k_{\tau,\eps}|^2_V \le C \eps\,, \\
&\bullet\,\eps \|\dot{\widetilde{u}}_{\tau,\eps}\|_{L^2(\tau,T;V)} \leq \eps \|\dot{\widetilde{u}}_{\tau,\eps}-\dot{\widehat{u}}_{\tau,\eps}\|_{L^2(\tau,T;V)} + \eps \|\dot{\widehat{u}}_{\tau,\eps}\|_{L^2(0,T;V)} \leq C\sqrt{\eps}\,,
\end{split}
\end{equation*}
where we used $(v')$ in Proposition~\ref{prop:inequality0} and \eqref{eq:mism5}. As regards $(c')$, we exploit the weak lower semicontinuity of the $\mc R$-variation so that we estimate
\begin{equation*}
\begin{split}
\displaystyle V_\mc R(u;s,t) & \le \liminf_{j\to +\infty}\int_{s}^{t}\mc R(\dot{\widehat{u}}_{\tau_j,\eps_j}(r))\d r \\
& \le\liminf_{j\to +\infty}\left(\int_{\pi_{\tau_j}(s)}^{\pi_{\tau_j}(t)}\mc R(\dot{\widehat{u}}_{\tau_j,\eps_j}(r))\d r + \int_{s}^{\pi_{\tau_j}(s)}\mc R(\dot{\widehat{u}}_{\tau_j,\eps_j}(r))\d r\right) \,.
\end{split}
\end{equation*}
We now conclude since by $(iii')$ in Proposition~\ref{prop:inequality0} we have
\begin{equation*}
\int_{s}^{\pi_{\tau_j}(s)}\mc R(\dot{\widehat{u}}_{\tau_j,\eps_j}(r))\d r \le C \int_{s}^{\pi_{\tau_j}(s)}|\dot{\widehat{u}}_{\tau_j,\eps_j}(r)|_W\d r \le C \frac{\pi_{\tau_j}(s)-s}{\eps_j} 
\le C \frac{\tau_j}{\eps_j} \xrightarrow[j\to +\infty]{} 0 \,. 
\end{equation*}
\end{proof}

\subsection{Limit passage in the stability condition}

	We now show how the minimality property of $\{u^k_{\tau,\eps}\}_{k\in \mc K_\tau}$ leads in the limit to the local stability condition \eqref{eq:LS}. Next proposition is the analogue of proposition~\ref{prop:estimate} in the discrete setting.
	\begin{prop}
		Assume that $\varepsilon u_1^\varepsilon$ is uniformly bounded in $W$, $\mc E(0,u_0^\eps)$ is uniformly bounded and \eqref{tauepsbdd}. Then for any function $g_\tau\colon (-\tau,T]\to U$ of the form
		\begin{equation*}
			g_\tau(t)=\sum_{k=0}^{T/\tau}g^k_\tau\bm 1_{(t^{k-1},t^k]}(t), \qquad\text{ with }g^k_\tau\in U,
		\end{equation*}
	there exists a constant $C_\tau>0$, depending on $\|g_\tau\|_{B([0,T];U)\cap BV([0,T];W)}$, such that for all $0\le s\le t\le T$ there holds
	\begin{equation}\label{eq:estimatediscr}
		\begin{aligned}
		&\quad\,\int_{\pi_\tau(s)}^{\pi_\tau(t)}\mc E(\pi_\tau(r),\overline{u}_{\tau,\eps}(r))\d r\\
		&\le \int_{\pi_\tau(s)}^{\pi_\tau(t)}\Big(\mc E(\pi_\tau(r),g_\tau(r))+\mc R(g_\tau(r)-\overline{u}_{\tau,\eps}(r))+\frac \lambda 2 |g_\tau(r)-\overline{u}_{\tau,\eps}(r)|_W^2\Big)\d r+C_\tau\sqrt{\eps}.
	\end{aligned}
	\end{equation}
	\end{prop}
	\begin{proof}
Without loss of generality, we may assume that $g^k_\tau\in D$ for every $k\in \mc K_\tau$. By testing \eqref{eq:EL2} by $u^k_{\tau,\eps}-g^k_\tau\in U$ and using \eqref{lambdaconvexineq}, for all $k\in \mc K_\tau$ we obtain
\begin{equation*}
\begin{split}
\tau \mc E(t^k,u^k_{\tau,\eps}) & \le \tau \mc E(t^k,g^k_\tau) + \tau \mc R(g^k_\tau - u^k_{\tau,\eps}) + \frac{\lambda}{2} \tau |g^k_\tau-u^k_{\tau,\eps}|_W^2 \\
& \,\,\,\,\,\, - \eps^2 \langle \mathbb{M}(v^k_{\tau,\eps}-v^{k-1}_{\tau,\eps}, u^k_{\tau,\eps}-g_\tau^k) \rangle_W  - \eps\tau \langle \V v^k_{\tau,\eps}, u^k_{\tau,\eps}-g_\tau^k \rangle_V \,.
\end{split}
\end{equation*}
By summing from $k=m+1$ to $n$, where $\pi_\tau(s)=t^m$ and $\pi_\tau(t)=t^n$, we obtain
\begin{equation*}
		\begin{aligned}
		\int_{\pi_\tau(s)}^{\pi_\tau(t)}\!\!\mc E(\pi_\tau(r),\overline{u}_{\tau,\eps}(r))\d r\le& \int_{\pi_\tau(s)}^{\pi_\tau(t)}\!\!\Big(\mc E(\pi_\tau(r),g_\tau(r))+\mc R(g_\tau(r){-}\overline{u}_{\tau,\eps}(r))+\frac \lambda 2 |g_\tau(r){-}\overline{u}_{\tau,\eps}(r)|_W^2\Big)\d r \\
& + \left| \sum_{k=m+1}^n \left(\eps^2\langle \mathbb{M}(v^k_{\tau,\eps}-v^{k-1}_{\tau,\eps}) , u^k_{\tau,\eps}-g_\tau^k \rangle_W + \eps\tau \langle \V v^k_{\tau,\eps}, u^k_{\tau,\eps}-g_\tau^k\rangle_V \right)  \right|.
	\end{aligned}
	\end{equation*}
We now estimate separately the terms within the absolute value by using $(ii')$, $(iii')$ and $(v')$ in Proposition~ \ref{prop:inequality0}. 

We have
\begin{equation*}
\begin{split}
&\quad\,\left| \sum_{k=m+1}^n \eps\tau \langle \V v^k_{\tau,\eps}, u^k_{\tau,\eps}-g_\tau^k\rangle_V \right|  \le C \sqrt{\eps} \sum_{k=m+1}^n \sqrt{\eps} \tau |v^k_{\tau,\eps}|_V (\|u^k_{\tau,\eps}\|_U + \|g_\tau^k\|_U) \\
& \le C \sqrt{\eps} (1+\|g_\tau\|_{B([0,T];U)}) \left(\sum_{k=m+1}^n \tau\right)^\frac{1}{2}\left(\sum_{k=m+1}^n \eps \tau |v^k_{\tau,\eps}|_V^2 \right)^\frac{1}{2} \le C (1+\|g_\tau\|_{B([0,T];U)}) \sqrt{\eps}. 
\end{split}
\end{equation*}
As for the first term, we first note that
\begin{equation*}
\begin{split}
\sum_{k=m+1}^n  \eps^2\langle \mathbb{M}(v^k_{\tau,\eps}-v^{k-1}_{\tau,\eps}) ,& u^k_{\tau,\eps}-g_\tau^k \rangle_W = \eps \bigg[ \langle \eps \mathbb{M}v^n_{\tau,\eps}  , u^n_{\tau,\eps}-g_\tau^n \rangle_W - \langle \eps \mathbb{M}v^m_{\tau,\eps} , u^m_{\tau,\eps}-g_\tau^m \rangle_W \\ & - \sum_{k=m+1}^n\left( \tau \langle \sqrt{\eps} \mathbb{M}v^{k-1}_{\tau,\eps}  , \sqrt{\eps} v^{k}_{\tau,\eps} \rangle_W 
 - \langle \eps \mathbb{M}v^{k-1}_{\tau,\eps}  , g_\tau^k - g_\tau^{k-1} \rangle_W\right) \bigg] \,.
\end{split}
\end{equation*}
The first two terms within square brackets can be estimated as
\begin{equation*}
 |\langle \eps \mathbb{M}v^i_{\tau,\eps}  , u^i_{\tau,\eps}-g_\tau^i \rangle_W |\le C \eps |v^i_{\tau,\eps}|_W (\|u^i_{\tau,\eps}\|_U + \|g_\tau^i\|_U) \le C (1+\|g_\tau\|_{B([0,T];U)})\,, \quad i=n,m\,,
\end{equation*}
while concerning the third one we have
\begin{equation*}
\begin{split}
&\quad\,\left|\sum_{k=m+1}^n \tau \langle \sqrt{\eps} \mathbb{M}v^{k-1}_{\tau,\eps}  , \sqrt{\eps} v^{k}_{\tau,\eps} \rangle_W \right|  \le C \sum_{k=m+1}^n (\sqrt{\tau \eps} |v^{k-1}_{\tau,\eps}|_W) (\sqrt{\tau \eps} |v^{k}_{\tau,\eps}|_W ) \\
&  \le C \left(\tau\eps |u_1^\eps|_W^2 + \sum_{k=1}^{T/\tau} \tau \eps |v^{k}_{\tau,\eps}|_W^2 \right)^\frac{1}{2} \left(\sum_{k=1}^{T/\tau}\tau \eps |v^{k}_{\tau,\eps}|_W^2 \right)^\frac{1}{2} 
 \le C .
\end{split}
\end{equation*}
Finally, we bound the last term as
\begin{equation*}
\left|\sum_{k=m+1}^n \langle \eps \mathbb{M}v^{k-1}_{\tau,\eps}  , g_\tau^k - g_\tau^{k-1} \rangle_W \right|\le C \sum_{k=m+1}^n \eps |v^{k-1}_{\tau,\eps}|_W  |g_\tau^k - g_\tau^{k-1}|_W \\
\le C \displaystyle V_W(g_\tau; 0,T),
\end{equation*}
and the proof is concluded.
\end{proof}

By sending $(\tau,\eps)\to (0,0)$ in \eqref{eq:estimatediscr}, as in Proposition~\ref{prop:lambdastab} we obtain the validity of the global $\lambda$-stability for the limit function $u$.
	\begin{prop}\label{prop:lambdastabdiscr}
			Assume that $\varepsilon u_1^\varepsilon$ is uniformly bounded in $W$, $\mc E(0,u_0^\eps)$ is uniformly bounded for a sequence $(\tau,\eps)\to (0,0)$ such that $\tau/\eps\to 0$. Then the limit function $u$ obtained in Proposition~\ref{prop:compactnessdiscr} satisfies $(\lambda GS)$, $(\lambda GS^\pm)$ and \eqref{globalstability}.
			
			Moreover it holds
			\begin{subequations}
			\begin{equation} \label{eq:(A)}
				\liminf_{j\to +\infty}\mc E(\pi_{\tau_j}(t), \overline{u}_{\tau_j,\eps_j}(t))=\mc E(t,u(t)),\qquad\text{for almost every }t\in [0,T].
			\end{equation}
		Assuming in addition \ref{hyp:E7}, up to a further subsequence, there holds
		\begin{equation}\label{eq:(B)}
			\lim_{j\to +\infty}\mc E(\pi_{\tau_j}(t), \overline{u}_{\tau_j,\eps_j}(t))=\mc E(t,u(t)),\qquad\text{for almost every }t\in [0,T].
		\end{equation}
	\end{subequations}
	In particular, if also \ref{hyp:E8} is in force, then
	\begin{equation}
		\overline{u}_{\tau_j,\eps_j}(t)\xrightarrow[j\to +\infty]{U}u(t),\quad\text{ for almost every }t\in [0,T].
\label{eq:(C)}
	\end{equation}
	\end{prop}
	\begin{proof}
Let $x\in D$ be fixed. By Fatou's Lemma and exploiting \eqref{eq:estimatediscr} with $g_\tau(t)\equiv x$ we deduce
\begin{equation*}
		\begin{aligned}
&\quad\,\int_{s}^{t}\mc E(r,u(r))\d r\leq \liminf_{j\to +\infty} \int_{\pi_{\tau_j}(s)}^{\pi_{\tau_j}(t)}\mc E(\pi_{\tau_j}(r),\overline{u}_{\tau_j,\eps_j}(r))\d r\\
		&\le \lim_{j\to +\infty} \left[\int_{\pi_{\tau_j}(s)}^{\pi_{\tau_j}(t)}\Big(\mc E(\pi_{\tau_j}(r),x)+\mc R(x-\overline{u}_{\tau_j,\eps_j}(r))+\frac \lambda 2 |x-\overline{u}_{\tau_j,\eps_j}(r)|_W^2\Big)\d r +C_x\sqrt{\eps_j} \right] \\
& = \int_{s}^{t}\Big(\mc E(r,x)+\mc R(x-u(r))+\frac \lambda 2 |x-u(r)|_W^2\Big)\d r \,,  
	\end{aligned}
\end{equation*}
where we used the strong convergence $\overline{u}_{\tau_j,\eps_j}(r)\xrightarrow[j\to +\infty]{W,Z}u(r)$ for all $r\in[0,T]$ and Lebesgue Dominated Convergence Theorem. Arguing as in Proposition~\ref{prop:lambdastab} we now obtain $(\lambda GS)$, $(\lambda GS^\pm)$, \eqref{globalstability} and \eqref{eq:(A)}. 

Assume now \ref{hyp:E7}. By Lemma~\ref{lemma:approximation} we can find a sequence $u_n:[0,T]\to F:=\{\mc E (0,\cdot) \leq \widetilde{C}\}$ of the form \eqref{simplefunction} such that 
$u_n\xrightarrow[n\to +\infty]{L^1(0,T;U)}u$. We then define
\begin{equation*}
			u_n^\tau(t) := \sum_{k=0}^{T/\tau}u_n(t^k)\bm 1_{(t^{k-1},t^k]}(t)\,, 
		\end{equation*} 
and we observe that there holds $\|u_n^\tau\|_{B([0,T];U)} \le C$, since $u_n$ is $F$-valued . Furthermore, since $u_n$ is a simple function, one has $\displaystyle V_W(u_n^\tau; 0,T) \le \displaystyle V_W(u_n ; 0,T)$ and 
\begin{equation*}
u_n^\tau (t) \xrightarrow[\tau\to 0]{U}u_n(t), \quad \mbox{ for almost every }t\in[0,T]\,. 
\end{equation*}
By arguing as before we thus obtain
\begin{equation*}
		\begin{aligned}
&\quad\,\int_{0}^{T}\mc E(r,u(r))\d r \leq \liminf_{j\to +\infty} \int_{0}^{T}\mc E(\pi_{\tau_j}(r),\overline{u}_{\tau_j,\eps_j}(r))\d r \leq \limsup_{j\to +\infty} \int_{0}^{T}\mc E(\pi_{\tau_j}(r),\overline{u}_{\tau_j,\eps_j}(r))\d r\\
		&\le \lim_{j\to +\infty} \left[\int_{0}^{T}\Big(\mc E(\pi_{\tau_j}(r),u_n^{\tau_j} (r))+\mc R(u_n^{\tau_j} (r)-\overline{u}_{\tau_j,\eps_j}(r))+\frac \lambda 2 |u_n^{\tau_j} (r)-\overline{u}_{\tau_j,\eps_j}(r)|_W^2\Big)\d r +C_n\sqrt{\eps_j} \right] \\
& = \int_{0}^{T}\Big(\mc E(r,u_n(r))+\mc R(u_n(r)-u(r))+\frac \lambda 2 |u_n(r)-u(r)|_W^2\Big)\d r \,. 
	\end{aligned}
\end{equation*}
We then conclude as in the proof of Proposition~\ref{prop:lambdastab}.   
	\end{proof}

We finally point out that, once the proof of Proposition~\ref{prop:compactnessdiscr} is achieved, one may immediately deduce that the limit function $u$ satisfies the conclusions of Lemmas~\ref{lemma:energy+-} and \ref{lemma:limz*}. Indeed the proof of these results makes only use of the global $\lambda$-stability condition.

\subsection{Limit passage in the discrete energy-dissipation inequality}

The aim of this section is performing the asymptotic analysis of the energy-dissipation inequality \eqref{eq:eniquality1} as $(\tau,\eps)\to (0,0)$.
\begin{prop}\label{prop:EBdiscr}
	Assume that $\varepsilon u_1^\varepsilon$ is uniformly bounded in $W$, $\mc E(0,u_0^\eps)$ is uniformly bounded for a sequence $(\tau,\eps)\to (0,0)$ such that $\tau/\eps\to 0$. Then the limit function $u$ obtained in Proposition~\ref{prop:compactnessdiscr} fulfils \eqref{eq:ineqpm} and \eqref{eq:claim}. If in addition $\eps u_1^\eps\xrightarrow[\eps\to 0]{W}0$ and $\mc E(0,u_0^\eps)\xrightarrow[\eps\to 0]{}\mc E(0,u_0)$, then \eqref{ineq-} holds true also for $s=0$ and \eqref{eneq1} holds as well.
	
	Moreover, assume that either $\lambda=0$ in \ref{hyp:E5} or \ref{hyp:E9}, \ref{hyp:E10} are in force. Then for every $0\le s\le t\le T$ there holds
	\begin{equation}\label{eq:EBdiscr}
		\mc E(t,u^+(t))+V_\mc R(u_{\rm co};s,t)+\sum_{r\in J_u^{\rm e}\cap [s,t]}\mu(\{r\})= \mc E(s,u^-(s))+\int_{s}^{t}\partial_t\mc E(r,u(r))\d r,
	\end{equation}
for a certain positive Radon measure $\mu$.
\end{prop}
\begin{proof}
By arguing as in the proof of Proposition~\ref{prop:ineqpm}, it is enough to prove \eqref{eq:claim}, namely
\begin{equation*}
			\mc E(t,u(t))+V_\mc R(u;s,t)\le \mc E(s,u(s))+\int_{s}^{t}\partial_t\mc E(r,u(r))\d r,\quad \text{for all $t\in [0,T]$ and for a.e. $s\in [0,t]$,}
	\end{equation*}
with $s=0$ admissible if $\eps u_1^\eps\xrightarrow[\eps\to 0]{W}0$ and $\mc E(0,u_0^\eps)\xrightarrow[\eps\to 0]{}\mc E(0,u_0)$. 

By \eqref{eq:eniquality1}, rewritten in terms of the different interpolants, we have that for all $t\in[0,T]$ and for all $s\in[0,t]\backslash\Pi_{\tau_j}$ there holds 
\begin{equation*}
		\begin{split}
			\mc E(\pi_{\tau_j}(t),\overline{u}_{\tau_j,\eps_j}(t)) + \int_{\pi_{\tau_j}(s)}^{\pi_{\tau_j}(t)}&\mc R(\dot{\widehat{u}}_{\tau_j,\eps_j}(r))\d r  \leq  \frac{\varepsilon_j^2}{2}{|\dot{\widehat{u}}_{\tau_j,\eps_j}(s)|^2_{\mathbb{M}}} + \mc E(\pi_{\tau_j}(s),\overline{u}_{\tau_j,\eps_j}(s))  \\
& + \int_{\pi_{\tau_j}(s)}^{\pi_{\tau_j}(t)}\partial_t\mc E(r,\underline{u}_{\tau_j,\eps_j}(r))\,\mathrm{d}r +\frac{\lambda}{2}\frac{\tau_j}{\eps_j}\eps_j\int_{\pi_{\tau_j}(s)}^{\pi_{\tau_j}(t)}|\dot{\widehat{u}}_{\tau_j,\eps_j}(r)|_{W}^2\,\mathrm{d}r\,.
		\end{split}
			\end{equation*}
By taking $s\in[0,t]\backslash \left(\bigcup_{j\in\N}\Pi_{\tau_j}\right)$ such that \eqref{eq:(A)} and $(b')$ in Proposition~\ref{prop:compactnessdiscr} hold, by means of $(c')$ in Proposition~\ref{prop:compactnessdiscr} and $(v')$ in Proposition~\ref{prop:inequality0} we deduce 
\begin{equation*}
\begin{split}
			& \mc E(t,u(t))+V_\mc R(u;s,t) \\
& \,\,\,\,\,\, \le \liminf_{j\to +\infty} \left( \frac{\varepsilon_j^2}{2}{|\dot{\widehat{u}}_{\tau_j,\eps_j}(s)|^2_{\mathbb{M}}} + \mc E(\pi_{\tau_j}(s),\overline{u}_{\tau_j,\eps_j}(s)) + \int_{\pi_{\tau_j}(s)}^{\pi_{\tau_j}(t)}\partial_t\mc E(r,\underline{u}_{\tau_j,\eps_j}(r))\,\mathrm{d}r  + C \frac{\tau_j}{\eps_j}\right) \\
&\,\,\,\,\,\,  = \mc E(s,u(s))+\int_{s}^{t}\partial_t\mc E(r,u(r))\d r \,, 
\end{split}
	\end{equation*}
where we exploited \ref{hyp:E1}, \ref{hyp:E2}, \ref{hyp:E4} and the fact that $\underline{u}_{\tau_j,\eps_j}(r)\xrightarrow[j\to +\infty]{W}u(r)$ for every $r\in[0,T]$. 
Finally, for what concerns \eqref{eq:EBdiscr}, we point out that the results stated in Proposition~\ref{prop:prop6.11} hold true also in this setting, since their proof relies only on the properties of the limit function $u$. The proof is hence concluded. 
\end{proof}

\subsection{Proof of Theorems~\ref{mainthm:nonconvexdiscr}, \ref{mainthm:convexdiscr} and \ref{mainthm:unifconvexdiscr}}

We are now in a position to conclude the proof of the results stated in Section~\ref{subsec:minmov}. We begin with the simpler convex setting.
\begin{proof}[Proof of Theorem~\ref{mainthm:convexdiscr}]
	Using Proposition~\ref{prop:lambdastabdiscr} and \eqref{eq:EBdiscr} one is only left to show that
	\begin{equation*}
		\mu(\{t\})=\mc R(u^+(t)-u^-(t)),\qquad\text{for all }t\in J^{\rm e}_u.
	\end{equation*}
This can be done by the very same arguments of Proposition~\ref{prop:costconvex}.
\end{proof}

\begin{proof}[Proof of Theorem~\ref{mainthm:unifconvexdiscr}]
	By Proposition~\ref{prop:EBdiscr} we know that \eqref{eq:claim} holds. By Proposition~\ref{prop:lambdastabdiscr} we know that the global $\lambda$-stability is satisfied. Using these two properties, arguing exactly as in the proof of Theorem~\ref{mainthm:unifconvex} we first deduce that $u$ belongs to $C([0,T];X)$. Hence, inserting this information into \eqref{eq:EBdiscr} we obtain that $u$ is an energetic solution to \eqref{riprob}.
	Eventually, assuming \ref{hyp:E4'} we infer that $u$ is a classic solution to \eqref{riprob} by repeating the arguments in the proof of Theorem~\ref{mainthm:unifconvex}.
\end{proof}

	In the nonconvex case, in view of Proposition~\ref{prop:uppercost} the only missing point for obtaining Theorem~\ref{mainthm:nonconvexdiscr} is proving that the precise viscoinertial cost $c^{\M,\V}_0$ is a lower bound for the atomic measure $\mu$ provided by Proposition~\ref{prop:EBdiscr}. This requires in particular the assumptions $D=U=V$ and \ref{hyp:E11} as stated in the proposition below.

	\begin{prop}\label{prop:costdiscr}
		If $\lambda>0$ in \ref{hyp:E5}, assume $D=U=V$ together with \ref{hyp:E6}-\ref{hyp:E11}. Suppose furthermore that $\eps u^\eps_1\xrightarrow[\eps\to 0]{U}0$ and $u^\eps_0\xrightarrow[\eps\to 0]{U}u_0$ for some $(\tau,\eps)\to(0,0)$ such that $\tau/\eps\to 0$. Then
			\begin{equation*}
			\mu(\{t\})\ge c^{\M,\V}_0(t;u^-(t),u^+(t)),\quad\text{for all }t\in J_u^{\rm e}.
		\end{equation*}
	\end{prop}
\begin{proof}
Fix $t\in J_u^{\rm e}$ and $\alpha>0$. We will prove that 
\begin{equation*}
\mu(\{t\})\ge c^\alpha_0(t;u^-(t),u^+(t)) \,.   
\end{equation*}
We start observing that, since $V=U$ and $\eps u^\eps_1\xrightarrow[\eps\to 0]{U}0$, by \eqref{eq:mism4} combined with \eqref{eq:mism6} we deduce $\|{\widetilde{u}}_{\tau,\eps}-{\overline{u}}_{\tau,\eps}\|_{L^2(0,T;U)}\xrightarrow[(\tau,\eps)\to(0,0)]{}0$. Since by \eqref{eq:(C)} we have $\overline{u}_{\tau_j,\eps_j}\xrightarrow[j\to +\infty]{L^2(0,T;U)}u$, we infer $\widetilde{u}_{\tau_j,\eps_j}\xrightarrow[j\to +\infty]{L^2(0,T;U)}u$. Up to possibly extracting further subsequences we thus obtain
\begin{equation*}
\widetilde{u}_{\tau_j,\eps_j}(s)\,,\, \overline{u}_{\tau_j,\eps_j}(s) \xrightarrow[j\to+\infty]{U}u(s), \quad \mbox{ for a.e. }s\in[0,T]\,. 
\end{equation*}
By using \eqref{eq:mism4} we also deduce
\begin{equation*}
\widehat{u}_{\tau_j,\eps_j}(s)\,,\, \underline{u}_{\tau_j,\eps_j}(s) \xrightarrow[j\to+\infty]{U}u(s), \quad \mbox{ for a.e. }s\in[0,T]\,. 
\end{equation*}
Moreover, from $(b')$ in Proposition~\ref{prop:compactnessdiscr} and \eqref{eq:(B)} we know that 
\begin{equation*}
\begin{split}
\eps_j\dot{\widehat{u}}_{\tau_j,\eps_j}(s)\,,\, \eps_j\dot{\widetilde{u}}_{\tau_j,\eps_j}(s) & \xrightarrow[j\to+\infty]{U} 0, \quad \mbox{ for a.e. }s\in[0,T]; \\
\mc E(\pi_{\tau_j}(s),\overline{u}_{\tau_j,\eps_j}(s)) & \xrightarrow[j\to+\infty]{} \mc E(s,u(s)), \quad \mbox{ for a.e. }s\in[0,T] \,. 
\end{split}
\end{equation*}
We also recall that
\begin{equation*}
\begin{split}
u(r)&\xrightharpoonup[r\to s^\pm]{U}u^\pm(s),\quad\text{ for every }s\in [0,T];\\
\mc E(r,u(r))&\xrightarrow[\substack{r\to s^\pm\\ r\notin J_u}]{}\mc E(s, u^\pm(s)),\quad\text{ for every }s\in [0,T] \,. 
\end{split}
\end{equation*}
So, by \ref{hyp:E2} we obtain $\mc E(s,u(r))\xrightarrow[\substack{r\to s^\pm\\ r\notin J_u}]{}\mc E(s, u^\pm(s))$, which by \ref{hyp:E8} implies 
\begin{equation*}
u(r)\xrightarrow[\substack{r\to s^\pm\\ r\notin J_u}]{U}u^\pm(s),\quad\text{ for every }s\in [0,T] \,.
\end{equation*}
By arguing as in \cite[Proposition~5.9]{SciSol18}, by a diagonal argument we can thus construct two sequences $t_j^+\xrightarrow[j\to+\infty]{}t^+$, $t_j^-\xrightarrow[j\to+\infty]{}t^-$ and extract further subsequences $\eps_j,\tau_j$ such that 
\begin{subequations}
\begin{gather}
\widetilde{u}_{\tau_j,\eps_j}(t_j^\pm)\,,\, \overline{u}_{\tau_j,\eps_j}(t_j^\pm)  \xrightarrow[j\to+\infty]{U}u^\pm(t); \label{eq:vanisheq2}\\
\eps_j\dot{\widetilde{u}}_{\tau_j,\eps_j}(t_j^\pm)\,,\,\eps_j\dot{\widehat{u}}_{\tau_j,\eps_j}(t_j^\pm)  \xrightarrow[j\to+\infty]{U} 0; \label{eq:vanisheq1}\\
\mc E(\pi_{\tau_j}(t_j^\pm),\overline{u}_{\tau_j,\eps_j}(t_j^\pm))  \xrightarrow[j\to+\infty]{} \mc E(t,u^\pm(t)) \,.
\end{gather}
\end{subequations}
Thus, we infer
\begin{equation*}
\begin{split}
\mu(\{t\}) & = \mc E(t,u^-(t)) - \mc E(t,u^+(t)) \\
& = \lim_{j\to+\infty} \biggl[\frac{\varepsilon_j^2}{2}{|\dot{\widehat{u}}_{\tau_j,\eps_j}(t_j^-)|^2_{\mathbb{M}}} + \mc E(\pi_{\tau_j}(t_j^-),\overline{u}_{\tau_j,\eps_j}(t_j^-)) - \frac{\varepsilon_j^2}{2}{|\dot{\widehat{u}}_{\tau_j,\eps_j}(t_j^+)|^2_{\mathbb{M}}} - \mc E(\pi_{\tau_j}(t_j^+),\overline{u}_{\tau_j,\eps_j}(t_j^+)) \\
& \,\,\,\,\,\,\,\, \,\,\,\,\,\, \,\, \,\,\,\,\,\,\,\,+ \int_{\pi_{\tau_j}(t_j^-)}^{\pi_{\tau_j}(t_j^+)}\partial_t\mc E(r,\underline{u}_{\tau_j,\eps_j}(r))\,\mathrm{d}r +\frac{\tau_j}{\eps_j}\frac{\lambda}{2}\int_{\pi_{\tau_j}(t_j^-)}^{\pi_{\tau_j}(t_j^+)}\eps_j|\dot{\widehat{u}}_{\tau_j,\eps_j}(r)|_{W}^2\,\mathrm{d}r  \biggr ] \,.
\end{split}
\end{equation*}
Observe that necessarily 
\begin{equation}
\pi_{\tau_j}(t_j^+) > \pi_{\tau_j}(t_j^-),  \quad \mbox{ for $j$ large enough},
\label{eq:star}
\end{equation}
otherwise the term within square brackets above is identically 0, and this leads to a contradiction since $\mu(\{t\})>0$ for $t\in J_u^e$. 
By using \eqref{eq:eniquality1}, rewritten in terms of the different interpolants, we obtain
\begin{equation*}
\mu(\{t\}) \geq \limsup_{j\to+\infty} \int_{\pi_{\tau_j}(t_j^-)}^{\pi_{\tau_j}(t_j^+)}\left(\mc R_{\varepsilon_j}(\dot{\widehat{u}}_{\tau_j,\eps_j}(r)) + \mc R_{\varepsilon_j}^*\left(-\eps_j^2 \mathbb{M}\ddot{\widetilde{u}}_{\tau_j,\eps_j}(r) - \overline{\xi}_{\tau_j,\eps_j}(r)\right)\right) \,\mathrm{d}r,
\end{equation*}
where $\overline{\xi}_{\tau,\eps}$ has been introduced in \eqref{eq:xitaueps}. By \eqref{eq:nodeineq3} we hence deduce
\begin{equation*}
\begin{split}
\mu(\{t\}) & \geq \limsup_{j\to+\infty} \int_{\pi_{\tau_j}(t_j^-)}^{\pi_{\tau_j}(t_j^+)}\left(\mc R_{\varepsilon_j}(\dot{\widetilde{u}}_{\tau_j,\eps_j}(r)) + \mc R_{\varepsilon_j}^*\left(-\eps_j^2 \mathbb{M}\ddot{\widetilde{u}}_{\tau_j,\eps_j}(r) - \overline{\xi}_{\tau_j,\eps_j}(r)\right)\right) \,\mathrm{d}r  \\
& \geq \limsup_{j\to+\infty} \int_{\pi_{\tau_j}(t_j^-)}^{\pi_{\tau_j}(t_j^+)} p_\V \left(\dot{\widetilde{u}}_{\tau_j,\eps_j}(r),-\eps_j^2 \mathbb{M}\ddot{\widetilde{u}}_{\tau_j,\eps_j}(r) - \overline{\xi}_{\tau_j,\eps_j}(r)\right)\,\mathrm{d}r \,.
\end{split}
\end{equation*}
Since $V=U$ and $\eps u^\eps_1\xrightarrow[\eps\to 0]{U}0$, by Proposition~\ref{prop:compactnessdiscr} and Lemma~ \ref{lemma:mismatch} both $\widetilde{u}_{\tau_j,\eps_j}$ and $\overline{u}_{\tau_j,\eps_j}$ are uniformly bounded in $B([0,T];U)$. Thus, by \ref{hyp:E11}, it holds that for all $r\in(0,T)$ there exists $\widetilde{\xi}_{\tau_j,\eps_j}(r) \in \partial^U\mc E(\pi_{\tau_j}(r),{\widetilde{u}}_{\tau_j,\eps_j}(r))$ such that 
\begin{equation*}
|\overline{\xi}_{\tau_j,\eps_j}(r)-\widetilde{\xi}_{\tau_j,\eps_j}(r)|_{U^*} \le C_M |\overline{u}_{\tau_j,\eps_j}(r)- \widetilde{u}_{\tau_j,\eps_j}(r)|_U . 
\end{equation*} 
We will then obtain
\begin{equation}\label{eq:tilde}
\mu(\{t\}) \geq \limsup_{j\to+\infty} \int_{\pi_{\tau_j}(t_j^-)}^{\pi_{\tau_j}(t_j^+)} p_\V \left(\dot{\widetilde{u}}_{\tau_j,\eps_j}(r),-\eps_j^2 \mathbb{M}\ddot{\widetilde{u}}_{\tau_j,\eps_j}(r) - \widetilde{\xi}_{\tau_j,\eps_j}(r)\right)\,\mathrm{d}r ,
\end{equation}
once we have proved that the difference
\begin{equation*}
 \int_{\pi_{\tau_j}(t_j^-)}^{\pi_{\tau_j}(t_j^+)} \left[p_\V \left(\dot{\widetilde{u}}_{\tau_j,\eps_j}(r),{-}\eps_j^2 \mathbb{M}\ddot{\widetilde{u}}_{\tau_j,\eps_j}(r) {-} \overline{\xi}_{\tau_j,\eps_j}(r)\right) - p_\V \left(\dot{\widetilde{u}}_{\tau_j,\eps_j}(r),{-}\eps_j^2 \mathbb{M}\ddot{\widetilde{u}}_{\tau_j,\eps_j}(r) {-} \widetilde{\xi}_{\tau_j,\eps_j}(r)\right) \right]\,\mathrm{d}r,
\end{equation*}
vanishes as $j\to+\infty$. Denoting the above integral by $I_j$, by the explicit expression \eqref{eq:contactpotexpl} of $p_\V$ we note that
\begin{equation*}
\begin{split}
|I_j| & \leq  \int_{\pi_{\tau_j}(t_j^-)}^{\pi_{\tau_j}(t_j^+)} |\dot{\widetilde{u}}_{\tau_j,\eps_j}(r)|_\V \biggl|\dist_{\V^{-1}}(-\eps_j^2 \mathbb{M}\ddot{\widetilde{u}}_{\tau_j,\eps_j}(r) - \overline{\xi}_{\tau_j,\eps_j}(r);\partial^Z\mc R(0)) \\
& \qquad \,\,\,\, \,\,\,\,\,\,\,\,\,\,\,\,\,\,  - \dist_{\V^{-1}}(-\eps_j^2 \mathbb{M}\ddot{\widetilde{u}}_{\tau_j,\eps_j}(r) - \widetilde{\xi}_{\tau_j,\eps_j}(r);\partial^Z\mc R(0))\biggr|\,\mathrm{d}r \\
& \leq\! C\!\! \int_{\pi_{\tau_j}(t_j^-)}^{\pi_{\tau_j}(t_j^+)}\!\!\! |\dot{\widetilde{u}}_{\tau_j,\eps_j}(r)|_U |\overline{\xi}_{\tau_j,\eps_j}(r) {-} \widetilde{\xi}_{\tau_j,\eps_j}(r)|_{U^*}\,\mathrm{d}r\! \leq\! C\!\!\int_{\pi_{\tau_j}(t_j^-)}^{\pi_{\tau_j}(t_j^+)}\!\!\! |\dot{\widetilde{u}}_{\tau_j,\eps_j}(r)|_U |\overline{u}_{\tau_j,\eps_j}(r) {-} \widetilde{u}_{\tau_j,\eps_j}(r)|_{U}\,\mathrm{d}r  \\
& \leq C\|\dot{\widetilde{u}}_{\tau_j,\eps_j}\|_{L^2(0,T;U)} \|\overline{u}_{\tau_j,\eps_j}- \widetilde{u}_{\tau_j,\eps_j}\|_{L^2(0,T;U)} .
\end{split}
\end{equation*}
Observing that by Lemma~\ref{lemma:mismatch} we have
\begin{equation*}
 \|\dot{\widetilde{u}}_{\tau_j,\varepsilon_j}\|_{L^2(0,T;U)} \le \|\dot{\widetilde{u}}_{\tau_j,\varepsilon_j}-\dot{\widehat{u}}_{\tau_j,\varepsilon_j}\|_{L^2(0,T;U)} + \|\dot{\widehat{u}}_{\tau_j,\varepsilon_j}\|_{L^2(0,T;U)}\le \frac{C}{\sqrt{\eps_j}},
\end{equation*}
and
\begin{equation*}
\|\widetilde{u}_{\tau_j,\varepsilon_j}-\overline{u}_{\tau_j,\varepsilon_j}\|_{L^2(0,T;U)}\le C\frac{\tau_j}{\sqrt{\eps_j}} \,,
\end{equation*}
we deduce that $|I_j|\le C \frac{\tau_j}{\eps_j}\xrightarrow[j\to +\infty]{}0$, and so \eqref{eq:tilde} is proved. 

We now set
\begin{equation*}
\sigma_j:= \frac{\pi_{\tau_j}(t_j^+) - \pi_{\tau_j}(t_j^-)}{\eps_j}\,,
\end{equation*}
which is strictly positive for $j$ large enough by virtue of \eqref{eq:star}, and, correspondingly,
\begin{equation*}
v_j(s):= \widetilde{u}_{\tau_j,\eps_j} (\eps_j s+\pi_{\tau_j}(t_j^-))\,, \quad \xi_j(s):=\widetilde{\xi}_{\tau_j,\eps_j}(\eps_j s+\pi_{\tau_j}(t_j^-)), \quad \mbox{ for }s\in[0,\sigma_j]\,,
\end{equation*} 
so that by a change of variable one has
\begin{equation*}
\begin{split}
\mu(\{t\}) \geq \limsup_{j\to+\infty} \int_{0}^{\sigma_j} p_\V \left(\dot{{v}}_{j}(s),- \mathbb{M}\ddot{{v}}_{j}(r) - {\xi}_{j}(s)\right)\,\mathrm{d}s \,. 
\end{split}
\end{equation*}
We now observe that
\begin{itemize}
\item $v_j\in W^{2,\infty}(0,\sigma_j;U)\subseteq G(0,\sigma_j)$ by Remark~\ref{rmk:V=U};
\item $\xi_j\in L^\infty(0,\sigma_j;U^*)$ since by \ref{hyp:E6} we have $\widetilde{\xi}_{\tau_j,\eps_j}\in L^\infty(0,T;U^*)$; 
\item the bound
\begin{equation*}
\begin{split}
\int_{0}^{\sigma_j} \|\dot{v}_j(s)\|_Z \,\mathrm{d}s = \int_{\pi_{\tau_j}(t_j^-)}^{\pi_{\tau_j}(t_j^+)} \|\dot{\widetilde{u}}_{\tau_j,\eps_j}(r)\|_Z \,\mathrm{d}r & \le \int_{\pi_{\tau_j}(t_j^-)}^{\pi_{\tau_j}(t_j^+)} \|\dot{\widehat{u}}_{\tau_j,\eps_j}(r)\|_Z \,\mathrm{d}r + C \frac{\tau_j}{\eps_j} \\
& \le \widetilde{C} + C \frac{\tau_j}{\eps_j} \le \frac{3}{2} \widetilde{C},
\end{split}
\end{equation*}
holds for $j$ large enough, where in the last inequality we exploited \eqref{eq:nodeineq1};  
\item by using \ref{hyp:E9}, there holds
\begin{equation*}
\begin{split}
&\quad\,\xi_j(s) \in  \partial^U\mc E(\pi_{\tau_j}(\eps_j s+\pi_{\tau_j}(t_j^-)), v_j(s)) \subseteq \partial^U\mc E(t, v_j(s)) + \overline{B}^{Z^*}_{\widetilde{\omega}_{\widetilde{C}}(|t-\pi_{\tau_j}(\eps_j s+\pi_{\tau_j}(t_j^-))|)} \\
& \subseteq \partial^U\mc E(t, v_j(s)) + \overline{B}^{Z^*}_{\widetilde{\omega}_{\widetilde{C}}(|t-\pi_{\tau_j}(t_j^-)| \vee |t-\pi_{\tau_j}(t_j^+)|)} \subseteq \partial^U\mc E(t, v_j(s)) + \overline{B}^{Z^*}_{\alpha},
\end{split}
\end{equation*}
for $j$ large enough. 
\end{itemize}
If we prove that
\begin{subequations}	
	\begin{gather}
v_j(0)\xrightarrow[j\to+\infty]{U}u^-(t)\,,\,\,\,\, v_j(\sigma_j)\xrightarrow[j\to+\infty]{U}u^+(t) \,, \label{eq:limit2}\\
 \dot{v}_j(0)\,,\, \dot{v}_j(\sigma_j)\xrightarrow[j\to+\infty]{U}0 \,, \label{eq:limit1}
\end{gather}
\end{subequations}
then we conclude by arguing exactly as in the proof of Proposition~\ref{prop:mugec0}. 

First we notice that $\dot{v}_j(0) = \eps_j \dot{\widetilde{u}}_{\tau_j,\eps_j}(\pi_{\tau_j}(t_j^-))$ and $\dot{v}_j(\sigma_j) = \eps_j \dot{\widetilde{u}}_{\tau_j,\eps_j}(\pi_{\tau_j}(t_j^+))$. Then, since by construction there holds $\dot{\widetilde{u}}_{\tau_j,\eps_j}(\pi_{\tau_j}(t_j^\pm)) = \dot{\widehat{u}}_{\tau_j,\eps_j}(t_j^\pm)$, by \eqref{eq:vanisheq1} we deduce \eqref{eq:limit1}. Second, we observe that ${v}_j(0) = {\widetilde{u}}_{\tau_j,\eps_j}(\pi_{\tau_j}(t_j^-))$, ${v}_j(\sigma_j) = {\widetilde{u}}_{\tau_j,\eps_j}(\pi_{\tau_j}(t_j^+))$, so, taking into account \eqref{eq:vanisheq2}, in order to prove \eqref{eq:limit2} it is enough to show that 
\begin{equation*}
|\widetilde{u}_{\tau_j,\eps_j}(\pi_{\tau_j}(t_j^\pm)) - \widetilde{u}_{\tau_j,\eps_j}(t_j^\pm) |_U \xrightarrow[j\to+\infty]{}0\,.
\end{equation*}

For this we note that, if $t^{k_j^\pm}:=\pi_{\tau_j}(t_j^\pm)$ for some $k_j^\pm\in\mathcal{K}_{\tau_j}$, then there holds
\begin{equation*}
\begin{split}
&\quad\,|\widetilde{u}_{\tau_j,\eps_j}(\pi_{\tau_j}(t_j^\pm)) - \widetilde{u}_{\tau_j,\eps_j}(t_j^\pm) |_U^2  \le \tau_j \int_{t^{k_j^\pm-1}}^{t^{k_j^\pm}} |\dot{\widetilde{u}}_{\tau_j,\eps_j}(r)|_U^2\,\mathrm{d}r \\
& = \tau_j \int_{t^{k_j^\pm-1}}^{t^{k_j^\pm}} \left|\frac{v^{k_j^\pm}_{\tau_j,\eps_j}}{\tau_j}(r-t^{k_j^\pm-1}) + \frac{v^{k_j^\pm-1}_{\tau_j,\eps_j}}{\tau_j}(\tau_j-(r-t^{k_j^\pm-1}))\right|_U^2\,\mathrm{d}r  \le C \tau_j^2 (|v^{k_j^\pm}_{\tau_j,\eps_j}|_U^2 + |v^{k_j^\pm-1}_{\tau_j,\eps_j}|_U^2)\\
&\le C \tau_j^2 \left(|u_1^{\eps_j}|_U^2 + \sum_{k=1}^{T/\tau_j}|v^{k}_{\tau_j,\eps_j}|_U^2\right)  \le C \frac{\tau_j}{\eps_j} \left(\frac{\tau_j}{\eps_j}+1\right)\xrightarrow[j\to +\infty]{}0\,.
\end{split}
\end{equation*}
This concludes the proof. 
\end{proof}

\appendix
\section{Existence of exact dissipative dynamic solutions in the case $V=U$}\label{App:Existence}
 This first appendix is devoted to the proof of the following theorem.
	\begin{thm}\label{thm:existence}
		Assume \eqref{eq:embeddings} with $V=U$. Let the mass and viscosity operators $\mathbb{M}$ and $\mathbb{V}$ satisfy \eqref{mass} and \eqref{viscosity}. Let the rate-independent dissipation potential $\mc R$ satisfy \ref{hyp:R1} and let the potential energy $\mc E$ satisfy \ref{hyp:E1}-\ref{hyp:E6}. Then for any $\eps>0$ and for any initial data $u_0^\eps\in D$, $u_1^\eps\in U$ there exists an exact dissipative dynamic solution $u^\eps$ to \eqref{mainprob} in the sense of Definition~\ref{def:dynsol}.
	\end{thm}
	We will make use of the following well-known lemmas, whose proof can be found in \cite[Ch. XVIII, §5, Lemma~6]{DL} and \cite{Simon}, respectively.
		\begin{lemma}\label{lem:lemma5}
			Let $X,Y$ be Banach spaces such that $X\hookrightarrow Y$. If $X$ is reflexive, then $$L^\infty(0,T;X)\cap C_{\rm w}([0,T];Y) \subset C_{\rm w}([0,T];X).$$
		\end{lemma}
	
	\begin{lemma}[\textbf{Aubin-Lions}]\label{AubinLions}
		Let $X,Y,Z$ be Banach spaces such that $X\hookrightarrow\hookrightarrow Y\hookrightarrow Z$. Then
		\begin{itemize}
			\item[(1)] for $p\in [1,+\infty)$ and $q\in [1,+\infty]$ there holds
			\begin{equation*}
				\{u\in L^p(0,T;X):\, \dot{u}\in L^q(0,T;Z)\}\hookrightarrow\hookrightarrow L^p(0,T;Y);
			\end{equation*}
		\item[(2)] for $q\in (1,+\infty]$ there holds
		\begin{equation*}
			\{u\in L^\infty(0,T;X):\, \dot{u}\in L^q(0,T;Z)\}\hookrightarrow\hookrightarrow C([0,T];Y).
		\end{equation*}
		\end{itemize}
	\end{lemma}

Let now $\widetilde{u}_{\tau,\eps}, \widehat{u}_{\tau,\eps}, \overline{u}_{\tau,\eps}, \underline{u}_{\tau,\eps}, \overline{\xi}_{\tau,\eps}, \overline{\eta}_{\tau,\varepsilon}$ be the interpolants of the discrete scheme \eqref{schemeincond} defined in \eqref{interpolants}, \eqref{eq:xitaueps} and \eqref{eq:etataueps}. By the uniform bounds of Proposition~\ref{prop:inequality0}, recalling that we are assuming $V=U$ and that $\eps$ in this appendix is considered fixed, we know that
\begin{itemize}
	\item $\mc E(\pi_\tau(\cdot),\overline{u}_{\tau,\eps})$ is uniformly bounded;
	\item $ \widehat{u}_{\tau,\eps}, \overline{u}_{\tau,\eps}, \underline{u}_{\tau,\eps}$ are uniformly bounded in $B([0,T];U)\cap BV([0,T];Z)$;
	\item $ \widehat{u}_{\tau,\eps}$ is uniformly bounded in $W^{1,2}(0,T;U)\cap W^{1,\infty}(0,T;W)$;
	\item $ \widetilde{u}_{\tau,\eps}$ is uniformly bounded in $W^{2,2}(0,T;U^*)$;
	\item $ \overline{\xi}_{\tau,\eps}$ is uniformly bounded in $L^{\infty}(0,T;U^*)$;
	\item $ \overline{\eta}_{\tau,\eps}$ is uniformly bounded in $L^{\infty}(0,T;Z^*)$.
\end{itemize}

By standard compactness arguments, together with Lemmas~\ref{helly} and \ref{lemma:mismatch}, we deduce that, up to a subsequence $\tau_j\to0$, there hold
\begin{itemize}
	\item $\widehat{u}_{\tau_j,\eps}(t), \overline{u}_{\tau_j,\eps}(t), \underline{u}_{\tau_j,\eps}(t)\xrightharpoonup[j\to +\infty]{U}u^\eps(t), \quad$ for all $t\in [0,T]$,
	\item $\widehat{u}_{\tau_j,\eps}\xrightharpoonup[j\to +\infty]{W^{1,2}(0,T;U)}u^\eps,\quad$ and $\quad\widehat{u}_{\tau_j,\eps}\xrightharpoonup[j\to +\infty]{W^{1,\infty}(0,T;W)\,\,\ast}u^\eps$,
	\item $\widetilde{u}_{\tau_j,\eps}\xrightharpoonup[j\to +\infty]{W^{2,2}(0,T;U^*)}u^\eps$,
	\item $\overline{\xi}_{\tau_j,\eps}\xrightharpoonup[j\to +\infty]{L^{\infty}(0,T;U^*)\,\,\ast}\xi^\eps,\quad$ and $\quad\overline{\eta}_{\tau_j,\eps}\xrightharpoonup[j\to +\infty]{L^{\infty}(0,T;Z^*)\,\,\ast}\eta^\eps$,
\end{itemize}
for some functions $u^\eps\in W^{1,2}(0,T;U)\cap W^{1,\infty}(0,T;W)\cap W^{2,2}(0,T;U^*)$, $\xi^\eps\in L^\infty(0,T;U^*)$ and $\eta^\eps\in L^\infty(0,T;Z^*).$ By using Lemma~\ref{lem:lemma5} we also deduce $\dot{u}^\eps\in C_{\rm w}([0,T];W)$, while by \ref{hyp:E1} we obtain that $\mc E(\cdot,u^\eps(\cdot))$ is bounded. In particular, $u^\eps$ is $D$-valued, hence the first request in Definition~\ref{def:dynsol} is satisfied.

Passing to the limit in \eqref{eq:EL2}, we now infer that $u^\eps, \xi^\eps$ and $\eta^\eps$ solve the Cauchy problem

\begin{equation}\label{eq:eqlimit}
	\begin{cases}
		\eps^2 \mathbb{M}\xepsdd(t)+\eps \mathbb{V}\xepsd(t)+\eta^\eps(t)+\xi^\eps(t)= 0, \quad\text{in $U^*$,}\quad \text{for a.e. }t\in [0,T],\\
		\xeps(0)=u^\eps_0,\quad\xepsd(0)=u^\eps_1.
	\end{cases}
\end{equation}

We conclude the proof of Theorem~\ref{thm:existence} once we show that
\begin{equation}\label{eq:claimappendix}
	\eta^\eps(t)\in\partial^Z\mc R(\dot{u}^\eps(t)),\quad\text{and}\quad\xi^\eps(t)\in\partial^U\mc E(t,u^\eps(t)),\quad\text{for almost every }t\in[0,T].
\end{equation}
Indeed, by using Lemma~\ref{lemma:crappendix}, the above inclusions will yield \eqref{eq:energybalance} as an equality for all $t\in [0,T]$.

The proof of \eqref{eq:claimappendix} will rely on the following two lemmas.
\begin{lemma}\label{lemma:app1}
	For almost every $t\in [0,T]$ there holds
	\begin{equation*}
		\lim_{j\to +\infty}\int_{0}^{\pi_{\tau_j}(t)}\langle\overline{\xi}_{\tau_j,\eps}(r),\overline{u}_{\tau_j,\eps}(r)\rangle_U\d r=\int_0^t\langle\xi^\eps(r),u^\eps(r)\rangle_U \d r.
	\end{equation*}
\end{lemma}
\begin{lemma}\label{lemma:app2}
	For all $0\le a\le b\le T$ there holds
	\begin{equation*}
		\limsup_{j\to +\infty}\int_a^b\int_{0}^{\pi_{\tau_j}(t)}\langle \overline{\eta}_{\tau_j,\eps}(r), \dot{\widehat{u}}_{\tau_j,\eps}(r)\rangle_Z\d r\d t\le\int_a^b\int_0^t\langle\eta^\eps(r),\dot{u}^\eps(r)\rangle_Z\d r\d t.
	\end{equation*}
\end{lemma}

\begin{proof}[Proof of \eqref{eq:claimappendix}]
	We start proving that $\xi^\eps(t)\in\partial^U\mc E(t,u^\eps(t))$ for a.e. $t\in [0,T]$. 
	
	Since $\overline{\xi}_{\tau_j,\eps}(r)\in\partial^U\mc E(\pi_{\tau_j}(r),\overline{u}_{\tau_j,\eps}(r))$, by \eqref{lambdaconvexineq} we obtain that for all $u\in D$ there holds
	\begin{equation}\label{eq:first}
		\begin{aligned}
			&\quad\,\int_{\pi_{\tau_j}(s)}^{\pi_{\tau_j}(t)}\mc E(\pi_{\tau_j}(r), u)\d r\\
			&\ge \int_{\pi_{\tau_j}(s)}^{\pi_{\tau_j}(t)}\left(\mc E(\pi_{\tau_j}(r), \overline{u}_{\tau_j,\eps}(r))+\langle\overline{\xi}_{\tau_j,\eps}(r), u-\overline{u}_{\tau_j,\eps}(r)\rangle_U-\frac\lambda 2|u-\overline{u}_{\tau_j,\eps}(r)|^2_W\right)\d r,
		\end{aligned}
	\end{equation}
for all $0\le s\le t\le T$.

By using Fatou's Lemma, Lemma~\ref{lemma:app1} and Lebesgue Dominated Convergence Theorem on the first, second and third term in the right-hand side of \eqref{eq:first}, respectively, letting $j\to +\infty$ we deduce
\begin{equation*}
	\int_s^t \mc E(r,u)\d r\ge \int_s^t\left(\mc E(r,u^\eps(r))+\langle\xi^\eps(r), u-u^\eps(r)\rangle_U-\frac\lambda 2|u-u^\eps(r)|_W^2\right)\d r,
\end{equation*}
for almost every $0\le s\le t\le T$. By the arbitrariness of $s,t$ and $u\in D$ we conclude that $\xi^\eps(t)\in\partial^U\mc E(t,u^\eps(t))$ for a.e. $t\in [0,T]$.

Let us now show that $\eta^\eps(t)\in\partial^Z\mc R(\dot{u}^\eps(t))$ for a.e. $t\in [0,T]$. It is easy to see that $\eta^\eps(t)\in\partial^Z\mc R(0)$ for a.e. $t\in [0,T]$, whence by \eqref{2.2mielke} it is enough to show that 
\begin{equation}\label{eq:second}
	\mc R(\dot{u}^\eps(t))=\langle\eta^\eps(t),\dot{u}^\eps(t)\rangle_Z,\qquad\text{for a.e. }t\in [0,T],
\end{equation}
in order to conclude.

Since $\overline{\eta}_{\tau_j,\eps}(r)\in\partial^Z\mc R(\dot{\widehat{u}}_{\tau_j,\eps}(r))$, for all $0\le a\le b\le T$ we obtain
\begin{equation*}
	\int_a^b\int_{0}^{\pi_{\tau_j}(t)}\mc R(\dot{\widehat{u}}_{\tau_j,\eps}(r))\d r\d t=\int_a^b\int_{0}^{\pi_{\tau_j}(t)}\langle \overline{\eta}_{\tau_j,\eps}(r), \dot{\widehat{u}}_{\tau_j,\eps}(r)\rangle_Z\d r\d t.
\end{equation*}
By weak lower-semicontinuity on the left-hand side, together with Lemma~\ref{lemma:app2} for the right-hand side, letting $j\to +\infty$ we deduce
\begin{equation*}
	\int_a^b\int_0^t\mc R(\dot{u}^\eps(r))\d r\d t\le\int_a^b\int_0^t\langle\eta^\eps(r),\dot{u}^\eps(r)\rangle_Z\d r\d t.
\end{equation*}

Note that the opposite inequality is granted by the fact that $\eta^\eps(r)\in\partial^Z\mc R(0)$. Thus, by the arbitrariness of $a,b$ and by the continuity of the integrands we obtain
\begin{equation*}
	\int_0^t\mc R(\dot{u}^\eps(r))\d r=\int_0^t\langle\eta^\eps(r),\dot{u}^\eps(r)\rangle_Z\d r,\qquad\text{for every }t\in [0,T],
\end{equation*}
whence \eqref{eq:second}. The proof of \eqref{eq:claimappendix} is hence concluded.
\end{proof}
\begin{proof}[Proof of Lemma~\ref{lemma:app1}]
	We shall prove separately
	\begin{subequations}
		\begin{equation}\label{(alfa)}
				\liminf_{j\to +\infty}\int_{0}^{\pi_{\tau_j}(t)}\langle\overline{\xi}_{\tau_j,\eps}(r),\overline{u}_{\tau_j,\eps}(r)\rangle_U\d r\ge\int_0^t\langle\xi^\eps(r),u^\eps(r)\rangle_U \d r;
		\end{equation}
	\begin{equation}\label{(beta)}
			\limsup_{j\to +\infty}\int_{0}^{\pi_{\tau_j}(t)}\langle\overline{\xi}_{\tau_j,\eps}(r),\overline{u}_{\tau_j,\eps}(r)\rangle_U\d r\le\int_0^t\langle\xi^\eps(r),u^\eps(r)\rangle_U \d r.
	\end{equation}
	\end{subequations}
By \eqref{lambdaconvexineq} we infer that
\begin{equation*}
	\begin{split}
		&\quad\,\int_{0}^{\pi_{\tau_j}(t)}\langle\overline{\xi}_{\tau_j,\eps}(r),\overline{u}_{\tau_j,\eps}(r)\rangle_U\d r\\
		&\ge \int_{0}^{\pi_{\tau_j}(t)}\left(\mc E(\pi_{\tau_j}(r),\overline{u}_{\tau_j,\eps}(r))-\mc E(\pi_{\tau_j}(r),u^\eps(r))+\langle\overline{\xi}_{\tau_j,\eps}(r),u^\eps(r)\rangle_U-\frac \lambda 2|u^\eps(r)-\overline{u}_{\tau_j,\eps}(r)|_W^2\right)\d r.
	\end{split}
\end{equation*}
Recalling that $\overline{\xi}_{\tau_j,\eps}\xrightharpoonup[j\to +\infty]{L^{\infty}(0,T;U^*)\,\,\ast}\xi^\eps$, $\overline{u}_{\tau_j,\eps}\xrightarrow[j\to +\infty]{W}u^\eps$ pointwise and that $\pi_{\tau_j}(t)\ge t$, by using Fatou's Lemma on the first term of the right-hand side above, and Lebesgue Dominated Convergence on the remaining terms, by letting $j\to +\infty$ we deduce \eqref{(alfa)}.

The proof of \eqref{(beta)} is much more involved. We first test \eqref{eq:EL2} with $\overline{{u}}_{\tau_j,\eps}$, obtaining
\begin{equation*}
	\begin{split}
		\int_0^{\pi_{\tau_j}(t)}\langle\overline{\xi}_{\tau_j,\eps}(r),\overline{{u}}_{\tau_j,\eps}(r)\rangle_U \d r=&{-}\eps^2\int_0^{\pi_{\tau_j}(t)}\langle\M\ddot{\widetilde{u}}_{\tau_j,\eps}(r),\overline{{u}}_{\tau_j,\eps}(r)\rangle_U\d r\\&{-}\eps\int_0^{\pi_{\tau_j}(t)}\langle\V\dot{\widehat{u}}_{\tau_j,\eps}(r),\overline{{u}}_{\tau_j,\eps}(r)\rangle_U\d r{-}\int_0^{\pi_{\tau_j}(t)}\langle\overline{\eta}_{\tau_j,\eps}(r),\overline{{u}}_{\tau_j,\eps}(r)\rangle_Z\d r.
	\end{split}
\end{equation*}
Notice that $\overline{u}_{\tau_j,\eps}\xrightarrow[j\to +\infty]{W}u^\eps$ pointwise and $\overline{u}_{\tau_j,\eps}$ is bounded in $B([0,T];W)$; thus by Dominated Convergence Theorem there holds $\overline{u}_{\tau_j,\eps}\xrightarrow[j\to +\infty]{L^1(0,T;Z)}u^\eps$. Recalling that $\overline{\eta}_{\tau_j,\eps}\xrightharpoonup[j\to +\infty]{L^{\infty}(0,T;Z^*)\,\,\ast}\eta^\eps$, the last term above converges to $-\int_0^t\langle\eta^\eps(r),u^\eps(r)\rangle_Z\d r$.

As regards the second term, setting $t^n:=\pi_{\tau_j}(t)$, we observe that
\begin{equation*}
	\begin{split}
		{-}\int_0^{\pi_{\tau_j}(t)}\langle\V\dot{\widehat{u}}_{\tau_j,\eps}(r),\overline{{u}}_{\tau_j,\eps}(r)\rangle_U\d r&={-}\sum_{k=1}^n\langle\V(u^k_{\tau_j,\eps}{-}u^{k-1}_{\tau_j,\eps}), u^k_{\tau_j,\eps}\rangle_U\le {-}\sum_{k=1}^n\frac 12|u^k_{\tau_j,\eps}|_\V^2{-}\frac 12|u^{k-1}_{\tau_j,\eps}|_\V^2\\
		&=\frac 1 2|u_0^\eps|_\V^2{-}\frac 1 2|u^{n}_{\tau_j,\eps}|_\V^2 =\frac 1 2|u_0^\eps|_\V^2{-}\frac 1 2|\overline{{u}}_{\tau_j,\eps}(t)|_\V^2 .
	\end{split}
\end{equation*}
By weak lower semicontinuity of the norm we hence deduce
\begin{equation*}
	\limsup_{j\to +\infty}{-}\eps\int_0^{\pi_{\tau_j}(t)}\langle\V\dot{\widehat{u}}_{\tau_j,\eps}(r),\overline{{u}}_{\tau_j,\eps}(r)\rangle_U\d r\le \frac \eps 2|u_0^\eps|_\V^2{-}\frac \eps 2|u^\eps(t)|_\V^2={-}\eps\int_{0}^{t}\langle\V\dot{u}^\eps(r),u^\eps(r)\rangle_U\d r.
\end{equation*}

We now claim that
\begin{equation}\label{(claimapp)}
	\lim\limits_{j\to +\infty}\int_0^{\pi_{\tau_j}(t)}\langle\M\ddot{\widetilde{u}}_{\tau_j,\eps}(r),\overline{{u}}_{\tau_j,\eps}(r)\rangle_U\d r=\int_0^t\langle\M \ddot{u}^\eps(r), u^\eps(r)\rangle_U\d r,\qquad\text{for a.e. }t\in [0,T].
\end{equation}
Once the claim is proved we conclude the proof of \eqref{(beta)} by collecting the just obtained convergence results and using the equation \eqref{eq:eqlimit}.

In order to show \eqref{(claimapp)} we first recall that $\dot{\widetilde{u}}_{\tau,\eps}$ is bounded in $W^{1,2}(0,T;U^*)$ and in $L^2(0,T;U)$ by \eqref{eq:mism5} and \eqref{eq:mism6}. We notice that it is also bounded in $L^\infty(0,T;W)$; indeed by definition one has
\begin{equation*}
	\essup\limits_{t\in(0,T)}|\dot{\widetilde{u}}_{\tau,\eps}(t)|_W\le\max\limits_{k\in \mc K_\tau}(|v^k_{\tau,\eps}|_W+|v^{k-1}_{\tau,\eps}|_W)\le C,
\end{equation*}
where we used $(iii')$ in Proposition~\ref{prop:inequality0}.

By Lemma~\ref{AubinLions} we hence deduce $\dot{\widetilde{u}}_{\tau_j,\eps}\xrightarrow[j\to +\infty]{L^2(0,T;W)}\dot{u}^\eps$, which, by means of \eqref{eq:mism7}, also yields
\begin{equation}\label{(convapp)}
	\dot{\widehat{u}}_{\tau_j,\eps}\xrightarrow[j\to +\infty]{L^2(0,T;U^*)}\dot{u}^\eps.
\end{equation} 
By using Young's inequality, for all $\delta>0$ we now have
\begin{equation*}
	\int_0^T|\dot{\widehat{u}}_{\tau_j,\eps}(r)-\dot{u}^\eps(r)|^2_W\d r\le \delta\int_0^T|\dot{\widehat{u}}_{\tau_j,\eps}(r)-\dot{u}^\eps(r)|^2_U\d r+\frac 1\delta \int_0^T|\dot{\widehat{u}}_{\tau_j,\eps}(r)-\dot{u}^\eps(r)|^2_{U^*}\d r.
\end{equation*}
Hence, by letting first $j\to +\infty$ and then $\delta\to 0$, from \eqref{(convapp)} we finally obtain $\dot{\widehat{u}}_{\tau_j,\eps}\xrightarrow[j\to +\infty]{L^2(0,T;W)}\dot{u}^\eps$, and so, up to possibly extracting a further subsequence, one has
\begin{equation}\label{ptwsconv}
	\dot{\widehat{u}}_{\tau_j,\eps}(t)\xrightarrow[j\to +\infty]{W}\dot{u}^\eps(t),\qquad\text{for almost every } t\in [0,T].
\end{equation} 
By a summation by parts we now have
\begin{equation*}
	\begin{split}
		\int_0^{\pi_{\tau_j}(t)}\langle\M\ddot{\widetilde{u}}_{\tau_j,\eps}(r),\overline{{u}}_{\tau_j,\eps}(r)\rangle_U\d r&=\langle\M \dot{\widehat{u}}_{\tau_j,\eps}(t), \overline{u}_{\tau_j,\eps}(t)\rangle_W-\langle \M u_1^\eps, u_0^\eps\rangle_W\\
		&-\int_0^{\pi_{\tau_j}(t)}\langle\M\dot{\widehat{u}}_{\tau_j,\eps}(r) ,\dot{\widehat{u}}_{\tau_j,\eps}(r-\tau_j)\rangle_W\d r.
	\end{split}
\end{equation*}
So, by means of \eqref{ptwsconv}, for almost every $t\in [0,T]$ we finally deduce
\begin{equation*}
	\begin{split}
		\lim\limits_{j\to +\infty}\int_0^{\pi_{\tau_j}(t)}\langle\M\ddot{\widetilde{u}}_{\tau_j,\eps}(r),\overline{{u}}_{\tau_j,\eps}(r)\rangle_U\d r&=\langle \M\dot{u}^\eps(t),u^\eps(t)\rangle_W{-}\langle \M u_1^\eps, u_0^\eps\rangle_W{-}\int_0^t\langle\M \dot{u}^\eps(r),\dot{u}^\eps(r)\rangle_W\d r\\
		&= \int_0^t\langle\M \ddot{u}^\eps(r), u^\eps(r)\rangle_U\d r,
	\end{split}
\end{equation*}
and the proof of \eqref{(claimapp)} is complete.
\end{proof}
\begin{proof}[Proof of Lemma~\ref{lemma:app2}]
	By testing \eqref{eq:EL2} with $\dot{\widehat{u}}_{\tau_j,\eps}$, for every $t\in [0,T]$ we deduce
	\begin{equation}\label{(1)}
		\begin{split}
		&\quad\,-\int_0^{\pi_{\tau_j}(t)}\langle\overline{\eta}_{\tau_j,\eps}(r),\dot{\widehat{u}}_{\tau_j,\eps}(r)\rangle_Z\d r\\
		&=\eps^2\int_0^{\pi_{\tau_j}(t)}\langle\M\ddot{\widetilde{u}}_{\tau_j,\eps}(r),\dot{\widehat{u}}_{\tau_j,\eps}(r)\rangle_U\d r+\int_0^{\pi_{\tau_j}(t)}\langle\overline{\xi}_{\tau_j,\eps}(r),\dot{\widehat{u}}_{\tau_j,\eps}(r)\rangle_U \d r+\eps\int_0^{\pi_{\tau_j}(t)}|\dot{\widehat{u}}_{\tau_j,\eps}(r)|^2_\V\d r.
	\end{split}
	\end{equation}
Observe that, denoting $t^n:=\pi_{\tau_j}(t)$, one has 
\begin{equation}\label{(2)}
	\begin{split}
	\int_0^{\pi_{\tau_j}(t)}\langle\M\ddot{\widetilde{u}}_{\tau_j,\eps}(r),\dot{\widehat{u}}_{\tau_j,\eps}(r)\rangle_U\d r&=\sum_{k=1}^n\langle\M(v^k_{\tau_j,\eps}-v^{k-1}_{\tau_j,\eps}), v^k_{\tau_j,\eps}\rangle_W\ge \frac 12|v^n_{\tau_j,\eps}|_\M^2-\frac 12|u_1^\eps|_\M^2\\
	&=\frac 12|\dot{\widehat{u}}_{\tau_j,\eps}(t)|_\M^2-\frac 12|u_1^\eps|_\M^2,
\end{split}
\end{equation}
and that by \eqref{lambdaconvexineq} there also holds
\begin{equation}\label{(3)}
	\begin{split}
		&\quad\,\int_0^{\pi_{\tau_j}(t)}\langle\overline{\xi}_{\tau_j,\eps}(r),\dot{\widehat{u}}_{\tau_j,\eps}(r)\rangle_U \d r=\sum_{k=1}^n\langle\xi^k_{\tau_j,\eps},u^k_{\tau_j,\eps}-u^{k-1}_{\tau_j,\eps} \rangle_U\\
		&\ge \sum_{k=1}^n \mc E(t^k,u^k_{\tau_j,\eps})-\mc E(t^k,u^{k-1}_{\tau_j,\eps})-\frac{\lambda}{2}|u^k_{\tau_j,\eps}-u^{k-1}_{\tau_j,\eps}|_W^2\\
		&=\mc E(\pi_{\tau_j}(t), \overline{u}_{\tau_j,\eps}(t))-\mc E(0,u_0^\eps)-\int_0^{\pi_{\tau_j}(t)}\partial_t\mc E(r, \underline{u}_{\tau_j,\eps}(r))\d r-\frac{\lambda\tau_j}{2}\int_0^{\pi_{\tau_j}(t)}|\dot{\widehat{u}}_{\tau_j,\eps}(r)|_W^2\d r.
	\end{split}
\end{equation}
By plugging \eqref{(2)} and \eqref{(3)} into \eqref{(1)}, integrating between $a$ and $b$, after letting $j\to +\infty$ we obtain
\begin{equation*}
	\begin{split}
	&\quad\,\liminf_{j\to +\infty}-\int_a^b\int_{0}^{\pi_{\tau_j}(t)}\langle \overline{\eta}_{\tau_j,\eps}(r), \dot{\widehat{u}}_{\tau_j,\eps}(r)\rangle_Z\d r\d t\\
	&\ge \int_a^b\left(\frac{\eps^2}{2}|\dot{u}^\eps(t)|_\M^2{-}\frac{\eps^2}{2}|u_1^\eps|_\M^2{+}\mc E(t,u^\eps(t)){-}\mc E(0,u_0^\eps){-}\int_0^t\partial_t\mc E(r,u^\eps(r))\d r{+}\eps\int_0^t|\dot{u}^\eps(r)|^2_\V \d r\right)\d t\\
	&=-\int_a^b\int_0^t\langle\eta^\eps(r),\dot{u}^\eps(r)\rangle_Z\d r\d t,
\end{split}
\end{equation*}
where in the last equality we exploited Lemma~\ref{lemma:crappendix} together with equation \eqref{eq:eqlimit}. Thus we conclude.
\end{proof}

\section{Nonsmooth chain-rule formulas}\label{App:Chainrule}
In this second appendix we collect some useful nonsmooth chain-rule formulas we employed throughout the paper. Here we tacitly assume the validity of \eqref{eq:embeddings} and \eqref{mass}.

	\begin{lemma}\label{lemma:chainruleequality}
		Assume $V=U$ together with \ref{hyp:E1},\ref{hyp:E2},\ref{hyp:E5} and \ref{hyp:E6}. Then for all $t\in [0,T]$ and for all $v\in W^{1,2}(a,b;U)\cap W^{2,2}(a,b;U^*)$ such that $\sup\limits_{r\in [a,b]}\mc E(0,v(r))<+\infty$ and $\partial^U\mc E(t,v(r))\neq\emptyset$ for almost every $r\in [a,b]$, the map $r\mapsto\frac 12 |\dot v(r)|_\M^2+\mc E(t,v(r))$ is absolutely continuous in $[a,b]$. Moreover, for every $\xi_t\in \partial^U\mc E(t,v(\cdot))$ almost everywhere in $[a,b]$ there holds
		\begin{equation}\label{eq:chainruleequality}
			\frac{\d}{\d r}\left(\frac 12 |\dot v(r)|_\M^2+\mc E(t,v(r))\right)=\langle\M\ddot{v}(r)+\xi_t(r),\dot v(r)\rangle_U,\qquad\text{for a.e. }r\in [a,b].
		\end{equation}
	\end{lemma}
	\begin{proof}
		We first note that for any function $v\in W^{1,2}(a,b;U)\cap W^{2,2}(a,b;U^*)$ it is well known that $\frac 12 |\dot v(\cdot)|_\M^2$ is absolutely continuous with
		\begin{equation}\label{eq:derin}
			\frac{\d}{\d r}\frac 12 |\dot v(r)|_\M^2=\langle\M\ddot{v}(r),\dot v(r)\rangle_U,\qquad\text{for a.e. }r\in [a,b].
		\end{equation}
		Let us now focus on $\mc E(t,v(\cdot))$. Since by assumption we have $\sup\limits_{r\in [a,b]}\mc E(0,v(r))<+\infty$, an application of \eqref{eq:ineqgronwall} yields that $\mc E(t,v(\cdot))$ is bounded in $[a,b]$ as well. Hence, by \ref{hyp:E6}, we obtain that $\xi_t$ is in $L^\infty(a,b;U^*)$.
		By \eqref{lambdaconvexineq} we thus obtain that for every $r_1\in [a,b]$ and for almost every $r_2\in [a,b]$ there holds
		\begin{align}\label{chineqq}
			&\,\mc E(t,v(r_2))-\mc E(t,v(r_1))\nonumber\\
			\le&\, \langle\xi_t(r_2), v(r_2)-v(r_1)\rangle_U+\frac\lambda 2|v(r_2)-v(r_1)|_W^2\nonumber\\
			\le&\, \|\xi_t\|_{L^\infty(a,b;U^*)}|v(r_2)-v(r_1)|_U+C\lambda(|v(r_2)|_U+|v(r_1)|_U)|v(r_2)-v(r_1)|_U\\
			\le&\,(\|\xi_t\|_{L^\infty(a,b;U^*)}+C\lambda\|v\|_{C([a,b];U)})|v(r_2)-v(r_1)|_U.\nonumber
		\end{align}
		Recalling that $W^{1,2}(a,b;U)\subseteq C([a,b];U)$, the above inequality can be extended to all $r_2\in [a,b]$ by lower semicontinuity of $\mc E(t,\cdot)$, whence we deduce the absolutely continuity of $\mc E(t,v(\cdot))$ in $[a,b]$.
		By exploiting again \eqref{lambdaconvexineq}, for almost every $r\in [a,b]$ and for all $h\in\R$ small we have 
		\begin{equation}\label{eq:chh1}
			\mc E(t, v(r+h))-\mc E(t,v(r))\ge \langle\xi_t(r), v(r+h)-v(r)\rangle_U-\frac \lambda 2|v(r+h)-v(r)|_W^2.
		\end{equation}
		By dividing the above inequality by $h>0$ and $h<0$, and by sending $h\to 0$ we conclude the proof of \eqref{eq:chainruleequality} by using also \eqref{eq:derin}.
	\end{proof}

\begin{lemma}\label{lemma:chainrule}
	Assume \ref{hyp:E1}-\ref{hyp:E3}, \ref{hyp:E4'} and \ref{hyp:E5} with $\lambda=0$. Let the Banach space $X$ in \ref{hyp:E4'} be reflexive and such that $U \hookrightarrow X\hookrightarrow Z$. Let $u\in AC([0,T];X)$ be such that $\sup\limits_{t\in [0,T]}\mc E(t,u(t))<+\infty$, and let us suppose that for almost every $t\in [0,T]$ there exists $\xi(t)\in \partial^U\mc E(t,u(t))\cap Z^*$ such that $\essup\limits_{t\in [0,T]}\|\xi(t)\|_{Z^*}<+\infty$.
	
	Then the map $t\mapsto \mc E(t,u(t))$ is absolutely continuous in $[0,T]$ and there holds
	\begin{equation}\label{eq:chain}
		\frac{\d}{\d t}\mc E(t,u(t))=\partial_t \mc E(t,u(t))+\langle\xi(t),\dot{u}(t)\rangle_Z,\quad\text{for a.e. }t\in [0,T].
	\end{equation} 
\end{lemma}

\begin{proof}
	We first observe that by conditions \ref{hyp:E2} and \ref{hyp:E3} (see Remark~\ref{rmk:gronwall}) we deduce that $\sup\limits_{t\in [0,T]}\|u(t)\|_U<+\infty$. Since $u$ is (absolutely) continuous with values in $X$ and $U$ is embedded in $X$, this uniform bound implies that $u$ also belongs to $C_{\rm w}([0,T];U)$. Moreover, since $X\hookrightarrow Z$, from $u\in AC([0,T];X)$ we infer $u\in AC([0,T];Z)$.
	
	By convexity of $\mc E(t,\cdot)$, for almost every $t\in [0,T]$ and for all $s\in [0,T]$ we also obtain
	\begin{equation*}
		\mc E(t,u(t))-\mc E(t,u(s))\le \langle\xi(t),u(t)-u(s)\rangle_Z\le \|\xi (t)\|_{Z^*}\|u(t)-u(s)\|_Z\le C\|u(t)-u(s)\|_Z.
	\end{equation*}
	By weak lower semicontinuity of $\mc E$ together with the continuity of $u$ in $Z$ we now deduce
	\begin{equation*}
		\mc E(t,u(t))-\mc E(t,u(s))\le  C\|u(t)-u(s)\|_Z,\quad\text{for all }s,t\in [0,T].
	\end{equation*}
	Hence, exploiting \ref{hyp:E2} and \eqref{eq:ineqgronwall}, for all $s,t\in [0,T]$ there holds
	\begin{align}\label{eq:chineq}
		\mc E(t,u(t))-\mc E(s,u(s))&=	\mc E(t,u(t))-\mc E(t,u(s))+	\mc E(t,u(s))-\mc E(s,u(s))\nonumber\\
		&\le C \|u(t)-u(s)\|_Z+\left|\int_{s}^{t}b(r)e^B(\mc E(s,u(s))+1)\d r\right|\\
		&\le  C\left( \|u(t)-u(s)\|_Z+\left|\int_{s}^{t}b(r)\d r\right|\right),\nonumber
	\end{align}
	where we set $B:=\|b\|_{L^1(0,T)}$. Thus, the map $t\mapsto \mc E(t,u(t))$ belongs to $AC([0,T])$.
	
	In order to prove \eqref{eq:chain} we exploit again convexity, obtaining for almost every $t\in [0,T]$ and for all $h\in \R$ small enough the following inequality:
	\begin{align}\label{eq:chh2}
		&\mc E(t+h,u(t+h))-\mc E(t,u(t))\nonumber\\
		\ge& \int_{t}^{t+h}\partial_t\mc E(r,u(t+h))\d r+\langle\xi(t),u(t+h)-u(t)\rangle_Z\\
		=&\int_{t}^{t+h}\partial_t\mc E(r,u(t))\d r+\langle\xi(t),u(t+h)-u(t)\rangle_Z+\int_{t}^{t+h}(\partial_t\mc E(r,u(t+h))-\partial_t\mc E(r,u(t)))\d r.\nonumber
	\end{align}
	By dividing the above inequality by $h>0$ and $h<0$, and by sending $h\to 0$ we conclude by exploiting \eqref{lipschitzpartialt} and the fact that $u$ is absolutely continuous with values in the reflexive Banach space $X$.	
\end{proof}

\begin{lemma}\label{lemma:crappendix}
	Assume $V=U$ together with \ref{hyp:E1},\ref{hyp:E2},\ref{hyp:E4},\ref{hyp:E5} and \ref{hyp:E6}. Let the function $u\in W^{1,2}(a,b;U)\cap W^{2,2}(a,b;U^*)$ be such that $\sup\limits_{t\in [0,T]}\mc E(t,u(t))<+\infty$ and let $\xi\in \partial^U\mc E(\cdot,u(\cdot))$ almost everywhere in $[0,T]$.
	
	Then the map $t\mapsto \frac{\eps^2}{2}|\dot{u}(t)|^2_\M+\mc E(t,u(t))$ is absolutely continuous in $[0,T]$ and there holds
		\begin{equation}\label{eq:lastchain}
		\frac{\d}{\d t}\left(\frac{\eps^2}{2}|\dot{u}(t)|^2_\M+\mc E(t,u(t))\right)=\langle\eps^2\M\ddot{u}(t)+\xi(t),\dot u(t)\rangle_U+\partial_t\mc E(t,u(t)),\quad\text{for a.e. }t\in [0,T].
	\end{equation}
\end{lemma}
\begin{proof}
	As observed in Lemma~\ref{lemma:chainruleequality}, under these assumptions the map $\frac{\eps^2}{2} |\dot u(\cdot)|_\M^2$ is absolutely continuous with
	\begin{equation*}
		\frac{\d}{\d t}\frac{\eps^2}{2} |\dot u(t)|_\M^2=\langle\eps^2\M\ddot{u}(t),\dot u(t)\rangle_U,\qquad\text{for a.e. }t\in [0,T].
	\end{equation*}
By arguing as in \eqref{chineqq} and \eqref{eq:chineq} we deduce that for all $s,t\in [0,T]$ we have
	\begin{align*}
	\mc E(t,u(t))-\mc E(s,u(s))\le (\|\xi\|_{L^\infty(0,T;U^*)}+C\lambda\|u\|_{C([0,T];U)})  |u(t)-u(s)|_U+C\left|\int_{s}^{t}b(r)\d r\right|,
\end{align*}
whence also the map $t\mapsto \mc E(t,u(t))$ is absolutely continuous.

In order to prove \eqref{eq:lastchain}, we argue as in \eqref{eq:chh1} and \eqref{eq:chh2} obtaining for almost every $t\in [0,T]$ and for all $h$ small
\begin{align*}
	&\mc E(t+h,u(t+h))-\mc E(t,u(t))\\
	\ge&\int_{t}^{t+h}\partial_t\mc E(r,u(t))\d r+\langle\xi(t),u(t+h)-u(t)\rangle_U+\int_{t}^{t+h}(\partial_t\mc E(r,u(t+h))-\partial_t\mc E(r,u(t)))\d r\\
	&\,-\frac\lambda 2|u(t+h)-u(t)|^2_W.
\end{align*}
By dividing the above inequality by $h>0$ and $h<0$, and by sending $h\to 0$ we conclude by exploiting \ref{hyp:E4}, which implies the the third term in the last line vanishes.
\end{proof}

	\bigskip

	\noindent\textbf{Acknowledgements.} The authors are members of Gruppo Nazionale per l'Analisi Matematica, la Probabilit\`a e le loro Applicazioni (GNAMPA) of INdAM, whose support through the Project \lq\lq MATERIA: Metodi Analitici nella Trattazione di Evoluzioni Rate-independent e Inerziali e Applicazioni\rq\rq (CUP\_E55F22000270001) is acknowledged by	F. R. and F. S. 
	
	G. S. and F. S. have been supported by the project STAR
	PLUS 2020 – Linea 1 (21-UNINA-EPIG-172) “New perspectives in the Variational modeling of Continuum Mechanics”. 
	
	The authors have been funded by the Italian Ministry of Education, University and Research through the Project “Variational methods for stationary and evolution problems with singularities and interfaces” (PRIN 2017).

	{\small
		
		\vspace{10pt} (Filippo Riva) Dipartimento di Matematica “Felice Casorati”, Università degli Studi di Pavia,\\ \textsc{Via Ferrata, 5, 27100, Pavia, Italy}
		\\ 
		\textit{e-mail address}: \textsf{filippo.riva@unipv.it}
		\par	
		
		\vspace{10pt} (Giovanni Scilla) Dipartimento di Matematica ed Applicazioni “R. Caccioppoli”, Università degli Studi di Napoli Federico II,\\
		\textsc{Via Cintia, Monte Sant'Angelo, 80126, Naples, Italy}
		\\
		\textit{e-mail address}: \textsf{giovanni.scilla@unina.it}
		\par
		
		\vspace{10pt} (Francesco Solombrino) Dipartimento di Matematica ed Applicazioni “R. Caccioppoli”, Università degli Studi di Napoli Federico II,\\
		\textsc{Via Cintia, Monte Sant'Angelo, 80126, Naples, Italy}
		\\
		\textit{e-mail address}: \textsf{francesco.solombrino@unina.it}
		\par
		
	}
	

\bigskip

	\begin{thebibliography}{99}
		
		{\frenchspacing			
			
			\bibitem{AFP} {\sc L.~Ambrosio, N.~Fusco and D.~Pallara}, {\em Functions of Bounded Variation and Free Discontinuity Problems}, Oxford Mathematical Monographs, Clarendon Press, Oxford (2000).

            \bibitem{AGS2008} {\sc L.~Ambrosio, N.~Gigli, and G.~Savar\'{e}}, {\em Gradient flows in metric spaces and in the space of probability measures}, Lectures in Mathematics ETH Z\"{u}rich, Birkh\"{a}user Verlag, Basel, second~ed., 2008.
			
			\bibitem{BachoViscoplast}
			{\sc A.~Bacho}, {\em Abstract nonlinear evolution inclusions of second order with applications in visco-elasto-plasticity}, J. Differential Equations, 363 (2023), pp.~126--169.
			
			\bibitem{Brez}
			{\sc H.~Brezis}, {\em Operateurs Maximaux Monotones et Semi-groupes de Contractions dans les Espaces de Hilbert},  North--Holland Publishing Company Amsterdam, (1973).
				
			\bibitem{DMSap}
			{\sc G.~Dal Maso and F.~Sapio}, {\em Quasistatic limit of a dynamic viscoelastic model with memory}, Milan J. Math., 89 (2021), pp.~485--522.	
				
			\bibitem{DMSca}
			{\sc G.~Dal Maso and R.~Scala}, {\em Quasistatic evolution in perfect plasticity as limit of dynamic processes},  J. Dynam. Differential Equations, 26 (2014), pp.~915--954.

			\bibitem{DL} {\sc R. Dautray and J. L. Lions}, \emph{Mathematical Analysis and Numerical Methods for Science and Technology}, Vol.5: Evolution Problems I. Berlin etc., Springer-Verlag 1992.

			\bibitem{DiBen} {\sc E. DiBenedetto}, {\em $C^{1+\alpha}$ local regularity of weak solutions of degenerate elliptic equations},
			Nonlinear Anal. Theory Methods Appl. 7 (1983), pp.~827--850.
 
			\bibitem{GiaqMart}
			{\sc M.~Giaquinta and L.~Martinazzi}, {\em An Introduction to the Regularity Theory for Elliptic Systems, Harmonic Maps and Minimal Graphs}, Publications of the Scuola Normale Superiore, (2012).
			
			\bibitem{GidRiv}
			{\sc P.~Gidoni and F.~Riva}, {\em A vanishing inertia analysis for finite dimensional rate-independent systems with non-autonomous dissipation and an application to soft crawlers}, Calc. Var. PDEs, 60 (2021), art.~191.
			
			\bibitem{LazNar}
			{\sc G.~Lazzaroni and L.~Nardini}, {\em On the quasistatic limit of dynamic evolutions for a peeling test in dimension one}, J. Nonlinear Sci., 28 (2018), pp.~269--304.
			
			\bibitem{LazRosThomToad}
			{\sc G.~Lazzaroni, R.~Rossi, M.~Thomas and R.~Toader}, {\em Rate-independent damage in thermo-viscoelastic materials with inertia}, J. Dyn. Diff. Equat., 30 (2018), pp.~1311--1364.
			
			\bibitem{MielkPetr}
			{\sc  A.~Mielke, A.~Petrov and J.~A.~C.~Martins}, {\em Convergence of solutions of kinetic variational inequalities in the rate-independent quasistatic limit}, J. Math. Anal. Appl., 348 (2008), pp.~1012--1020.
			
			\bibitem{MielkRosSav09}
			{\sc A.~Mielke, R.~Rossi and G.~Savar\'e}, {\em Modeling solutions with jumps for rate-independent systems on metric spaces}, Discrete and Continuous Dynamical Systems, 25 (2009), pp.~585--615.
						
			\bibitem{MielkRosSav12}
			{\sc A.~Mielke, R.~Rossi and G.~Savar\'e}, {\em BV solutions and viscosity approximations of rate-independent systems}, ESAIM Control Optim. Calc. Var., 18 (2012), pp.~36--80.
						
			\bibitem{MielkRosSav16}
			{\sc A.~Mielke, R.~Rossi and G.~Savar\'e}, {\em Balanced viscosity (BV) solutions to infinite-dimensional rate-independent systems}, J.
			Eur. Math. Soc., 18 (2016), pp.~2107--2165.
			
			\bibitem{MielkRosmulti}
			{\sc A.~Mielke, R.~Rossi}, {\em Balanced-Viscosity solutions to infinite-dimensional multi-rate systems}, Arch. Rational Mech. Anal., 247, 53 (2023).
			
			\bibitem{MielkRoubbook}
			{\sc A.~Mielke and T.~Roub\'i\v cek}, {\em Rate-independent Systems: Theory and Application}, Springer--Verlag New York, (2015).
			
			\bibitem{MielkTheil1}
			{\sc  A.~Mielke and F.~Theil}, {\em A mathematical model for rate-independent phase transformations with hysteresis}, Proceedings of
			the Workshop on Models of Continuum Mechanics in Analysis and Engineering, D.~Alber, R.~Balean and R.~Farwig Eds., Shaker-Verlag, Aachen (1999) pp.~117--129.
			
			\bibitem{MielkTheil}
			{\sc  A.~Mielke and F.~Theil}, {\em On rate-independent hysteresis models}, NoDEA Nonlinear Differ. Equ. Appl., 11 (2004), pp.~151--189.

			\bibitem{MNRT} {\sc B. S. Mordukhovich, N. M. Nam, R. B. Rector and T. Tran}, {\em Variational Geometric Approach to Generalized Differential and Conjugate Calculi in Convex Analysis}, Set-Valued Var. Anal, 25 (2017), pp.~731--755. 
			
			\bibitem{Mor} {\sc J. J. Moreau}, \emph{Jump functions of a real interval to a Banach space}, Annales de la faculté des sciences de Toulouse 5 série, S10 (1989), pp.~77--91.

			\bibitem{Rind}
			{\sc F.~Rindler, S.~Schwarzacher and J.~J.~L.~Vel\'azquez}, {\em Two-Speed solutions to non-convex rate-independent systems}, Arch. for Rat. Mech. and Anal., 239 (2021), pp.~1667--1731.
			
			\bibitem{Rivquas}
			{\sc F.~Riva}, {\em On the approximation of quasistatic evolutions for the debonding of a thin film via vanishing inertia and viscosity}, J. Nonlinear Sci., 30 (2020), pp.~903--951.
			
			\bibitem{RivScilSol}
			{\sc F.~Riva, G.~Scilla and F.~Solombrino}, {\em The notions of Inertial Balanced Viscosity and Inertial Virtual Viscosity solution for rate-independent systems}, Adv. Calc. Var. (2022), https://doi.org/10.1515/acv-2021-0073.
			
			\bibitem{Rocka}
			{\sc R.~T.~Rockafellar}, {\em Convex Analysis}, Princeton University Press, 1970.
			
			\bibitem{RosThom}
			{\sc R.~Rossi and M.~Thomas}, {\em Coupling rate-independent and rate-dependent processes: existence results}, SIAM J. Math. Anal., 49 (2012), pp.~1419--1494.		
			
			\bibitem{Sca}
			{\sc R.~Scala}, {\em Limit of viscous dynamic processes in delamination as the viscosity and inertia vanish},  ESAIM: Control Optim. Calc. Var., 23 (2017), pp.~593--625.
			
			\bibitem{ScilSol}
			{\sc G.~Scilla and F.~Solombrino}, {\em A variational approach to the quasistatic limit of viscous dynamic evolutions in finite dimension}, J. Differential Equations, 267 (2019), pp.~6216--6264.
			
			\bibitem{SciSol18}
			{\sc G.~Scilla and F.~Solombrino}, {\em Multiscale analysis of singularly perturbed finite dimensional gradient flows: the minimizing movement approach}, Nonlinearity, 31 (2018), pp.~5036--5074.
			
			\bibitem{Simon}
			{\sc J.~Simon}, {\em Compact sets in the space $L^p(0,T;B)$}, Ann. Mat. Pura Appl., 146 (1987), pp.~65--96.
			
		}
	\end{thebibliography}
\end{document}